\documentclass[10pt]{amsart}
\usepackage{latexsym, amsmath,amssymb}
\usepackage{centernot}

	\usepackage{xcolor}

\usepackage{graphicx}

\renewcommand{\epsilon}{\varepsilon}
\newcommand{\newsection}[1]
{\subsection{#1}\setcounter{theorem}{0} \setcounter{equation}{0}
\par\noindent}

\newtheorem{theorem}{Theorem}

\newtheorem{lemma}[theorem]{Lemma}
\newtheorem{corr}[theorem]{Corollary}

\newtheorem{proposition}[theorem]{Proposition}
\newtheorem{deff}[theorem]{Definition}

\newcommand{\bth}{\begin{theorem}}
\newcommand{\ble}{\begin{lemma}}
\newcommand{\bcor}{\begin{corr}}

\newcommand{\bdeff}{\begin{deff}}

\newcommand{\bprop}{\begin{proposition}}
\newcommand{\ele}{\end{lemma}}
\newcommand{\ecor}{\end{corr}}
\newcommand{\edeff}{\end{deff}}

\newcommand{\eprop}{\end{proposition}}

\newcommand{\Rn}{{\mathbb R}^n}

\newcommand{\la}{\lambda}

\newcommand{\e}{\varepsilon}
\renewcommand{\l}{\lambda}

\newcommand{\supp}{\mathrm{supp }}
\renewcommand{\Pi}{\varPi}
\renewcommand{\Re}{\rm{Re} \,}
\renewcommand{\Im}{\rm{Im} \,}

\newcommand{\Inj}{r_{\rm{Inj}}}

\newcommand{\one}{{\bf 1}}

\newcommand{\1}{{\bf 1}}

\renewcommand{\epsilon}{\varepsilon}

\newcommand{\dist}{{{\rm dist}}}

\newcommand{\tidle}{\tilde}

\newcommand{\R}{{\mathbb R}}

\newcommand{\A}{{\mathcal A}}

\newcommand{\N}{{\mathbb N}}

\newcommand{\diag}{\Upsilon^{\text{diag}}}

\newcommand{\far}{\Upsilon^{\text{far}}}

\newcommand{\farbar}{\overline\Upsilon^{\text{far}}}
\newcommand{\diagbar}{\overline\Upsilon^{\text{diag}}}

\begin{document}

\title[Lossless Strichartz and spectral projection estimates]
{Lossless Strichartz and spectral projection estimates on unbounded manifolds}
\thanks{The first author was supported in part by the Simons Foundation, the second author was  supported in part by the NSF (DMS-2348996)
and the third author was partially supported by the Simons Targeted Grant Award No. 896630.}
\author{Xiaoqi Huang}
\author{Christopher D. Sogge}
\author{Zhongkai Tao}
\author{Zhexing Zhang}
\address{XH: Department of Mathematics, Louisiana State University, Baton Rouge, LA 70803}
\address{CDS and ZZ: Department of Mathematics,  Johns Hopkins University,
Baltimore, MD 21218}
\address{ZT: Institut des Hautes \'Etudes Scientifiques, 91440 Bures-sur-Yvette, France}

\begin{abstract}
We prove new lossless Strichartz and spectral projection estimates on asymptotically hyperbolic surfaces,
and, in particular, on all convex cocompact hyperbolic surfaces.
In order to do this, we also obtain log-scale lossless Strichartz and spectral projection estimates on manifolds
of uniformly bounded geometry
with nonpositive and negative sectional curvatures, extending the recent works of the first two authors 
for compact manifolds.  
We
are able to use these along with 
  known $L^2$-local smoothing and new 
$L^2 \to L^q$ half-localized resolvent
estimates to obtain our lossless bounds.
\end{abstract}

\maketitle

\newsection{Introduction}

Two of the main goals of this paper are to prove lossless Strichartz and spectral projection estimates
on negatively curved asymptotically hyperbolic surfaces.
We also obtain frequency-dependent estimates on general manifolds of uniformly bounded geometry in all
dimensions all of whose sectional curvatures are negative or nonpositive.

Our first result is the following Strichartz estimates for solutions $u=e^{-it\Delta_g}u_0$ of the 
Schr\"odinger equation
\begin{equation}\label{ii.0}
i\partial_t u(x,t)=\Delta_g u(x,t), \quad u(x,0)=u_0(x).
\end{equation}

\begin{theorem}\label{thm1}
Let $(M,g)$ be an even asymptotically hyperbolic surface with negative curvature.  Then, for
$\tfrac1p+\tfrac1q=\tfrac12$, $p,q\ge 2$ and $(p,q)\ne(2,\infty)$, there exists $C_q=
C_q(M)$ such that
\begin{equation}\label{ii.1}
\bigl\| e^{-it\Delta_g}u_0\bigr\|_{L^p_tL^q_x(M\times [0,1])}\le C_q \|u_0\|_{L^2(M)}.
\end{equation}
Suppose further that $-\Delta_g$ does not have resonance at the bottom of the continuous spectrum, then for any $u_0\in L^2(M)$ orthogonal to the $L^2$ eigenfunctions of $\Delta_g$, we have
\begin{equation}\label{eq:str-global}
\bigl\| e^{-it\Delta_g}u_0\bigr\|_{L^p_tL^q_x(M\times \mathbb{R})}\le C_q \|u_0\|_{L^2(M)}.
\end{equation}
\end{theorem}

We shall review the hypotheses concerning $(M,g)$ in the next section.  We point out that any
convex cocompact hyperbolic surface is an even asymptotically hyperbolic surface of (constant)
negative curvature.  See the figure below, and see \cite{Borthwick} for more details.


\begin{figure}[ht]
\includegraphics[scale=0.1]{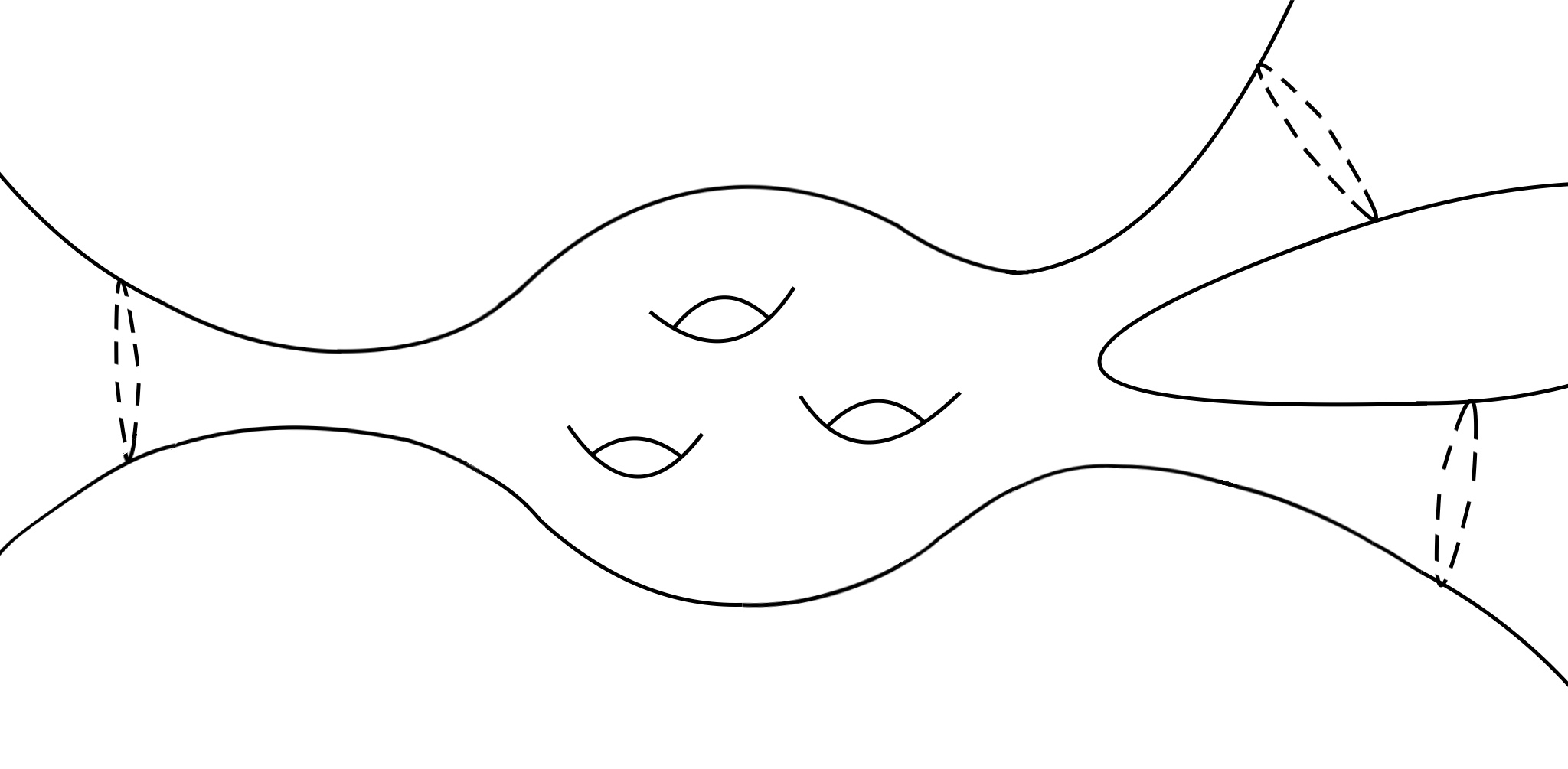}
\caption{Convex cocompact hyperbolic surfaces}
\centering
\end{figure}

The estimates \eqref{ii.1} are analogous to the standard Strichartz~\cite{Strichartz77} and Keel-Tao~\cite{KT} estimates
for ${\mathbb R}^2$, in which case, by scaling, estimates as above over $t\in [0,1]$ are
equivalent to ones over $t\in {\mathbb R}$.  We are only able to treat the two-dimensional case
of asymptotically hyperbolic manifolds here since some of the tools that we utilize, such as different types of $L^2$
local smoothing estimates for the Schr\"odinger propagators $e^{-it\Delta_g}$ seem to  only be available
in two-dimensions.  As we shall see, though, the lossless log-scale estimates that we also require hold
in all dimensions for manifolds of uniformly bounded geometry and nonpositive and negative sectional curvatures.

Besides the Euclidean estimates, there is a long history of Strichartz estimates for negatively curved asymptotically
hyperbolic manifolds.  On hyperbolic space ${\mathbb H}^n$, Anker and Pierfelice~\cite{AnPier}
and Ionescu and Staffilani~\cite{IoSt} independently proved the mixed-norm Strichartz
estimates via dispersive estimates that are unavailable for the manifolds that we are treating.
Subsequently, Bouclet~\cite{Boust} proved these results on
non-trapping asymptotically hyperbolic manifolds. Burq, Guillarmou and Hassell~\cite{BGH} then
were able to handle certain manifolds with trapped geodesics, including $n$-dimensional convex
cocompact hyperbolic manifolds whose limit set has Hausdorff dimension $<(n-1)/2$.  Among these
are hyperbolic cylinders ($n=2$) whose central geodesic $\gamma_0$ is periodic and hence
trapped. Burq, Guillarmou and Hassell~\cite{BGH} could obtain their Strichartz estimates for
these convex cocompact hyperbolic manifolds via a logarithmic time dispersive estimate.  Wang \cite{WangJ} proved Strichartz estimates for general (noncompact) convex cocompact hyperbolic surfaces with an $\epsilon$ loss of derivative. The results in Theorem~\ref{thm1}
above seem to be the first lossless Strichartz estimates with no pressure condition, which seems
to rule out the dispersive estimates that were used in these previous results. Note that for all asymptotically hyperbolic manifolds without trapping and with no resonance at the bottom of the spectrum, the analogous global-in-time Strichartz estimates were obtained by Chen~\cite{ChenSt}.

Burq, Guillarmou and Hassell also proved more general results involving
abstract hypotheses (cf. \cite[Theorem 3.3]{BGH}).  We are able to adapt their proof to obtain
our Theorem~\ref{thm1} using, as additional input, the local smoothing estimates following
from the local resolvent estimates of Bourgain and Dyatlov~\cite{bourgain2018spectral}
and the third author \cite{tao2024spectral}, as well as our new log-scale Strichartz estimates
for manifolds of nonpositive curvature and bounded geometry that we shall describe shortly.

As mentioned above, another of our main results concerns spectral projection operators
associated with the Laplace-Beltrami operator $\Delta_g$.  Before stating these, though,
let us recall the universal estimates for compact manifolds of the second author
\cite{sogge88} and the recent improvements by the first two authors \cite{HSp}.  If, for $q\in (2,\infty]$,
\begin{equation}\label{ii.2}
\mu(q)=
\begin{cases}
n(\tfrac12-\tfrac1q)-\frac12, \quad q\ge q_c=\tfrac{2(n+1)}{n-1},
\\
\tfrac{n-1}2(\tfrac12-\tfrac1q), \quad q\in (2,q_c],
\end{cases}
\end{equation}
and $(M,g)$ is an $n$-dimensional compact manifold then the main result in \cite{sogge88} 
says that for $\la\gg 1$ and $q>2$
\begin{equation}\label{ii.3}
\| \1_{[\la,\la+1]}(P) f\|_{L^q(M)}\le C\la^{\mu(q)} \|f\|_{L^2(M)}, \quad P=\sqrt{-\Delta_g},
\end{equation}
with $\1_I(P)$ being the spectral projection operator associated with the spectral window
$I\subset \R$.
It was shown by one of us in \cite{SFIO2} that the unit-band estimates \eqref{ii.3}
are always sharp.  On the other hand, recently, the first two authors were able to obtain
the following optimal bounds for compact manifolds all of whose sectional curvatures
are negative
\begin{equation}\label{ii.4}
\|\1_{[\la,\la+\delta]}(P) f\|_{L^q(M)}\le C_q \la^{\mu(q)}\delta^{1/2}
\|f\|_{L^2(M)}, \,  \, \delta\in [(\log\la)^{-1},1],
\la\gg 1 \, \, \text{and } \, q>2.
\end{equation}
Also, for later use, we note that, as was pointed out in \cite{AGL}, the proofs of the unit-band estimates
\eqref{ii.3} in \cite{sogge88} and \cite{SFIO2} also can be used to show that
\eqref{ii.3} is valid for any manifold of uniformly bounded geometry.

One of our main results (stated below) is that \eqref{ii.4} extends to all manifolds
of uniformly bounded geometry and curvature pinched below zero.  Using these
log-scale results and certain $L^2\to L^q$ localized resolvent estimates, 
we shall be able to adapt the proof of Theorem~\ref{thm1}
to obtain the following results optimal for much smaller spectral windows.

\begin{theorem}\label{thm2}
Let $(M,g)$ be an even asymptotically hyperbolic surface with negative curvature, and for $q>2$,
let $\mu(q)$ be as in \eqref{ii.2}.  Then for  $\la \gg 1$, we have
the uniform bounds
\begin{equation}\label{ii.5}
\| \1_{[\la,\la+\delta]}(P)f\|_{L^q(M)}\le C_{q} \,  \la^{\mu(q)}\delta^{1/2}\|f\|_{L^2(M)}, \, \, q \in (2,\infty],
\, \, \text{if } \, \delta \in (0,1].
\end{equation}
\end{theorem}

As we pointed out earlier, a special case of our results is when $(M,g)$ is a convex cocompact hyperbolic surface.
Spectral projection estimates on these were studied in Anker, Germain and L\'eger~\cite{AGL}, where somewhat
weaker estimates were obtained with a $\la^\e$, $\forall \, \e>0$, loss compared to our estimates. As was pointed
out in \cite{AGL}, using arguments of one of us \cite{SFIO2} mentioned before, one sees that the bounds in
\eqref{ii.5} are optimal.  

A direct corollary of the above theorem is 
\begin{corr}\label{cor}
    Let $(M,g)$ be an even asymptotically hyperbolic surface with negative curvature, and for $q>2$,
let $\mu(q)$ be as in \eqref{ii.2}.  Then for  $\la \gg 1$, if $dE_{P}(\la)$ denote the spectral measure for $ P=\sqrt{-\Delta_g}$,
\begin{equation}\label{coro}
\| dE_{P}(\la)f\|_{L^q(M)}\le 
     C_{q}   \la^{2\mu(q)}\|f\|_{L^{q'}(M)},\,\,\, q>2 \,\,\,\text{and}\,\,\,\tfrac1q+\tfrac1{q'}=1.
\end{equation}
\end{corr}

We note by $TT^*$ argument, \eqref{ii.5} is equivalent to
\begin{equation}\label{eq:proj-TT*}
    \|\1_{[\lambda,\lambda+\delta]}(P)f\|_{L^q(M)}\leq C_q \lambda^{2\mu(q)}\delta \|f\|_{L^{q'}(M)}, \, \, q \in (2,\infty],
\, \, \text{if } \, \delta \in (0,1].
\end{equation}
Moreover, by Stone's formula, and \cite[Theorem 5.33]{dyatlov2019mathematical}, $\1_{[0,\lambda)}(P)(x,y)$ is smooth in $\lambda$ for $\lambda> 1/2$, and 
\begin{equation*}
    dE_P(\lambda)=\lim_{\epsilon\to 0+}\frac{\1_{[\lambda-\epsilon,\lambda+\epsilon)}(P)}{2\epsilon} = \frac{1}{2\pi i}((P-(\lambda+i0))^{-1}-(P-(\lambda-i0))^{-1}).
\end{equation*}
Therefore, \eqref{coro} follows form \eqref{eq:proj-TT*} and Fatou's lemma.

For convex cocompact hyperbolic manifolds whose limit set has Hausdorff dimension $\delta < \frac{n-1}{2}$, \eqref{coro} is due to Han~\cite[Theorem 2]{han2021tomas}. For asymptotically hyperbolic manifolds without trapping, \eqref{coro} was proved by Chen and Hassell~\cite[Theorem 1.6]{ChHa}.

As is well known, the Euclidean variant of \eqref{coro} for $q\ge 6$ is equivalent to the Stein-Tomas~\cite{Tomas} restriction theorem.  Unlike in the Euclidean case, though, \eqref{coro} holds for 
all $q>2$.


Let us say a few more words about the proofs of Theorem~\ref{thm1} and \ref{thm2}.   First, it will be relatively straightforward to use
our generalizations in Theorem~\ref{bgst} below of the recent Strichartz estimates of two of us \cite{HSst} and known
local smoothing estimates for the Schr\"odinger propagator to obtain \eqref{ii.1} in Theorem~\ref{thm1}.  We are able to do this 
by using an argument from \cite{BGH}, which we shall recall in the next section.  Roughly speaking, near the compact
trapping region in $M$ we are able to obtain the needed dyadic results by gluing together uniform Strichartz estimates
on intervals of length $\la^{-1}\cdot \log\la$ for solutions of \eqref{ii.0} involving $\la$-frequency data using the known optimal
log-loss local smoothing estimates associated with this region.   This will allow us to show that the analog of the bounds in \eqref{ii.1}
are valid when the $L^p_tL^q_x$-norms are taken over $x$ in a relatively compact neighborhood of the trapping region.
The complement of this region then can easily be treated by the arguments in \cite{BGH}. 

The proof of the global estimate \eqref{eq:str-global} in Theorem~\ref{thm1} for the trapped region follows from a similar strategy. The main difference is that we must treat more carefully the nonlocal terms arising from commuting the frequency cutoff with the spatial cutoff, so that we are still able to use  $L^2$ local smoothing estimates for these nonlocal terms.
For the non-trapping region, 
we need to construct an auxiliary manifold $\tilde M$ that agrees with $M$ near infinity, such that it is a simply connected asymptotically hyperbolic surface of negative curvature with no $L^2$ eigenvalues or resonances. This allows us to apply global estimate of Chen \cite{ChenSt}.

We shall employ a similar strategy to prove our new spectral projection estimates.  As is standard,
in order to prove the bounds in Theorem~\ref{thm2} for $\1_{[\la,\la+\delta]}(P)$, it is equivalent to prove the
same bounds for the ``approximate spectral projection operators''
\begin{equation}\label{ii.7}
\rho\bigl((\la\delta)^{-1}(-\Delta_g-\la^2)\bigr)=(2\pi)^{-1} \int_{-\infty}^\infty
\la \delta \, \Hat \rho(\la\delta t) \, e^{-it\Delta_g} \, e^{-it\la^2} \, dt,
\end{equation}
with fixed $\rho\in {\mathcal S}(\R)$ satisfying $\rho(0)=1$ and its Fourier transform, $\Hat \rho$, supported in $[-1,1]$.
Using this simple formula (also used in \cite{AGL}), we can adapt the proof of Theorem~\ref{thm1}, which seems to be
a new approach.  Near the trapping region we introduce spatial cutoffs, as well as $t$-cutoffs localizing to intervals
of length $\la^{-1}\cdot \log\la$.  We are able to naturally estimate some of the terms arising from the time
cutoffs using Theorem~\ref{bgsp} below for manifolds of uniformly bounded geometry along with the aforementioned
local smoothing estimates for the Schr\"odinger propagator.  Unfortunately, it is not as straightforward to handle all of the
commutator terms that will arise in handling the complement of the trapping region.  For, unlike in the proof of 
Theorem~\ref{thm1}, we cannot appeal to the Christ-Kiselev lemma to handle the various ``Duhamel terms'' that
arise in estimating \eqref{ii.7}, which, of course, involves a weighted superposition of the Schr\"odinger propagator, as opposed
to the propagator itself occurring in the proof of the space-time estimates in Theorem~\ref{thm1}.  To deal with the
problematic Duhamel terms that arise, we are led to a simple integration by parts argument, and the resulting
boundary terms naturally give rise to half-localized $L^2\to L^q$ resolvent estimates paired with the available
$L^2$ local smoothing estimates.

As in the proof of lossless Strichartz estimate,
 we shall handle the myriad issues that arise by constructing a ``background manifold'' $\tilde M$ that agrees with
$M$ near infinity.  In the treatment of convex cocompact hyperbolic surfaces in \cite{AGL},  the background manifold was a finite
union of hyperbolic cylinders on which optimal spectral projection estimates could be proved and utilized.  In our case, $\tilde M$
is a simply connected asymptotically hyperbolic surface of negative curvature, which allows us to use the optimal spectral
projection estimates of Chen and Hassell~\cite{ChHa} for its Laplacian $\Delta_{\tilde g}$.  As we alluded to before, to glue these together
with estimates for the ``trapping'' compact region of $M$, we shall adapt the proof of the Strichartz estimates in Theorem~\ref{thm1}.
Since we cannot use the Christ-Kiselev lemma, certain half-localized $L^q$ resolvent estimates involving the Laplacian on the
background will arise.  These uniform bounds appear to be new and are of the form
$$\bigl\| \, (\Delta_{\tilde g}+(\lambda+i\delta)^2)^{-1}(\chi h)\, \bigr\|_{L^q(\tilde M)}\le C_q \, 
\la^{\mu(q)-1} \|h\|_{L^2(\tilde M)}, \, \,
\delta\in (0,1), \, \, q\in (2,\infty), \, \, \chi\in C^\infty_0(\tilde M),$$
with $\mu(q)$ as in \eqref{ii.2}.

We are oversimplifying a bit here how we shall use the optimal estimates for the ``background manifold''
$\tilde M$
in our proof of spectral projection estimates in Theorem~\ref{thm2} for $M$.   These are much more difficult to handle
compared to the Strichartz estimates due to the ``Duhamel terms'' that seem to inevitably arise because we cannot use
the Christ-Kiselev lemma.

Let us now describe the 
log-scale results on manifolds of uniformly  bounded geometry that we mentioned above.
These generalize recent joint work for compact manifolds of two of us
\cite{BHSsp}, \cite{SBLOg},  \cite{HSp}, \cite{HSqm} and \cite{HSst}.   Recall that $(M,g)$ is of uniformly
bounded geometry if the injectivity radius $\Inj(M)$ is {\em positive} and the Riemann curvature tensor
$R$ and all of its covariant derivatives are uniformly bounded.  (See, e.g. 
Eldering \cite[\S 2.1]{normhyp}.)

We then have the following two results for general manifolds of bounded geometry with
nonpositive sectional curvatures.

\begin{theorem}\label{bgst}  Suppose that $(M,g)$ is a complete $(n-1)$-dimensional manifold of uniformly bounded geometry
all of whose sectional curvatures are nonpositive.  Then if $u=e^{-it\Delta_g}f$ denotes the solution of
Schr\"odinger's equation
\begin{equation}\label{i.1}
i\partial_t u(t,x)=\Delta_g u(t,x), \, \, (t,x)\in \R \times M, \quad u|_{t=0}=f,
\end{equation}
we have for fixed $\beta\in C^\infty_0((1/2,2))$ and all $\la\gg 1$ the uniform dyadic estimates
\begin{equation}\label{i.2}
\bigl\| \, \beta(\sqrt{-\Delta_g}/\la) \, u\, \bigr\|_{L^p_tL^q_x(M\times [0,\la^{-1}\log\la])}\le C\|f\|_{L^2(M)}
\end{equation}
for all exponents $(p,q)$ satisfying the Keel-Tao condition
\begin{equation}\label{i.3}
(n-1)(1/2-1/q)=2/p, \, \, p\in [2,\infty) \, \, \text{if } \, n-1\ge 3 \, \, \text{and } \, 
\, p\in (2,\infty) \, \, \text{if } \, \, n-1=2.
\end{equation}
\end{theorem}

The arguments in Burq, G\'erard and Tzvetkov~\cite{BGT} yield the analog of \eqref{i.2}
with $[0,\la^{-1}\log\la]$ replaced by $[0,\la^{-1}]$
 for any complete manifold
of uniformly bounded geometry.  As in \cite{BHSsp} and \cite{HSst} we shall use the curvature
assumption in order to obtain the above logarithmic improvements.

As is well known, typically the standard Littlewood-Paley estimates which are valid for
$\Rn$ break down and must be modified for hyperbolic quotients; however, there are
variants that allow one to use dyadic estimates like \eqref{i.2}.  See Bouclet~\cite{Boulp}.
Using these, we obtain from Theorem~\ref{bgst} the following improvements of the compact
manifold estimates in \cite{BGT}.

\begin{corr}\label{bgsta}  Assume that $(M,g)$ is as in Theorem~\ref{bgst}.  Then for $(p,q)$ as in \eqref{i.3}
we have
\begin{equation}\label{i.3'}
\|(I+P)^{-1/p} \, (\log(2I+P))^{1/p} \, u\|_{L^p_tL^q_x(M\times [0,1])}\lesssim \|f\|_{L^2(M)}.
\end{equation}
\end{corr}

We shall postpone further discussion of the Littlewood-Paley estimates which can be used and the 
proof of this corollary in \S4.

We shall also be able to obtain similar improvements of the universal estimates \eqref{ii.3}:


\begin{theorem}\label{bgsp}  Suppose that $(M,g)$ is a complete $n$-dimensional manifold of uniformly bounded geometry all of whose
sectional curvatures are nonpositive. Then for $\la \gg 1$
\begin{equation}\label{i.6}
\| \, \one_{[\la,\la+(\log\la)^{-1}]}(P) \, \|_{L^2(M)\to L^q(M)}\lesssim
\begin{cases} \la^{\mu(q)}(\log\la)^{-1/2}, \quad \text{if } \, \, q>q_c
\\
\bigl(\la(\log\la)^{-1}\bigr)^{\mu(q)}, \quad \text{if } \, \, q\in (2,q_c].
\end{cases}
\end{equation}
Furthermore, if all of the sectional curvatures are pinched below $-\kappa_0^2$ with $\kappa_0>0$, then
\begin{equation}\label{i.7}
\| \, \one_{[\la,\la+(\log\la)^{-1}]}(P) \, \|_{L^2(M)\to L^q(M)}\le C_q \la^{\mu(q)}(\log\la)^{-1/2}, \, \, q\in (2,\infty].
\end{equation}
\end{theorem}

For $q=\infty$ and $q\in (q_c,\infty)$, the estimates in \eqref{i.6} for compact manifolds are due to
B\'erard~\cite{Berard} and Hassell and Tacy~\cite{HassellTacy}, respectively.  Also, the bounds in
\eqref{i.7} for hyperbolic space ${\mathbb H}^n$ were first proved by S. Huang and one of us \cite{SHSo} for
exponents $q\ge q_c$ and by Chen and Hassell~\cite{ChHa} for $q\in (2,q_c)$.  Additionally, it was shown in
\cite{HSp} and \cite{HSqm} that the bounds in \eqref{i.6} are sharp for flat compact manifolds,
and, as noted in \cite{HSp}, those in \eqref{i.7} can never be improved since they yield \eqref{ii.3}.  Also,
for $q\in (2,q_c)$, the standard Knapp example implies that the bounds in \eqref{i.7} do not hold for the spectral
projection operators associated with the Euclidean Laplacian in $\Rn$.

To prove the above results we shall need to make use of our assumption of (uniformly) bounded geometry.
To be able to adapt the Euclidean bilinear harmonic techniques of Tao, Vargas and Vega~\cite{TaoVargasVega} and 
Lee~\cite{LeeBilinear} that were used to prove analogous results for compact manifolds by two of us \cite{HSp},
\cite{HSst}, we shall make heavy use of the assumption regarding uniform bounds for derivatives of the
curvature tensor.  This will allow us to essentially reduce the local harmonic analysis step to individual coordinate
charts.  We shall also make heavy use of the assumption that $M$ has positive injectivity radius
in order to prove the global kernel estimates, which, along with the bilinear ones, will yield the above, just as
was done earlier for compact manifolds.

Let us now present a simple counterexample showing that the above estimates break down without the assumption that
$\Inj(M)>0$, even for hyperbolic quotients.  We shall use an argument in Appendix B of \cite{AGL} which provided counterexamples
for spectral projection bounds on hyperbolic surfaces with cusps.  (See also \cite[\S5.3]{Borthwick}.)

To be more specific let us consider the $n$-dimensional parabolic cylinder having a cusp at one end.  If we let
${\mathbb H}^n={\mathbb R}^{n-1}\times \R_+$ be the upperhalf space model for hyperbolic space, this is
$$M={\mathbb H}^n/\Gamma,$$
where $\Gamma$ is translation of ${\mathbb R}^{n-1}$ by elements of ${\mathbb Z}^{n-1}$.  So we identify 
$M$ with $x+ix_n\in {(-1/2\times 1/2]}^{n-1}\times \R_+$.

Recall that $\Delta_{{\mathbb H}^n}=(x_n)^2\sum_j \partial_j^2+(2-n)x_n\partial_n$.  
If we let $g(x)=x_n^{\frac{n-1}2-i\xi}$ a simple calculation shows that
$$-\Delta_{{\mathbb H}^n}g=\bigl( (\tfrac{n-1}2)^2+\xi^2\bigr) g.$$

Consider
$$\Phi_\la(x)=\frac1{\sqrt{2\pi}} \int_{-\infty}^\infty \phi(\delta^{-1}(\la-\xi))) \, x_n^{\frac{n-1}2-i\xi} \, d\xi,$$
where $\phi$ is supported in $[-1/10,1/10]$.  Note that $\Phi_\la$ is independent of
$(x_1,\dots,x_{n-1})$ and that
the $\sqrt{-\Delta_{{\mathbb H}^n}}$   spectrum of $\Phi_\la$ is in $[\la-\delta,\la+\delta]$
if $\la$ is large and $\delta\in (0,1]$.  Furthermore,
$$\Phi_\la(x)=\delta\,  x_n^{\frac{n-1}2-i\la}\Hat \phi(\delta \log x_n).$$
Using the change of coordinates $\omega =\log x_n$ we see that 
$$\|\Phi_\la\|_{L^2(M)}=\delta \Bigl(\int_0^\infty x_n^{n-1} \, |\Hat \phi(\delta \log x_n)|^2 \frac{dx_n}{x_n^n}\, \Bigr)^{1/2}
=\delta \Bigl(\int_{-\infty}^\infty |\Hat\phi (\delta \omega)|^2 \, d\omega \Bigr)^{1/2}.$$
On the other hand
\begin{align*}\|\Phi_\la\|_{L^q(M)}&=\delta \Bigl(\, \int_0^\infty x_n^{(n-1)\frac{q}2} \, |\Hat \phi(\delta \log x_n)|^q \, \frac{dx_n}{x_n^n} \, \Bigr)^{1/q}
\\
&=
\delta\Bigl( \int _{-\infty}^\infty e^{(\frac{q}2-1)(n-1)\omega} \, |\Hat \phi(\delta \omega)|^q \, d\omega\Bigr)^{1/q}.
\end{align*}
If we take $\phi(s)=a(s) \cdot \one_{[0,1]}(s)$ where $a\in C^\infty_0((-1/10,1/10))$ satisfies $a(0)=1$, then
$|\Hat \phi(\eta)|\approx |\eta|^{-1}$ for large $|\eta|$.  In this case, by the preceding two identities, $\Phi_\la\in L^2(M)$ but
$\Phi_\la\notin L^q(M)$ for any $q\in (2,\infty]$.  Based on this, it is clear that the spectral projection operators
$\one_{[\la,\la+\delta]}(\sqrt{-\Delta_{{\mathbb H}^n}})$ are unbounded between $L^2(M)$ and $L^q(M)$, and so
the estimates in Theorem~\ref{bgsp} cannot hold for this $M$, which has injectivity radius equal to zero.

One can similarly argue that the Strichartz estimates in Thereorm~\ref{bgst} also cannot hold for this hyperbolic quotient.
Indeed the proof of (1.15) in \cite{HSst} shows that, if the bounds in \eqref{i.2} were valid for a given pair $(p,q)$, then we would
have to have that for $\delta=\delta(\la)=(\log\la)^{-1}$ the spectral projection operators $\chi_{[\la,\la+\delta]}(\sqrt{-\Delta_{{\mathbb H}^n}})$
are bounded from $L^2$ to $L^q$ with norm $O((\la/\log\la)^{1/q})$, which, by the above discussion, is impossible.

This paper is organized as follows. In the next section we shall prove Theorems \ref{thm1} and \ref{thm2} using the above estimates for 
manifolds of uniformly bounded geometry and known local smoothing estimates for Schr\"odinger propagators.  
 In \S3 we shall prove our log-scale estimates for manifolds of uniformly bounded geometry
and appropriate curvature assumptions.  For the sake of completeness, in \S4, we shall also present the Littlewood-Paley estimates for manifolds of bounded
geometry that we are using.

Throughout this paper, we write $X\gg Y$ (or $X\ll Y$) to mean $X\ge C Y$ (or $X\le Y/C$) for some large constant $C>1$. Similarly, 
$X\gtrsim Y$ (or $X\lesssim Y$) denotes $X\ge C Y$ (or $X\le CY$) for some positive constant $C$.

{\bf Acknowledgments.}  The authors would like to express their gratitude to the referee for several comments and criticisms which improved the exposition.
We are also grateful to Xiaolong Han for raising a question that led us to strengthen an earlier version of our results to include
the global Strichartz estimates stated at the end of Theorem 1.1, as well as the uniform bounds for all $\delta\in (0,1]$
 in our spectral projection
estimates \eqref{ii.5}. 
 We also
would like to thank Maciej Zworski and Semyon Dyatlov for helpful discussions on the $L^2$ local smoothing estimate.

\newsection{Proofs of lossless estimates for asymptotically hyperbolic surfaces}

In this section, we shall see how we can apply Theorem~\ref{bgst} and ~\ref{bgsp} to prove lossless Strichartz and spectral projection estimates 
in Theorems \ref{thm1} and \ref{thm2}. The proof of Theorem~\ref{bgst} and ~\ref{bgsp} will be given in the next section.

Throughout this section, let us assume that
$(M,g)$ is a (even) asymptotically hyperbolic manifold. This means there exists a compactification $\overline{M}$, which is a smooth manifold with boundary $\partial M$, and the metric near the boundary takes the form
\begin{equation*}
    g=\frac{dx_1^2+g_1(x_1^2)}{x_1^2}, \quad x_1|_{\partial M}=0,\quad dx_1|_{\partial M}\neq 0
\end{equation*}
where $g_1(x_1^2)$ is a smooth family of metrics on $\partial M$. Examples include convex cocompact hyperbolic manifolds and their metric perturbation. A convex cocompact hyperbolic manifold is a hyperbolic manifold $M=\mathbb{H}^{n}/\Gamma$ such that the convex core is compact. Intuitively, it is a hyperbolic manifold with finitely many funnel ends and no cusps.

Let us also describe some dynamic properites of the geodesic flow $e^{tH_p}$ on asymptotically hyperbolic manifolds.
Let $S^*M=\{(x,\xi)\in T^*M: |\xi|_{g(x)}=1\}$ be the cosphere bundle of $M$ and $(x(t),\xi(t))=e^{tH_p}(x,\xi)$. The backward trapped set $\Gamma_+$ is defined as 
 \begin{equation*}
        \Gamma_+:=\{(x,\xi)\in S^*M:  x(t)\centernot\to \infty \,\,\,\text{as}\,\,\,t \to -\infty\}.
    \end{equation*}
    In other words, $(x,\xi)$ does not escape to $\infty$ along the backward geodesic flow. Similarly, the 
    forward trapped set $\Gamma_-$ is defined as
    \begin{equation*}
        \Gamma_-:=\{(x,\xi)\in S^*M: x(t)\centernot\to \infty \,\,\,\text{as}\,\,\,t \to +\infty \}.
    \end{equation*}
  The trapped set $K=\Gamma_+\cap \Gamma_-$ is the intersection of the backward trapped set and forward trapped set. In other words, $(x,\xi)\in K$ does not escape in either direction of the geodesic flow. For later use, let $\pi(K)$ be the projection of the trapped set $K$ onto $M$.

For all asymptotically hyperbolic manifolds such as convex cocompact hyperbolic manifolds, it is known that $\Gamma_\pm$ 
are both closed and the trapped set $K$ is compact, see 
e.g., Dyatlov--Zworski \cite[Chapter 6]{dyatlov2019mathematical} for more details. 
{Moreover, by the convexity of the geodesic flow at infinity \cite[Lemma 6.6]{dyatlov2019mathematical}, let $S\subset S^*M$ be a compact subset such that $S\cap \Gamma_-=\emptyset$ ($S\cap \Gamma_+=\emptyset$, respectively), then for any compact set $S'$, there exists a uniform constant $T=T(S,S')>0$ such that 
\begin{equation}\label{eq:convex}
    e^{tH_p}(x,\xi) \notin S',\quad (x,\xi)\in S
\end{equation}
for any $t\geq T$ ($t\leq -T$, respectively).
}

\noindent{\bf 2.1. Lossless Strichartz estimates.}

In this section, we shall give the proof of 
Theorem \ref{thm1}. We first establish \eqref{ii.1} and then modify the argument used to prove \eqref{eq:str-global}. We need three estimates 
to prove the Strichartz estimates for the exponents satisfying the Keel-Tao condition
 $\frac{2}{p}+\frac{n-1}{q}=\frac{n-1}{2}$, $p,q\geq 2$ and $(p,q)\neq (2,\infty)$.  Of course in the statement of Theorem~\ref{thm1}, $n-1=2$.  We are letting
$n$ denote the space-time dimension of $M\times \R$ to follow the convention that we are using in Theorem~\ref{bgst} (and used before in \cite{HSst}).

\noindent (a) Lossless Strichartz and local smoothing estimates in the nontrapping region. Let $\chi\in C_0^{\infty}(M)$ with $\chi=1$ near $\pi(K)$,
    \begin{equation}\label{eq:stri-nontrap}
    \|(1-\chi) e^{-it\Delta_g}u_0\|_{L^p_tL^q_x(M\times [0,1])}\leq C \|u_0\|_{L^2(M)}.
\end{equation}
One also needs a lossless local smoothing in the nontrapping region: Fix $\beta\in C_0^{\infty}((1/2,2))$,  for $\chi_1\in C_0^{\infty}(M)$ supported away from the trapped set $\pi(K)$ , we have 
\begin{equation}\label{eq:local-sm-nontrap}
   \|\chi_1 e^{-it\Delta_g}\beta(\sqrt{-\Delta_g}/\la)u_0 \|_{L^2_{t,x}(M\times [0,1])}\leq C\la^{-1/2}\|u_0\|_{L^2(M)}.
\end{equation}

\noindent (b) Local smoothing with logarithmic loss. Let $\chi\in C_0^{\infty}(M)$,
\begin{equation}\label{eq:local-sm-logloss}
    \|\chi e^{-it\Delta_g}\beta(\sqrt{-\Delta_g}/\la)u_0 \|_{L^2_{t,x}(M\times [0,1])}\leq C\la^{-\frac12}(\log \la)^{1/2}\|u_0\|_{L^2(M)}.
\end{equation}

\noindent(c) Lossless Strichartz with log-scale gains compared to the universal estimates in \cite{BGT}
 \begin{equation}\label{eq:stri-log-trap}
    \| e^{-it\Delta_g}\beta(\sqrt{-\Delta_g}/\la)u_0\|_{L^p_tL^q_x(M\times [0,\la^{-1}\log\la])}\leq C \|u_0\|_{L^2(M)}.
\end{equation}

We recall a lemma from \cite{BGH}.
\begin{lemma}\label{bghlemma}
    The estimates \eqref{eq:stri-nontrap}--\eqref{eq:stri-log-trap} imply the lossless Strichartz estimate
\begin{equation}\label{eq:stri-lossless}
    \|e^{-it\Delta_g}u_0\|_{L^p_tL^q_x(M\times [0,1])}\leq C \|u_0\|_{L^2(M)}.
\end{equation}
\end{lemma} 

To prove Lemma~\ref{bghlemma}, we use the following lemma whose proof we postpone to the end of this section.
\begin{lemma}\label{co}
    Let $\chi\in C_0^\infty(M)$ satisfy $\chi=1$ in a neighborhood of $\pi(K)$. Let $\beta\in C_0^{\infty}((1/2,2))$ be a Littlewood-Paley bump function satisfying $\sum_{k=-\infty}^\infty \beta(s/2^k)=1$,
$s>0$. Suppose $\tilde \chi\in C_0^\infty(M)$ supported away from the trapped set $\pi(K)$ and $\tilde\chi=1$ in a neighborhood of $\supp \nabla\chi$,  and let $\tilde \beta \in C^\infty_0((1/4,4))$ with $\tilde \beta=1$ in $(1/3,3)$, we have for $\la\gg 1$,
\begin{equation}\label{commuteerror}
    \chi \cdot \beta(\sqrt{-\Delta_g}/\la)f = \tilde\beta(\sqrt{-\Delta_g}/\la) \chi \cdot \beta(\sqrt{-\Delta_g}/\la)  f
+R_1f,
\end{equation}
as well as  
\begin{equation}\label{commuteerror1}
    [\Delta_g,\chi] \cdot \beta(\sqrt{-\Delta_g}/\la)f= \tilde\beta(\sqrt{-\Delta_g}/\la)  [\Delta_g,\chi] \cdot \beta(\sqrt{-\Delta_g}/\la)\tilde\chi \,\tilde\beta(\sqrt{-\Delta_g}/\la)  f
+R_2f,
\end{equation}
where for $i=1,2$,
\begin{equation}\label{r}
    \|(-\Delta_g)^{\alpha}R_if\|_{L^2(M)}\le C_{\alpha,N}\la^{-N}\|f\|_{L^2(M)}  \,\,\forall \alpha>0, N=1,2,3\cdots.
\end{equation}
Moreover,
\begin{equation}\label{commuteerror2}
\| [\Delta_g,\chi] \cdot \beta(\sqrt{-\Delta_g}/\la)f\|_{L^2}\lesssim \la\|f\|_{L^2}.
\end{equation}
\end{lemma}

\begin{proof}[Proof of Lemma~\ref{bghlemma}]
By the Littlewood Paley estimate in Lemma~\ref{littlewood} and the remark below it, we may assume $u_0=\beta(\sqrt{-\Delta_g}/\la)u_0$ with
$\beta$ as above and $\la\gg 1$.
By \eqref{eq:stri-nontrap}, it suffices to show for any $\chi\in C_0^{\infty}(M)$ with $\chi=1$ in a neighborhood of $\pi(K)$, we have
    \begin{equation*}
        \|\chi e^{-it\Delta_g}u_0\|_{L^p_tL^q_x(M\times [0,1])}\leq C \|u_0\|_{L^2(M)}.
    \end{equation*}

 Let $\alpha\in C_0^\infty((-1,1))$ satisfying $\sum \alpha(t-j)=1$, $t\in {\mathbb R}$. For $j\in \mathbb Z$  and $u(t)=e^{-it\Delta_g}u_0$, let us define $u_j=\alpha(t \,\la /\log \la-j)\chi u$. We have
    \begin{equation*}
        (i\partial_t-\Delta_g)u_j=v_j+w_j
    \end{equation*}
    where
    \begin{equation}\label{eq:def-ujvj}
        v_j=i\frac{\la}{\log \la}\alpha'(t\,\la /\log \la-j)\chi u,\quad w_j=-\alpha(t\,\la /\log \la-j)[\Delta_g,\chi]u.
    \end{equation}

Let $I_j=[(j-1)\la^{-1}\log\la, (j+1)\la^{-1}\log\la]\cap [0,1]$ and define
     \begin{equation}\label{oj}
  \tilde w_j=-\alpha(t\,\la /\log \la-j)[\Delta_g,\chi]\beta(\sqrt{-\Delta_g}/\la)\tilde\chi \,\tilde\beta(\sqrt{-\Delta_g}/\la)e^{-it\Delta_g}u_0,
    \end{equation}
  where  $\tilde \chi\in C_0^\infty(M)$ supported away from the trapped set $\pi(K)$ and $\tilde\chi=1$ in a neighborhood of $\supp \nabla\chi$.
 Let $\chi_-\in C_0^{\infty}$ satisfy $\chi_-=1$ on $\supp \chi$. Then 
    \begin{equation*}
    v_j=\chi_-v_j,\quad \tilde w_j=\tilde \chi\tilde w_j.
    \end{equation*}
By \eqref{commuteerror} and \eqref{commuteerror1}, we have 
\begin{equation}
\begin{aligned}
      u_j=&\int_{(j-1)\la^{-1}\log \la}^{t}e^{-i(t-s)\Delta_g}\tilde\beta(\sqrt{-\Delta_g}/\la)\chi_-v_j(s) ds\\
            &+  \int_{(j-1)\la^{-1}\log\la}^{t}e^{-i(t-s)\Delta_g} i\frac{\la}{\log \la}\alpha'(s\,\la /\log \la-j) R_1(e^{is\Delta_g}u_0)ds \\
      &+ \int_{(j-1)\la^{-1}\log\la}^{t}e^{-i(t-s)\Delta_g}\tilde\beta(\sqrt{-\Delta_g}/\la)\tilde \chi\tilde w_j(s)ds \\
      &+  \int_{(j-1)\la^{-1}\log\la}^{t}e^{-i(t-s)\Delta_g}\alpha(s\,\la /\log \la-j) R_2(e^{is\Delta_g}u_0)ds\\
      =& u_j^{(1)}+E_1+ u_j^{(2)}+E_2
\end{aligned}
\end{equation}
By Minkowski’s integral inequality, the Sobolev estimate, and \eqref{r},
\begin{equation}
\begin{aligned}
      \|E_1\|_{L_t^p L_x^q(M\times I_j)}&\lesssim \frac{\la}{\log \la}\left\|\int_{I_j} \|e^{-i(t-s)\Delta_g}(-\Delta_g)^{(\frac12-\frac1q)}R_1(e^{is\Delta_g}u_0)\|_{L^2(M)} ds\right\|_{L^p_t(I_j)} \\
      &\lesssim_N\la^{-N}\|u_0\|_{L^2(M)}.  
\end{aligned}
\end{equation}
A similar estimate holds for $E_2$. Thus, since the number of non-empty intervals $I_j$ is $O(\la(\log\la)^{-1})$, both error terms result in $O(\la^{-N})$ contributions to \eqref{eq:stri-lossless}.
    
Hence, to bound $u_j$, it suffices to estimate 
$u_j^{(1)}$
and $u_j^{(2)}$ 
on the interval $I_j$,  which contains the support of $u_j$.
 Let $\tilde{u}_j^{(1)}, \tilde{u}_j^{(2)}$ be the analog of 
${u}_j^{(1)}, {u}_j^{(2)}$ with the upper bound of the integrals replaced by $(j+1)\la^{-1}\log\la$. 
    Since $\tilde \chi \tilde w_j$ is supported in the nontrapped region, by \eqref{eq:local-sm-nontrap} and \eqref{eq:stri-log-trap} we have
    \begin{equation*}
        \|\tilde{u}_j^{(2)}\|_{L_t^p L_x^q(M\times I_j)}\lesssim \left\|\int_{(j-1)\la^{-1}\log\la}^{(j+1)\la^{-1}\log\la}e^{is\Delta_g}\tilde\beta(\sqrt{-\Delta_g}/\la)\chi_+\tilde w_j(s)ds\right\|_{L^2}\lesssim \la^{-1/2}\|\tilde w_j\|_{L^2_{t,x}}.
    \end{equation*}

    The same estimate holds for $u_j^{(2)}$ by the Christ--Kiselev lemma. On the other hand, on the trapped region, by \eqref{eq:stri-log-trap} and local smoothing estimate \eqref{eq:local-sm-logloss}, we have
    \begin{multline*}
        \|\tilde u_j^{(1)}\|_{L_t^pL_x^q(M\times I_j)}\lesssim \left\|\int_{(j-1)\la^{-1}\log\la}^{(j+1)\la^{-1}\log\la}e^{is \Delta_g}\tilde\beta(\sqrt{-\Delta_g}/\la)\chi_-v_j(s) ds\right\|_{L^2}\\\lesssim \la^{-\frac12}(\log \la)^{1/2}\|v_j\|_{L^2_{t,x}}.
    \end{multline*}
    The same estimate holds for $u_j^{(1)}$ by the Christ--Kiselev Lemma.

    Note that by using \eqref{commuteerror2},
    $$\|\tilde w_j\|_{L^2_{t,x}(M\times I_j)} \lesssim \la\|\alpha(t\,\la /\log \la-j)\tilde\chi \,\tilde\beta(\sqrt{-\Delta_g}/\la)e^{-it\Delta_g}u_0\|_{L^2_{t,x}}.
    $$
    Thus by the local smoothing estimates \eqref{eq:local-sm-nontrap} and \eqref{eq:local-sm-logloss}, we have
    \begin{multline*}
    \|\chi u\|_{L_t^pL_x^q(M\times I)}^{2}\le  \sum_j \|u_j\|_{L_t^pL_x^q(M\times I_j)}^2\\\lesssim \sum_j \left(\la^{-1}\|w_j\|_{L^2_{t,x}(M\times I_j)}^2+\la^{-1}\log \la\|v_j\|_{L^2_{t,x}(M\times I_j)}^2\right)+O(\la^{-N}\|u_0\|_{L^2}^2)\lesssim \|u_0\|_{L^2}^2.
    \end{multline*}
    This completes the proof of \eqref{eq:stri-lossless}.
\end{proof}

Turning to the proof of \eqref{ii.1} in Theorem~\ref{thm1}, let us first recall that 
\eqref{eq:stri-nontrap} is known for all asymptotically hyperbolic manifolds, see Bouclet \cite[Theorem 1.2]{Boust}. 
By \cite[Theorem 7.2]{dyatlov2019mathematical}, the assumptions \eqref{eq:local-sm-nontrap} and \eqref{eq:local-sm-logloss} follow from the following resolvent estimates. For $\chi\in C_0^{\infty}(M)$ supported away from the trapped set $\pi(K)$
\begin{equation}\label{eq:res-nontrap}
    \|\chi (-\Delta_g-(\lambda+i0)^2)^{-1}\chi\|_{L^2\to L^2}\leq C\lambda^{-1}.
\end{equation} 
Additionally, if $\chi\in C_0^{\infty}(M)$ with $\chi=1$ on $\pi(K)$,
\begin{equation}\label{b}
    \|\chi(-\Delta_g-(\lambda+i0)^2)^{-1}\chi\|_{L^2\to L^2}\le C\la^{-1}\log\la.
\end{equation}

\eqref{eq:res-nontrap} is known to hold for asymptotically hyperbolic manifolds if we assume $\chi$ is supported sufficiently far away from $\pi(K)$ by Cardoso--Vodev \cite{Vodev}, following the method of Carleman estimate in Burq \cite{Burq98}. Under the stronger condition that 
\begin{equation}\label{eq:res-trap-lambdaN}
    \|\chi (-\Delta_g-(\lambda+i0)^2)^{-1}\chi\|_{L^2\to L^2}\leq C\lambda^{N_0},\quad \chi\in C_0^{\infty}(M),
\end{equation}
the resolvent estimate \eqref{eq:res-nontrap} and the local smoothing estimate \eqref{eq:local-sm-nontrap} follows from standard propagation estimates, see Datchev--Vasy \cite{DatVa12} for the resolvent estimate in this case.

For convex cocompact hyperbolic surfaces, \eqref{b} follows from the result of Bourgain-Dyatlov \cite[Theorem 2]{bourgain2018spectral} and the interpolation argument of Burq \cite[Lemma 4.7]{burq2004smoothing}. This was generalized to even asymptotically hyperbolic surfaces with negative curvature by the third author in \cite{tao2024spectral}, following a result of Vacossin \cite{vacossin}. In higher dimensions, \eqref{b} also hold under certain conditional trapping conditions, such as the pressure condition and normally hyperbolic trapping, see Nonnenmacher and Zworski \cite{nozw-pressure, nozw-normhyp}. Finally, \eqref{eq:stri-log-trap} follows from Theorem ~\ref{bgst} which holds for all complete manifolds with bounded geometry and nonpositive sectional curvature.

Hence, \eqref{eq:stri-nontrap}–\eqref{eq:stri-log-trap} hold for all even asymptotically hyperbolic surfaces with  negative curvature, which completes the proof of Theorem~\ref{thm1}.  Addtionally, for $\la\gg 1$,  \eqref{eq:local-sm-nontrap} and \eqref{eq:local-sm-logloss} remain valid when $[0,1]$ is replaced by $\R$ on all asymptotically hyperbolic surfaces with  negative curvature. We will use this fact later in the proof of \eqref{eq:str-global} and Theorem~\ref{thm2}.

The lossless Strichartz estimate can be used to prove the following local well-posedness of the cubic nonlinear Schr\"odinger equation in the critical regularity.
\begin{proposition}
Let $(M,g)$ be as in Theorem~\ref{thm1}, or more generally, assume that $n-1=2$ and that \eqref{eq:stri-lossless} holds. Consider the Schr\"odinger equation
    \begin{equation}\label{eq:NLS}
        i\partial_t u -\Delta_g u =F(u),\quad u(0,\cdot)=u_0(x)\in L^2(M)
    \end{equation}
    where $F(u)$ is {a homogeneous cubic polynomial of $u$ and $\bar{u}$}. Then there exists  $T>0$ such that \eqref{eq:NLS} has a unique solution 
    \begin{equation*}
        u(t,x)\in C([-T,T];L^2(M)) \cap L^{3}([-T,T] ; L^6(M) ).
    \end{equation*}
\end{proposition}

\begin{proof}
    Consider the map
    \begin{equation*}
        G(u)(t,x)=e^{-it\Delta_g}u_0-i\int_0^t e^{-i(t-t')\Delta_g}F(u)(t',x) dt'.
    \end{equation*}
    Let $0<T\leq 1$. Define the norm 
    \begin{equation*}
        \|u\|_{Y_T}:= \sup_{t\in[-T,T]}\|u(t,\cdot)\|_{L^2(M)}+\|u\|_{L^{3}([-T,T]; L^6(M))}.
    \end{equation*}
    Then
    \begin{equation*}
        \|G(u)\|_{Y_T}\leq C\|u_0\|_{L^2}+ \int_{-T}^{T}\|F(u)\|_{L^2(M)}dt\leq C\|u_0\|_{L^2}+ C\|u\|_{L^3([-T,T];L^6(M))}^3
    \end{equation*}
    and
    \begin{equation}\label{eq:cubicNLS-diff}
    \begin{split}
        \|G(u)-G(v)\|_{Y_T}&\leq \int_{-T}^{T}\|F(u)-F(v)\|_{L^2(M)}dt\\
        &\leq C(\|u\|_{L^3([-T,T];L^6(M))}+\|v\|_{L^3([-T,T];L^6(M))})^{2}\|u-v\|_{Y_T}.
    \end{split}
    \end{equation}
    Choose $T>0$ such that $\|e^{-it\Delta_g}u_0\|_{L^3([-T,T];L^6(M))}$ is sufficiently small. Then $G$ is a contraction map on
    \begin{equation}\label{eq:contraction-G}
        \{u\in Y_T: \|u\|_{L^3([-T,T];L^6(M))}\leq \epsilon\}.
    \end{equation}
    Note that \eqref{eq:cubicNLS-diff} implies that $G$ maps \eqref{eq:contraction-G} to itself.
    This gives a unique fixed point of $G$, which is a solution to \eqref{eq:NLS} in the space $Y_T$.

    The uniqueness of the solution follows from \eqref{eq:cubicNLS-diff}.
    \end{proof}

\begin{proof}[Proof of Lemma~\ref{co}]

We first prove \eqref{commuteerror2}. It suffices to show for $\la_1\ge 8\la$
\begin{equation}\label{c1}
\| \beta(\sqrt{-\Delta_g}/\la_1)[\Delta_g,\chi] \cdot \beta(\sqrt{-\Delta_g}/\la)f\|_{L^2}\lesssim_N \la_1^{-N}\|f\|_{L^2}, 
\end{equation}
as well as 
\begin{equation}\label{c2}
\| \beta_0(\sqrt{-\Delta_g}/\la)[\Delta_g,\chi] \cdot \beta(\sqrt{-\Delta_g}/\la)f\|_{L^2}\lesssim \la\|f\|_{L^2},
\end{equation}
for any $\beta_0\in C_0^\infty(-10,10)$. 

To prove \eqref{c1}, 
 we can extended $\beta$ to be an even function by adding the negative part if necessary. Let $P=\sqrt{-\Delta_g}$, and let $\delta_0$ be fixed and smaller than the injectivity radius of $M$. 
Fix an even function $\rho\in C_0^\infty$   satisfying $\rho(t)= 1$, $|t|\le \delta_0/4$ and $\rho(t)=0 $, $|t|\ge \delta_0/2$ such that
\begin{equation}\label{c3}
\begin{aligned}
     \beta( P/\la)=&(2\pi)^{-1}\int_\R\la\hat \beta(\la t)\cos t P dt \\
     =&(2\pi)^{-1}\int\rho(t)\la\hat \beta(\la t)\cos t P dt+ (2\pi)^{-1}\int (1-\rho(t)) \la\hat \beta(\la t)\cos t P dt\\
     =& B_\la(P)+C_\la(P).
\end{aligned}
\end{equation}
It is not hard to check that the symbol of the operator $C_\la$ is $O_N((1+|\tau|+|\la|)^{-N})$, which implies that 
  \begin{equation}\label{cestimatee}
      \|\Delta_g^{N_0} C_\la(P)\|_{L^2\to L^2}\lesssim_{N_0,N}\la^{-N},
  \end{equation}
 for any fixed $N_0, N$.

   To handle the term in \eqref{c1} involving $C_{\la}(P)$, note that 
\begin{equation*}
    \begin{split}
         \|\beta(P/\la_1) [\Delta_g, \chi] C_\la(P) f\|_{L^2}
         &\lesssim \la_1^{-2N}\|\beta(P/\la_1)\Delta_g^N [\Delta_g, \chi] C_\la(P) f\|_{L^2}\\
         &\lesssim \la_1^{-2N}\|B_{\la_1}(P)\Delta_g^N [\Delta_g, \chi] C_\la(P) f\|_{L^2}\\
         &\quad+\la_1^{-2N}\|C_{\la_1}(P)\Delta_g^N [\Delta_g, \chi] C_\la(P) f\|_{L^2}.
    \end{split}
\end{equation*}
By \eqref{cestimatee}, the second term on the right side is bounded by
\begin{equation}\label{c4}
\begin{split}
&\la_1^{-2N}\|C_{\la_1}(P)\Delta_g^N [\Delta_g, \chi] C_\la(P) f\|_{L^2} 
\\ &\leq
    \la_1^{-2N}\|C_{\la_1}(P)\Delta_g^{N+1}  \chi C_\la(P) f\|_{L^2} 
    +  \la_1^{-2N}\|C_{\la_1}(P)\Delta_g^N \chi\,\Delta_g C_\la(P) f\|_{L^2}\\
    &\lesssim_N \la_1^{-N}\| f\|_{L^2}.
\end{split}
\end{equation}
For the first term, we have
\begin{equation}\label{c5}
    \begin{split}
    &\la_1^{-2N}\|B_{\la_1}(P)\Delta_g^N [\Delta_g, \chi] C_\la(P) f\|_{L^2} \\
    &\lesssim \la_1^{-2N}\sum\| B_{\la_1}(P)[\Delta_g,[\Delta_g,[\cdots,\chi]\cdots]\Delta_g^{k} C_{\lambda} (P)f\|_{L^2} \\
    &\lesssim \la_1^{-N+1}\sum_{k\le N}\|\Delta_g^k C_{\lambda} (P)f\|_{L^2}\\
    &\lesssim \la_1^{-N+1}\|f\|_{L^2}.
    \end{split}
\end{equation}
Here we used the fact that 
\begin{equation}\label{c6}
    \| B_{\la_1}(P)[\Delta_g,[\Delta_g,[\cdots,\chi]\cdots]\|_{L^2\to L^2}\lesssim \la_1^{N+1}
\end{equation}
if the number of brackets is bounded by $N+1$. By using the parametrix construction for $\cos tP$, 
 one can argue as in the compact manifold case to show that the local operator $B_{\la_1}$ is a 0-order pseudodifferential operator with principal symbol $\beta(p(x,\xi)/\la_1)$,
with $p(x,\xi)$ here being the principal symbol of $P$. Thus it is not hard to show  $B_{\la_1}(P)[\Delta_g,[\Delta_g,[\cdots,\chi]\cdots]$ is a pseudodifferential operator of order $N+1$ with symbol supported in the region $p(x,\xi)\approx \la_1$, which implies \eqref{c6}.

   To handle the term in \eqref{c1} involving $B_{\la}(P)$, note that
\begin{equation*}
\begin{split}
     \| \beta(P/\lambda_1)[\Delta_g, \chi]  B_{\la}(P)) f\|_{L^2}
\leq 
         \| B_{\la_1}(P)[\Delta_g, \chi]  B_{\la}(P)) f\|_{L^2}  +    \|C_{\la_1}(P)[\Delta_g, \chi]  B_{\la}(P)) f\|_{L^2}. 
\end{split}
\end{equation*}
The second term can be estimated as in \eqref{c4}.
Hence it suffices to show
\begin{equation}\label{c7}
    \|B_{\la_1}(P)[\Delta_g, \chi] B_{\la}(P)\|_{L^2\to L^2}\lesssim \la_1^{-N}.
\end{equation}
This follows from integration by parts after writing out the composition of the corresponding pseudodifferential operators explicitly.

To prove \eqref{c2},  write $  \beta_0( P/\la)
     = B^0_\la(P)+C^0_\la(P)$ as in \eqref{c3}. Repeating the argument above, it suffices to show that
\begin{equation}\label{c8}
    \|B^0_{\la}(P)[\Delta_g, \chi] B_{\la}(P)\|_{L^2\to L^2}\lesssim \la.
\end{equation}
This follows from the fact that the operator above is a pseudodifferential operator of order 1 with symbol supported in the region $p(x,\xi)\lesssim\la$.

To prove \eqref{commuteerror1}, it suffices to show 
\begin{equation}\label{c9}
    [\Delta_g,\chi] \cdot \beta(P/\la)f= \tilde\beta(P/\la)  [\Delta_g,\chi] \cdot \beta(P/\la)f
+R_1f,
\end{equation}
as well as 
\begin{equation}\label{c10}
   \tilde\beta(P/\la)  [\Delta_g,\chi] \cdot \beta(P/\la)f
   = \tilde\beta(\sqrt{-\Delta_g}/\la)  [\Delta_g,\chi] \cdot \beta(P/\la)\tilde\chi \,\tilde\beta(P/\la) f
+R_2f,
\end{equation}
where the operators $R_1$ and $R_2$ satisfy \eqref{r}.

As above, to prove \eqref{c9}, it suffices to show for $\la_1\ge 8\la$
\begin{equation}\label{c1a}
\| \beta(P/\la_1)(1-\tilde\beta)(P/\la)[\Delta_g,\chi] \cdot \beta(P/\la)f\|_{L^2}\lesssim_N \la_1^{-N}\|f\|_{L^2}, 
\end{equation}
as well as 
\begin{equation}\label{c2a}
\| \beta_0(P/\la)(1-\tilde\beta)(P/\la)[\Delta_g,\chi] \cdot \beta(P/\la)f\|_{L^2}\lesssim_N \la^{-N}\|f\|_{L^2},
\end{equation}
for any $\beta_0\in C_0^\infty(-10,10)$. 
Both \eqref{c1a} and \eqref{c2a} follow from the same argument used in the proof of \eqref{c1}.

To prove \eqref{c10}, note that by using \eqref{commuteerror2} and duality, it is not hard to show 
$$\|(-\Delta_g)^\alpha\tilde\beta(P/\la)  [\Delta_g,\chi]\|_{L^2\to L^2}\lesssim \la^{2\alpha}\|\tilde\beta(P/\la)  [\Delta_g,\chi]\|_{L^2\to L^2} \lesssim \la^{1+2\alpha}$$
Hence \eqref{c10} is a consequence of 
\begin{equation}\label{c12}
\chi_0 \cdot \beta(P/\la)f
   =\chi_0 \cdot \beta(P/\la)\tilde\chi \,\tilde\beta(P/\la) f
+Rf,
\end{equation}
if $\chi_0\in C_0^\infty$ is chosen so that $\chi_0=1$ on $\supp\, \nabla\chi$, and $\tilde \chi=1$ in a $\delta_0$-neighborhood of the support of $\chi_0$, with $\delta_0$ defined as in \eqref{c3}. Here we may assume $\chi_0$ is supported in a sufficiently small neighborhood of $\supp\, \nabla\chi$, and 
 $\delta_0$ is chosen small enough so that $\tilde \chi$ is supported away from the trapped set $\pi(K)$.

The proof of \eqref{c12}, as well as the remaining estimate \eqref{commuteerror} in Lemma~\ref{co}, follows from a simpler but analogous argument as above, together with the fact that the operator $B_\la(P)$ in \eqref{c3} is local due to finite propagation speed of the wave propagator $\cos(tP)$. We omit the details for simplicity.
\end{proof}

\bigskip

\noindent{\bf 2.1.1 Global Strichartz estimate.}

Now we prove the global Strichartz estimate \eqref{eq:str-global}. This relies on the construction of a ``background'' manifold $(\tilde M, \tilde g)$ which agrees with $M$ asymptotically and satisfies favorable no loss Strichartz estimates. 
Specifically, we shall assume that
$M=M_{tr}\cup M_\infty$ where $M_{tr} \subset M$ is compact and contains a neighborhood of the trapped set $\pi(K)$ defined at the beginning of this section. We shall construct $\tilde M$ such that  
the metric $\tilde g$ for $\tilde M$ agrees with the metric $g$ on $M_\infty$.

Recall that an even asymptotically hyperbolic surface has the topology of a finite disjoint union of $(0,\epsilon)\times S^1$ near infinity. We fill in a disk inside each $(0,\epsilon)\times S^1$ so that $\tilde{M}$ is a finite disjoint union of topological disks. Without loss of generality, we will discuss estimates on each connected component. On an even asymptotically hyperbolic surface, the metric near the boundary is given by
\begin{equation*}
\frac{dx_1^2}{4x_1^2}+\frac{h_1(x_1,\theta)d\theta^2}{x_1},\,\,\, h_1\in C^\infty\,\, \text{and} \,\, h_1(0,\theta)=1/4
\end{equation*}
For example, if $M$ is the hyperbolic plane, $h(x_1,\theta)=(x_1-1)^2/4$. See \cite[Chapter 5]{dyatlov2019mathematical} for more details.
Let $\chi\in C_0^\infty((-1,1))$ with $\chi=1$ in $(-1/2,1/2)$, let us define the 
 metric on $\tilde M$ as 
\begin{equation}\label{tildeM}
\tilde g= \frac{dx_1^2}{4x_1^2}+\frac{(x_1-1)^2 d\theta^2}{4x_1}+\chi(Rx_1)\frac{(h_1(x_1,\theta)-(x_1-1)^2/4)d\theta^2}{x_1}
\end{equation}
where $R$ is a fixed constant. Then we have the metric of $\tilde M$ agrees with $M$ on the set $x_1\le (2R)^{-1}$. Furthermore, note that $|h_1(x_1,\theta)-1/4|\lesssim R^{-1}$ in the support of 
$\chi(Rx_1)$. By choosing $R$ sufficiently large, it is straightforward to check that the Gaussian curvature $K$ (and hence the sectional curvature, since  $\dim\tilde M=2$) of $(\tilde M,g)$ satisfies the uniform bound $-\frac32\le K\le -\frac12$.


We need the following two lemmas.
\begin{lemma}\label{lemma:no-eigen}
    Let
\begin{equation*}
    \tilde g=\frac{dx_1^2}{4x_1^2}+\frac{h_0(x_1)d\theta^2}{x_1}+\chi(Rx_1)\frac{(h_1(x_1,\theta)-h_0(x_1))d\theta^2}{x_1}
\end{equation*}
where $h_0(x_1)=(x_1-1)^2/4$ corresponds to the hyperbolic plane and $h_1(0,\theta)=1/4$.
Then for sufficiently large $R>0$, the Laplacian has no $L^2$ eigenvalue and no resonance at the bottom of the spectrum. 
\end{lemma}
\begin{proof}
The proof strategy is to compare it with the Laplacian on the hyperbolic plane $g_0=\frac{dx_1^2}{4x_1^2}+\frac{h_0(x_1)d\theta^2}{x_1}$. By \cite[Section 5.3]{dyatlov2019mathematical}, we should consider the following operator
\begin{equation*}
    x_1^{i\lambda/2-5/4}(-\Delta_{\tilde g}-\lambda^2-1/4) x_1^{1/4-i\lambda/2}.
\end{equation*}
The eigenvalues of the operator are resonances and the $L^2$ eigenvalues of the original operator $\Delta_{\tilde g}$ corresponds to resonances in $\{\Re \lambda=0, \, \Im \lambda \in [0,1/2)\}$ (see the end of \cite[Section 5.2]{dyatlov2019mathematical}).
This operator has a natural extension on a larger manifold $X$ and is a Fredholm operator from 
\begin{equation*}
    D^s:=\{u \in \bar{H}^s(X): (-4x_1\partial_{x_1}^2-\Delta_{h_0}) u \in \bar{H}^{s-1}(X)  \}\to  \bar{H}^{s-1}(X).
\end{equation*}
We recall the formula \cite[Lemma 5.10]{dyatlov2019mathematical}:
\begin{equation*}
\begin{split}
     x_1^{i\lambda/2-5/4}(-\Delta_{\tilde g}-\lambda^2-1/4) x_1^{1/4-i\lambda/2} &= -4x_1\partial_{x_1}^2+4(i\lambda-1)\partial_{x_1}-\tilde \gamma(4x_1\partial_{x_1}+1-2i\lambda)-\Delta_{\tilde h},\\
     x_1^{i\lambda/2-5/4}(-\Delta_{g_0}-\lambda^2-1/4) x_1^{1/4-i\lambda/2} &= -4x_1\partial_{x_1}^2+4(i\lambda-1)\partial_{x_1}-\gamma_0(4x_1\partial_{x_1}+1-2i\lambda)-\Delta_{h_0},
\end{split}
\end{equation*}
where $\tilde h(x_1,\theta)=h_0(x_1)+ \chi(Rx_1)(h_1(x_1,\theta)-h_0(x_1))$ and
\begin{equation*}
    \tilde \gamma=(2\tilde h)^{-1}\frac{\partial \tilde h}{\partial x_1},\quad \gamma_0=(2h_0)^{-1}\frac{\partial h_0}{\partial x_1}.
\end{equation*}
Therefore,
\begin{equation*}
    x_1^{i\lambda/2-5/4}(\Delta_{\tilde g}-\Delta_{g_0})x_1^{1/4-i\lambda/2}=c_1 \partial_{x_1}+c_2\partial_{\theta}+c_3\partial_{\theta}^2+c_4
\end{equation*}
is a differential operator with coefficients satisfying $|c_1|+|c_2|+|c_3|\lesssim R^{-1}$, and $$c_4=\chi(Rx_1) \varphi_1(x_1,\theta)+Rx_1\chi'(Rx_1)\varphi_2(x_1,\theta)+O(R^{-1}),
$$ where $\varphi_1,\varphi_2$ are smooth functions. 


Since the first resonance of $\Delta_{g_0}$ appears at $-i/2$, we know the resolvent
\begin{equation*}
    \|x_1^{i\lambda/2-1/4}(-\Delta_{g_0}-\lambda^2-1/4)^{-1} x_1^{5/4-i\lambda/2}\|_{\bar{H}^{s-1}(X)\to D^s}\lesssim 1, \quad |\Re \lambda|<0.1,\,\,  -0.3<\Im \lambda<1.
\end{equation*}
We take $s=1$ so that $\Im \lambda>1/2-s$.
We can check that $x_1^{i\lambda/2-5/4}(\Delta_{\tilde g}-\Delta_{g_0})x_1^{1/4-i\lambda/2}$ has a small norm on $D^s\to \bar{H}^{s-1}(X)$, thus the operator with $g_0$ replaced by $\tilde g$ is also invertible by a Neumann series expansion. We use the fact that $u\mapsto \chi(Rx_1) u : H^{1}\to L^2$ and $u\mapsto Rx_1\chi'(Rx_1) u : H^{1}\to L^2$ is bounded by $R^{-1/2}$.
\end{proof}

Lemma~\ref{lemma:no-eigen} implies that the following local smoothing estimate holds: for any $\chi\in C_0^{\infty}(\tilde{M})$, $\beta\in C_0^{\infty}((1/2,2))$ and $\lambda>0$,
\begin{equation}\label{eq:local-sm-tilde}
    \int_{-\infty}^{\infty}\|\chi e^{-is\Delta_{\tilde{g}}} \beta(\sqrt{-\Delta_{\tilde{g}}}/\lambda) f\|_{L^2}^2 ds \lesssim (1+\lambda)^{-1}\|f\|_{L^2}^2.
\end{equation}

We also need the following variant of Lemma~\ref{co}.
\begin{lemma}\label{lem:loc-sm-diff}
Let $(\tilde M,\tilde g)$ be defined as in \eqref{tildeM}. Let $\beta$ be as in \eqref{commuteerror}, and define $\beta_k(s)=\beta(s/2^k)$, $\beta_0(s)= \sum_{k\le 0} \beta(s/2^k)$.
 Suppose $f\in L^2(M)$ is orthogonal to all the $L^2$ eigenvalues, and one of the following holds:
\begin{enumerate}
    \item $-\Delta_g$ does not have resonance at the bottom of the continuous spectrum;
    \item $\mathbf{1}_{[0,1/2+\epsilon]}(\sqrt{-\Delta_g})f=0$ for some $\epsilon>0$. 
\end{enumerate}
Then if $\chi_\infty=1-\chi$ for some $\chi \in C_0^{\infty}(M)$ with $\chi=1$ on $M_{tr}$, we have for $|k_1-k_2|> 10$
\begin{equation}\label{eq:guess}
   \left\| \int_{\mathbb{R}} e^{is\Delta_{\tilde{g}}} \beta_{k_1}(\sqrt{-\Delta_{\tilde{g}}})[\Delta_g,\chi_\infty] \beta_{k_2}(\sqrt{-\Delta_{g}}) e^{-is\Delta_g}f ds \right\|_{L^2}\lesssim_N 2^{-N\max\{k_1,k_2\}} \|f\|_{L^2}.
\end{equation}
Moreover, if $\bigcup_{j} I_j=\mathbb R$ where the intervals $I_j$ have finite overlap and length $\lesssim 1$, then for any $a_j\in C_0^\infty(I_j)$ with $|a_j|\lesssim 1$, we have
\begin{multline}\label{eq:guess-2}
  \left( \sum_{j}  \left\| \int_{I_j} e^{is\Delta_{g}} \beta_{k_1}(\sqrt{-\Delta_{g}})[\Delta_g,\chi] \beta_{k_2}(\sqrt{-\Delta_{g}}) a_j(s)e^{-is\Delta_g}f ds \right\|^2_{L^2}\right)^{\frac12}\\
  \lesssim_N     2^{-N\max\{k_1,k_2\}}    \|f\|_{L^2}.
\end{multline}
Similarly, if $|k_1-k_2|\le 10$, then
\begin{equation}\label{eq:guess1}
 \left\| \int_{\mathbb{R}} e^{is\Delta_{\tilde{g}}} \beta_{k_1}(\sqrt{-\Delta_{\tilde{g}}})[\Delta_g,\chi_\infty] \beta_{k_2}(\sqrt{-\Delta_{g}}) e^{-is\Delta_g}f ds \right\|_{L^2}\lesssim  \|f\|_{L^2},
\end{equation}
and 
\begin{equation}\label{eq:guess-22}
    \left( \sum_{j} \left\| \int_{I_j} e^{is\Delta_{g}} \beta_{k_1}(\sqrt{-\Delta_{g}})[\Delta_g,\chi] \beta_{k_2}(\sqrt{-\Delta_{g}}) a_j(s)e^{-is\Delta_g}f ds \right\|^2_{L^2}\right)^{\frac12}\lesssim \|f\|_{L^2}.
\end{equation}
Moreover, \eqref{eq:guess-2} and  \eqref{eq:guess-22} also hold if we replace $[\Delta_g,\chi]$ by $\chi$ (in \eqref{eq:guess-22}, the right side will be replaced by $2^{-k_2}(k_2+1)\|f\|_{L^2}$).
\end{lemma}

We defer the proof of this lemma to the end of this section.

\begin{proof}[Proof of the global Strichartz estimates \eqref{eq:str-global}]
Note that as a consequence of Lemma~\ref{lemma:no-eigen} and the fact that $\tilde M$ is a simply connected manifold with nonpositive sectional curvature, it follows from Chen \cite[Theorem 1]{ChenSt} that \eqref{eq:str-global} holds on $\tilde M$, namely,
\begin{equation}\label{tildems}
    \|e^{-it\Delta_{\tilde{g}}} u_0\|_{L_t^pL^q_x(\tilde{M}\times\R)}\lesssim \|u_0\|_{L^2}.
\end{equation}

Let $\chi\in C_0^\infty(M)$ with $\chi= 1$ on $M_{tr}$ and $\chi_\infty=1-\chi$. We write
\begin{equation*}
    u=e^{-it\Delta_g} u_0 =\chi_{\infty}e^{-it\Delta_g} u_0 +\chi e^{-it\Delta_g} u_0.
\end{equation*}
For the term $\chi e^{-it\Delta_g} u_0$, let $I_j=[j-1, j+1]$. 
We proceed similar to before by writing $u_j=\alpha(t-j) \chi u$ and 
\begin{equation*}
    (i\partial_t-\Delta_g) u_j = v_j+w_j
\end{equation*}
where
\begin{equation*}
    v_j=i\alpha'(t -j)\chi u,\quad w_j = -\alpha(t  -j) [\Delta_g,\chi] u .
\end{equation*}
Similar to the proof of \eqref{ii.1}, we have (where we use Christ--Kiselev lemma and the Strichartz estimate on the unit interval \eqref{ii.1})
\begin{equation}\label{eq:global-stri-add}
\begin{split}
    &\|\chi e^{-it\Delta_g} u_0\|_{L_t^pL_x^q(M\times \mathbb{R})}^2 \lesssim \sum_{j}\|u_j\|_{L_t^pL_x^q(M\times I_j)}^2\\
    &\lesssim \sum_j \left\|\int_{I_j} e^{is\Delta_g} v_j(s) ds\right\|_{L^2}^2+ \sum_j\left\|\int_{I_j} e^{is\Delta_g} w_j(s) ds\right\|_{L^2}^2.
\end{split}
\end{equation}
We only discuss the second term involving $w_j$ since the first term can be handled similarly. 
We write
\begin{equation}\label{eq:global-stri-diag}
\begin{split}
    &\sum_j\left\|\int_{I_j} e^{is\Delta_g} w_j(s) ds\right\|_{L^2}^2 \lesssim \sum_{j}\sum_{k_1\ge 0}\left\|\int_{I_j} e^{is\Delta_g} \beta_{k_1}(\sqrt{-\Delta_g}) w_j(s) ds\right\|_{L^2}^2\\
    &\lesssim \sum_{j}\left(\sum_{|k_1-k_2|>10}\left\|\int_{I_j} e^{is\Delta_g} \beta_{k_1}(\sqrt{-\Delta_g})\alpha(s-j) [\Delta_g,\chi]\beta_{k_2}(\sqrt{-\Delta_g})e^{-is\Delta_g}u_0 ds\right\|_{L^2}\right)^2\\
    &+\sum_{j}\sum_{|k_1-k_2|\leq 10}\left\|\int_{I_j} e^{is\Delta_g} \beta_{k_1}(\sqrt{-\Delta_g})\alpha(s-j) [\Delta_g,\chi]\beta_{k_2}(\sqrt{-\Delta_g})e^{-is\Delta_g}u_0 ds\right\|_{L^2}^2.
\end{split}
\end{equation}
By \eqref{eq:guess-2} and the Minkowski inequality, the first term in the right side of \eqref{eq:global-stri-diag} is bounded by
\begin{equation*}
    \left(\sum_{k_1,k_2\ge 0}(1+2^{k_1}+2^{k_2})^{-N}\|u_0\|_{L^2}\right)^2\lesssim \|u_0\|_{L^2}^2.
\end{equation*}
 Let  $\tilde \beta \in C_0^\infty(1/4, 4)$ satisfy $\tilde \beta\equiv 1$ in $(1/3, 3)$ as in Lemma~\ref{co}, 
and $\tilde \beta_k(s)=\tilde \beta(s/2^k)$, $\tilde \beta_0(s) =\sum_{k\leq 0} \tilde \beta_k(s)$. By \eqref{eq:guess-22},  the second term in the right side of \eqref{eq:global-stri-diag} is bounded by
\begin{equation*}
    \sum_{|k_1-k_2|\leq 10} \|\tilde \beta_{k_2}(\sqrt{-\Delta_g})u_0\|_{L^2}^2\lesssim \|u_0\|_{L^2}^2.
\end{equation*}
This finishes the proof for the second term on the right hand side of \eqref{eq:global-stri-add}. The first term on the right hand side of \eqref{eq:global-stri-add} can be handled similarly using \eqref{eq:guess-2} and \eqref{eq:guess-22} with $[\Delta_g,\chi]$ replaced by $\chi$.

For the term $\chi_{\infty}e^{-it\Delta_g} u_0$, we write
\begin{equation}\label{infty}
    \chi_{\infty} e^{-it\Delta_g} u_0 = e^{-it\Delta_{\tilde{g}}} \chi_{\infty} u_0 +\frac{i}{2\pi}\int_0^{t} e^{-i(t-s)\Delta_{\tilde{g}}}[\Delta_g,\chi_{\infty}] u(s,\cdot) ds. 
\end{equation}
By \eqref{tildems}, we have
\begin{equation*}
    \|e^{-it\Delta_{\tilde{g}}}\chi_{\infty} u_0\|_{L_t^pL^q_x(\tilde{M}\times \R)}\lesssim \|u_0\|_{L^2}.
\end{equation*}
By the  Christ--Kiselev lemma and \eqref{tildems}, to bound the second term in \eqref{infty} it suffices to estimate
\begin{equation*}
    \left\|\int_0^\infty e^{is\Delta_{\tilde{g}}}  [\Delta_g,\chi_{\infty}] u(s,\cdot) ds\right\|_{L^2(\tilde{M})} \lesssim \|u_0\|_{L^2}
\end{equation*}
This follows from the same argument as in \eqref{eq:global-stri-diag} if we use \eqref{eq:guess} and  \eqref{eq:guess1}. Thus, the proof of \eqref{eq:str-global} is complete.
\end{proof}

\begin{proof}[Proof of Lemma~\ref{lem:loc-sm-diff}]
We first remark that conditions (1) and (2) in Lemma~\ref{lem:loc-sm-diff} are used to ensure the desired local smoothing estimate at low frequencies, i.e., for any $\beta\in C_0^{\infty}(\mathbb{R})$ and $\chi\in C_0^{\infty}(M)$, we have
\begin{equation}\label{eq:loc-sm-low}
    \int_{-\infty}^{\infty} \| \chi e^{-is\Delta_g} \beta(\sqrt{-\Delta_g})f\|_{L^2}^2 ds \lesssim \|f\|_{L^2}^2.
\end{equation}
To see this, note that since there is no embedded eigenvalues in the continuous spectrum (see \cite{mazzeo,Vodev}), we have (see \cite[Theorem 5.33]{dyatlov2019mathematical}) for $\chi\in C_0^{\infty}(M)$ and any fixed constant $C>1$,
\begin{equation*}
     \|\chi (-\Delta_g-\lambda^2\pm i0)^{-1} \chi \|_{L^2\to L^2}\lesssim 1, \quad 1/2+\epsilon/2\leq \lambda\leq C;
\end{equation*}
and under the assumption (1):
\begin{equation*}
    \|\chi (-\Delta_g-\lambda^2\pm i0)^{-1} \chi \|_{L^2\to L^2}\lesssim 1,\quad 1/2-\epsilon\leq \lambda\leq C.
\end{equation*}
The fact that the resolvent estimates imply the local smoothing estimate \eqref{eq:loc-sm-low} is the same as \cite[Theorem 7.2]{dyatlov2019mathematical}.

In high frequencies, the local smoothing estimate holds as before: for $\beta\in C_0^{\infty}((1/2,2))$, $\chi\in C_0^{\infty}(M)$ and $\lambda>10$,
\begin{equation}\label{eq:loc-sm-high-0}
    \int_{-\infty}^{\infty} \|\chi e^{-is\Delta_g}\beta(\sqrt{-\Delta_g}/\lambda) f\|_{L^2}^2 ds\lesssim \lambda^{-1} \log \lambda\|f\|_{L^2}^2.
\end{equation}
If we further assume $\supp \chi \cap \pi(K)=\emptyset$, then
\begin{equation}\label{eq:loc-sm-high}
    \int_{-\infty}^{\infty} \|\chi e^{-is\Delta_g}\beta(\sqrt{-\Delta_g}/\lambda) f\|_{L^2}^2 ds\lesssim \lambda^{-1} \|f\|_{L^2}^2.
\end{equation}

The proofs of \eqref{eq:guess} and \eqref{eq:guess1} are analogous to those of \eqref{eq:guess-2} and \eqref{eq:guess-22} (except we use \eqref{eq:local-sm-tilde} on the left), so we only discuss \eqref{eq:guess-2} and \eqref{eq:guess-22} below. The proof of the last part of Lemma~\ref{lem:loc-sm-diff}, in which $[\Delta_g,\chi]$ is replaced by $\chi$, is similar and therefore omitted.

We first prove \eqref{eq:guess-22}.
Let $\tilde\beta $  be as in Lemma~\ref{co} and $\tilde \beta_k(s)=\tilde \beta(s/2^k)$,  it suffices to show for $P=\sqrt{-\Delta_g}$
\begin{equation}\label{goal}
   \left( \sum_{j\in \mathbb{Z}} \| \int_{I_j} e^{is\Delta_{{g}}} \beta_{k_1}( P)[\Delta_g, \chi] a_j(s) \beta_{k_2}(P)\tilde\beta_{k_2}(P)  e^{-is\Delta_g}f\,ds\|^2_{L^2}\right)^{\frac12}\\
   \lesssim \|f\|_{L^2}.
\end{equation}

As in \eqref{c3}, we may write $\beta_{k_2}(P)=B_{k_2}(P)+C_{k_2}(P)$. 
To handle the term in \eqref{goal} involving $C_{k_2}(P)$, if we apply the local smoothing estimate \eqref{eq:loc-sm-high-0} on the left (for low frequencies estimates when $2^{k_1}\le C$ in \eqref{eq:guess-2}, due to the lack of local smoothing estimate, we can instead use a simple $L^2$ estimate and apply Cauchy--Schwarz inequality in $ds$),  it suffices to show  
\begin{equation}\label{goal1}
   \left( \sum_{j\in \mathbb{Z}} \int_{I_j} \| [\Delta_g, \chi] C_{\la_2}(P)\beta_1(P/\lambda_2) a_j(s) e^{-is\Delta_g}f\|^2_{L^2}ds\right)^{\frac12}\lesssim 2^{-Nk_2}\|f\|_{L^2}.
\end{equation}
If we dyadically decompose frequencies again, \eqref{goal1} would be a consequence of 
\begin{equation}\label{goal2}
 \sum_{j\in \mathbb{Z}} \int_{I_j}  \|\beta_{k_3}(P) [\Delta_g, \chi] C_{k_2}(P)\tilde\beta_{k_2}(P) a_j(s) e^{-is\Delta_g}f\|^2_{L^2}ds\lesssim 2^{-N\max\{k_3,k_2\}}\|f\|_{L^2}^2.
\end{equation}

To prove \eqref{goal2}, if $k_3\le k_2$, then we have 
\begin{equation*}
\begin{split}
  \sum_{j\in \mathbb{Z}} \int_{I_j}  \|\beta_{k_3}(P) &[\Delta_g, \chi] C_{k_2}(P)\tilde\beta_{k_2}(P) a_j(s) e^{-is\Delta_g}f\|^2_{L^2}ds 
\lesssim \\
    & \sum_{j\in \mathbb{Z}} \int_{I_j}  \|\beta_{k_3}(P) \Delta_g \chi C_{k_2}(P)\tilde\beta_{k_2}(P)a_j(s)  e^{-is\Delta_g}f\|^2_{L^2}ds \\
    &+   \sum_{j\in \mathbb{Z}} \int_{I_j}  \|\beta_{k_3}(P) \chi\Delta_g C_{k_2}(P)\tilde\beta_{k_2}(P) a_j(s) e^{-is\Delta_g}f\|^2_{L^2}ds 
\end{split}
\end{equation*}
By using the local smoothing estimates \eqref{eq:loc-sm-low} and \eqref{eq:loc-sm-high-0}, the first term on the right side is bounded by
\begin{equation*}
\begin{split}
      \sum_{j\in \mathbb{Z}} \int_{I_j}&  \|\beta_{k_3}(P) \Delta_g \chi C_{k_2}(P)\tilde\beta_{k_2}(P)a_j(s)  e^{-is\Delta_g}f\|^2_{L^2}ds 
   \\
   &\le  2^{2k_2} \sum_{j\in \mathbb{Z}} \int_{I_j}\|\chi C_{k_2}(P)\tilde \beta_{k_2}(P) a_j(s) e^{-is\Delta_g}f\|^2_{L^2}ds  
   \le  2^{2k_2} \| C_{k_2}(P)f\|^2_{L^2} \lesssim 2^{-Nk_2} \|f\|_{L^2}^2.
\end{split}
\end{equation*}
The second term can be treated similarly.
On the other hand, if $k_3\ge k_2$, we write 
$\beta_{k_3}(P)=B_{k_3}(P)+C_{k_3}(P)$, then we have
\begin{equation*}
    \begin{split}
         & \sum_{j\in \mathbb{Z}} \int_{I_j} \|\beta_{k_3}(P) [\Delta_g, \chi] C_{k_2}(P)\tilde\beta_{k_2}(P) a_j(s) e^{-is\Delta_g}f\|^2_{L^2}ds\\
         &\lesssim 2^{-2Nk_3} \sum_{j\in \mathbb{Z}} \int_{I_j} \|\beta_{k_3}(P) \Delta_g^N [\Delta_g,\chi]C_{k_2}(P) \tilde\beta_{k_2}(P) a_j(s)  e^{-is\Delta_g}f\|_{L^2}^2 ds\\
         &\lesssim 2^{-2Nk_3} \sum_{j\in \mathbb{Z}} \int_{I_j} \|B_{k_3}(P) \Delta_g^N [\Delta_g,\chi]C_{k_2}(P) \tilde\beta_{k_2}(P) a_j(s)  e^{-is\Delta_g}f\|_{L^2}^2 ds\\
         &\quad +2^{-2Nk_3} \sum_{j\in \mathbb{Z}} \int_{I_j} \|C_{k_3}(P) \Delta_g^N [\Delta_g,\chi]C_{k_2}(P) \tilde\beta_{k_2}(P) a_j(s)  e^{-is\Delta_g}f\|_{L^2}^2 ds.
    \end{split}
\end{equation*}
The last term on the right side can be handled similarly as \eqref{goal2}. For the first term, by \eqref{c6} and the local smoothing estimate \eqref{eq:loc-sm-low} and \eqref{eq:loc-sm-high-0}, we have
\begin{equation}\label{goal2c}
    \begin{split}
    &2^{-2Nk_3}\sum_{j\in \mathbb{Z}} \int_{I_j}\|B_{k_3}(P)\Delta_g^{N}[\Delta_g, \chi] C_{k_2} (P)\tilde\beta_{k_2}(P) a_j(s)  e^{-is\Delta_g} f\|_{L^2}^2 ds \\
    &\lesssim 2^{-2Nk_3}\sum_{j\in \mathbb{Z}} \sum_{\ell\leq N}\int_{I_j} \| B_{k_3}(P)[\Delta_g,[\Delta_g,[\cdots,\chi]\cdots]\Delta_g^{\ell} C_{k_2} (P) \tilde\beta_{k_2}(P) a_j(s)  e^{-is\Delta_g} f\|_{L^2}^2 ds\\
    &\lesssim 2^{-(N-1)k_3}\sum_{j\in \mathbb{Z}}\sum_{\ell\leq N} \int_{I_j} \|\tilde \chi\Delta_g^{\ell} C_{k_2} (P)\tilde\beta_{k_2}(P) a_j(s) e^{-is\Delta_g} f\|_{L^2}^2 ds\\
    &\lesssim 2^{-(N-1)k_3}\|f\|_{L^2}^2.
    \end{split}
\end{equation}
Here $\tilde \chi\in C_0^\infty(M)$ is a cut-off function such that $\tilde\chi=1$ in a neighborhood of $\supp \nabla\chi$. 

So it remains to show 
\begin{equation}\label{goal3}
   \sum_{j\in \mathbb{Z}}\| \int_{I_j} e^{is\Delta_{{g}}} \beta_{k_1}( P)[\Delta_g, \chi] a_j(s) B_{k_2}(P)\tilde\beta_{k_2}(P)   e^{-is\Delta_g}f\,ds\|^2_{L^2}\lesssim \|f\|^2_{L^2}.
\end{equation}
which is a consequence of 
\begin{equation}\label{goal4}
   \sum_{j\in \mathbb{Z}} \| \int_{I_j} e^{is\Delta_{{g}}}\tilde\beta_{k_1}(P)B_{k_1}(P)[\Delta_g, \chi]a_j(s) B_{k_2}(P)\tilde\beta_{k_2}(P)  e^{-is\Delta_g}f\,ds\|^2_{L^2}\lesssim \|f\|^2_{L^2},
\end{equation}
and 
\begin{equation}\label{goal4a}
  \sum_{j\in \mathbb{Z}} \| \int_{I_j} e^{is\Delta_{{g}}}\tilde\beta_{k_1}(P)C_{k_1}(P)[\Delta_g, \chi] a_j(s) B_{k_2}(P)\tilde\beta_{k_2}(P)  e^{-is\Delta_g} f \,ds\|^2_{L^2}\lesssim 2^{-Nk_2}\|f\|^2_{L^2}.
\end{equation}
If $\tilde \chi \in C_0^\infty(M)$ is a cut-off function such that $\tilde \chi = 1$ in a $\delta_0$-neighborhood of $\supp\nabla \chi$, then by \eqref{c3} and the finite propagation speed of the wave propagator $\cos(tP)$, the estimates \eqref{goal4} and \eqref{goal4a} are equivalent to
\begin{equation}\label{goal4b}
  \sum_{j\in \mathbb{Z}} \| \int_{I_j} e^{is\Delta_{{g}}}\tilde\beta_{k_1}(P)\tilde\chi B_{k_1}(P)[\Delta_g, \chi] a_j(s) B_{k_2}(P)\tilde\chi \tilde\beta_{k_2}(P) e^{-is\Delta_g} f\,ds\|^2_{L^2}\lesssim \|f\|^2_{L^2}
\end{equation}
and 
\begin{equation}\label{goal4c}
   \sum_{j\in \mathbb{Z}}\| \int_{I_j} e^{is\Delta_{{g}}}\tilde\beta_{k_1}(P) C_{k_1}(P)[\Delta_g, \chi] a_j(s) B_{k_2}(P)\tilde\chi\tilde\beta_{k_2}(P)  e^{-is\Delta_g}f \,ds\|^2_{L^2}\lesssim 2^{-Nk_2}\|f\|^2_{L^2}.
\end{equation}
Here we choose $\delta_0$ in \eqref{c3} sufficiently small so that  $\tilde \chi$ is supported entirely in the nontrapping region.

\eqref{goal4b} follows from the two-fold application of the local smoothing estimates \eqref{eq:loc-sm-low} and \eqref{eq:loc-sm-high} (in low frequencies $2^{k_1}\leq C$, we apply a simple $L^2$ estimate and Cauchy--Schwarz inequality on the left), and the fact that 
\begin{equation}\label{btob}
    \|B_{k_1}(P)[\Delta_g, \chi] B_{k_2}(P)\|_{L^2\to L^2}\lesssim 2^{k_2},
\end{equation}
which follows from the fact that the operator above is a pseudodifferential operator of order 1 with symbol supported in the region $p(x,\xi)\approx 2^{k_2}$. 
The  $2^{k_2}$ loss in \eqref{btobb} can be cancelled by the two-fold application of the local smoothing estimate.

To prove \eqref{goal4c}, by the local smoothing estimates \eqref{eq:loc-sm-low} and \eqref{eq:loc-sm-high},
\begin{equation*}
   \sum_{j\in \mathbb{Z}}\int_{I_j}\|a_j(s) \tilde\chi\tilde\beta_{k_2}(P)  e^{-is\Delta_g}f \|^2_{L^2}ds\lesssim 2^{-{k_2}}\|f\|^2_{L^2}.
\end{equation*}
Hence it suffices to show 
\begin{equation}\label{goal4ca}
  \| \int_{I_j} e^{is\Delta_{{g}}}\tilde\beta_{k_1}(P)C_{k_1}(P)[\Delta_g, \chi] B_{k_2}(P)F(s,\cdot)\,ds\|^2_{L^2}\lesssim 2^{-Nk_2}\|F\|^2_{L^2{(I_j\times M)}}.
\end{equation}
By duality,  this is a consequence of 
\begin{equation}\label{goal4d}
  \int_{I_j} \|B_{k_2}(P) [\Delta_g, \chi] C_{k_1}(P)\tilde\beta_{k_1}(P)  e^{-is\Delta_{ g}}f\|^2_{L^2}ds\lesssim 2^{-Nk_2}\|f\|_{L^2}^2.
\end{equation}
This follows from the same argument as in \eqref{goal2c}. Thus, the proof of \eqref{eq:guess-2} is complete.

To prove \eqref{eq:guess-2}, we simply repeat the above argument. The only difference is that, instead of \eqref{btob}, we use 
\begin{equation}\label{btobb}
    \|B_{k_1}(P)[\Delta_g, \chi] B_{k_2}(P)\|_{L^2\to L^2}\lesssim 2^{-N\max\{k_1,k_2\}},
\end{equation}
which follows from the assumption that $|k_1-k_2|> 10$ and integration by parts after we write out the composition of pseudo-differential operators explicitly. 
 The proof of Lemma~\ref{lem:loc-sm-diff} is complete.
\end{proof}

For later reference, we record the following lemma, which will be used in the proof of the lossless spectral projection estimates.
\begin{lemma}\label{lemsp}
Let  $\beta$ be as in \eqref{commuteerror}. 
Then if $\chi \in C_0^{\infty}(M)$ satisfies $\chi=1$ on $M_{tr}$, for $\lambda\gg 1$ we have 
\begin{equation}\label{a}
     \|[\Delta_g,\chi] \beta(\sqrt{-\Delta_{g}}/\lambda) e^{-is\Delta_g}f \|_{L^2(\mathbb{R}\times M)} \lesssim \la^{\frac12}\|f\|_{L^2}.
\end{equation}
\end{lemma}
The proof of the lemma follows from similar arguments as in \eqref{goal}-\eqref{btobb}, we omit the details here.

\bigskip 

\noindent{\bf 2.2. Lossless spectral projection estimates.}

In this section we shall give the proof of Theorem \ref{thm2}.
We may assume $\delta<(\log\la)^{-1}$ since 
the sharp estimates for $\delta=(\log\la)^{-1}$ follow from Theorem~\ref{bgsp} which 
concerns a larger
class of unbounded manifolds. 

The proof of Theorem \ref{thm2} also relies on the auxiliary manifold  manifold $(\tilde M, \tilde g)$ which agrees with $M$ asymptotically and satisfies favorable spectral projection estimates. 
We define the metric $\tilde g$ on $\tilde M$ as in \eqref{tildeM} so that it is simply connected and has nonpositive sectional curvature.
As in the previous section, we decompose
$M=M_{tr}\cup M_\infty$ where $M_{tr} \subset M$ is compact and contains a neighborhood of the trapped set $\pi(K)$ and 
the metric $\tilde g$ for $\tilde M$ agrees with the metric $g$ on $M_\infty$.

By
Chen-Hassell \cite[Theorem 1.6]{ChHa}, we have the following sharp spectral projection estimates for $\tilde M$:
If $\tilde P=\sqrt{-\Delta_{\tilde g}}$, for $\mu\in [\la/2, 2\la]$ with $\la\gg1$, 
\begin{equation}\label{3.7}
\|\one_{[\mu,\mu+\delta]}(\tilde P)h\|_{L^q(\tilde M)}\lesssim \la^{\mu(q)}\delta^{\frac12}\|h\|_{L^2(\tilde M)}, \, \, \, \delta\in (0,1).
\end{equation}

To prove  Theorem \ref{thm2}, 
it suffices to prove the estimates in \eqref{ii.5} for $q<\infty$  since the bounds for a given $q\in [6,\infty)$ imply those for larger
$q$ by a simple argument using dyadic Sobolev estimates.  So, in what follows, we shall assume that $q\in (2,\infty)$. And if we fix $\beta\in C^\infty_0((1/2,2))$ with $\beta= 1$ in $(3/4,5/4)$, it suffices to replace $f$ in the left side of \eqref{ii.5} with $f_\la=\beta(\sqrt{-\Delta_g}/\la) f$.

Let $\rho\in {\mathcal S}(\R)$ satisfy $\rho(0)=1$ and have Fourier transform vanishing outside of $[-1,1]$, and let $\chi_0\in C_0^\infty(M)$ with $\chi_0= 1$ on $M_{tr}$ 
 and $\chi_\infty=1-\chi_0$.
To 
prove  \eqref{ii.5}, it suffices to show that for $\delta$ in this inequality we have
\begin{equation}\label{3.3}
\| \chi_\infty \rho((\la \delta)^{-1}(-\Delta_g-\la^2))f_\la\|_{L^q(M)}\lesssim 
    \la^{\mu(q)}\delta^{{\frac{1}{2}}}\|f\|_{L^2(M)},
\end{equation}
as well as 
\begin{equation}\label{3.4}
\| \chi_0 \rho((\la \delta)^{-1}(-\Delta_g-\la^2))f_\la\|_{L^q(M)}\lesssim\la^{\mu(q)}\delta^{\frac12} \|f\|_{L^2(M)}.
\end{equation}




To prove \eqref{3.3}, note that if $u=e^{-it\Delta_g}f_\la$ and we set $v=\chi_\infty u$, where $\chi_\infty$ is as above,
then $v$ solves the Cauchy problem on $(M,g)$
\begin{equation}\label{5}
\begin{cases}
(i\partial_t-\Delta_g)v=[\chi_\infty, \Delta_g]u
\\
v|_{t=0}=\chi_\infty f_\la.
\end{cases}
\end{equation}
Since $\Delta_g=\Delta_{\tilde g}$ on $\text{supp }\chi_\infty$, $v$ also solves the following Cauchy problem
on the ``background manifold'' $(\tilde M,\tilde g)$,
\begin{equation}\label{6}
\begin{cases}
(i\partial_t-\Delta_{\tilde g})v=[\chi_\infty, \Delta_g]u
\\
v|_{t=0}=\chi_\infty f_\la.
\end{cases}
\end{equation}
Thus,
\begin{equation}\label{7}
v=e^{-it\Delta_{\tilde g}}(\chi_\infty f)+i\int_0^t e^{-i(t-s)\Delta_{\tilde g}}\bigl(
[\Delta_g,\chi_\infty] u(s, \, \cdot \, )\bigr) \, ds. 
\end{equation}
By using the inverse Fourier transform, \eqref{7} implies
\begin{multline}\label{2.3}
\chi_\infty \rho((\la \delta)^{-1}(-\Delta_g-\la^2))f_\la = \rho((\la \delta)^{-1}(-\Delta_{\tilde g}-\la^2))(\chi_\infty f_\la)
\\
+(2\pi)^{-1} i
\int_{-\infty}^\infty \la\delta \, \Hat \rho\bigl( \la\delta t\bigr) \, e^{-it\la^2}
\, 
\Bigl(\int_0^t e^{-i(t-s)\Delta_{\tilde g}}\bigl([\Delta_g,\chi_\infty]u(s,\, \cdot\,)) \, ds\Bigr) \, dt.
\end{multline}

By using
the spectral projection estimates
 \eqref{3.7}, for $\tilde M$
it is not hard to check that we have the desired bounds for the first term in the right side of \eqref{2.3}.  So to prove \eqref{3.3} it suffices to show that
\begin{equation}\label{3.4.1}
\|R_\la f\|_{L^q(M_\infty)}\lesssim 
    \la^{\mu(q)}\delta^{1/2}\|f\|_{L^2(M)},  \,\,  2<q<\infty,
\end{equation}
where, 
if we set $\tilde \rho(t)=e^{-t} \hat{\rho}(t)$,
\begin{align}\label{3.4.2}
R_\la f&=\la\delta\int_{-\infty}^\infty
e^{-it(\Delta_{\tilde g}+\la^2+i\la\delta )} \tilde \rho(\la\delta t) \Bigl(\int_0^t
e^{is\Delta_{\tilde g}} [\Delta_g,\chi_\infty] \bigl(e^{-is\Delta_g} f_\la \bigr)\, ds\, \Bigr) \, dt
\\
&=-i (\Delta_{\tilde g}+\la^2+i\la \delta )^{-1} (\la\delta)\int_{-\infty}^\infty
e^{-it(\Delta_{\tilde g}+\la^2+i\la\delta )} \frac{d}{dt} \bigl( \tilde \rho(\la\delta t)\bigr) \Bigl(\int_0^t
e^{is\Delta_{\tilde g}} [\Delta_g,\chi_\infty] \bigl(e^{-is\Delta_g} f_\la \bigr)\, ds\, \Bigr) \, dt
\notag
\\
&\quad -i (\Delta_{\tilde g}+\la^2+i\la\delta )^{-1}  \int_{-\infty}^\infty  \la\delta[\Delta_g,\chi_\infty] {\hat{\rho}}(\la\delta t)e^{-it\la^2} e^{-it\Delta_g}f_\la \, dt
\notag
\\
&= -i(\Delta_{\tilde g}+\la^2+i\la\delta )^{-1}  \bigl[ R'_\la f+S_\la f\bigr],
\notag
\end{align}
where $R'_\la$ is the analog of $R_\la$ with $\tilde \rho(\la\delta t)$ replaced by its derivative, and where $S_\la$ is the last integral.

Note that by using Minkowski's integral inequality in the $t$-variable followed by a two-fold application of local smoothing as in the previous section, we have
\begin{equation}\label{3.6}
\|R'_\la f\|_{L^2(M_\infty)}\lesssim (\la\delta)\cdot\la \cdot (\la^{-1/2})^2 \|f\|_{L^2(M)}=\lambda\delta\|f\|_{L^2(M)},
\end{equation}
with the $\la$-factor arising due to the commutator. 
 Here we need to use Lemma~\ref{lem:loc-sm-diff} when applying local smoothing estimates.

Also, it is not hard to use \eqref{3.7} along with the Cauchy-Schwarz inequality and $L^2$ orthogonality to prove that
\begin{equation}\label{3.8}
\| (\Delta_{\tilde g}+\la^2+i\la\delta )^{-1} h\|_{L^q(\tilde M)} \lesssim \la^{\mu(q)-1}\delta^{-\frac12}\|h\|_{L^2(\tilde M)},\,\,\, q<\infty.
\end{equation}
For further details, see, for example, \cite[Proposition 1.3]{SZqm} and its proof.

By \eqref{3.6} and \eqref{3.8} we know that the second to last term in \eqref{3.4.2} satisfies the desired bounds  posited in 
\eqref{3.4.1}.

To handle the other term in \eqref{3.4.2}
 involving $S_\la$, we will rely on the following key result concerning half-localized resolvent
operators on the background manifold.
\begin{proposition}\label{keya}
Let $(\tilde M,\tilde g)$ be defined as in \eqref{tildeM}, which is asymptotically hyperbolic, simply connected and has negative curvature. If $\tilde \chi_0\in C^\infty_0(M_\infty)$ then for $\la\gg 1$ and $\delta\in (0,1/2)$, we have 
\begin{equation}\label{3.10}
\| (\Delta_{\tilde g}+\la^2+i\delta \la)^{-1} (\tilde \chi_0 h)\|_{L^q(\tilde M)}\lesssim 
    \la^{\mu(q)-1}\|h\|_{L^2(\tilde M)},\, 2<q<\infty.
\end{equation}
\end{proposition}

The estimate in \eqref{3.10} shows a gain of $\delta^{\frac12}$
  compared to the estimate in \eqref{3.8}. This gain arises from the presence of the compact cutoff 
$\tilde \chi_0$, and such estimates do not hold on compact manifolds.

We shall postpone the proof of \eqref{3.10} until the end of the section. 
Let us first see how we can use it to handle the last term in \eqref{3.4.2}.  Since $S_\la f $ is 
compactly supported in $M_\infty$, we find from \eqref{3.10}  that
\begin{equation}
    \begin{aligned}
        \| (\Delta_{\tilde g}+\la^2+i\delta \la)^{-1} S_\la f\|_{L^q(\tilde M)}\lesssim 
    \la^{\mu(q)-1}\|S_\la f\|_{L^2( \tilde M)},  \, \,  2<q<\infty.
    \end{aligned}
\end{equation}
On the other hand,  by Schwarz's inequality and Lemma~\ref{lemsp}
$$\|S_\la f\|_{L^2_x(\tilde M)}
\lesssim (\la\delta)^{\frac12}\cdot 
\| [\Delta_g, \chi_\infty] e^{-it\Delta_g}f_\la \|_{L^2([-(\la\delta)^{-1}, \, (\la\delta)^{-1}]\times M)}\lesssim (\la\delta)^{\frac12} \la^{\frac12}\|f\|_{L^2(M)}.$$

Let us see how we can combine this argument along with the proof of the Strichartz estimates in the previous section to obtain \eqref{3.4}.   To this end,
just as before,
choose $\alpha \in C^\infty_0((-1,1))$ satisfying $\sum \alpha(t-j)=1$, $t\in {\mathbb R}$.  Also, let
$$\alpha_j(t)=\alpha((\la/\log\la)t-j),$$
to obtain, like before a smooth partition of unity associated with $\la^{-1}\cdot \log\la$-intervals.  Then, if $\rho$ is as above and 
$$u_j =\alpha_j(t) \chi_0 e^{-it\Delta_g}f_\la,$$
 as was previously done, split
$$\Hat \rho(\la\delta t) u_j(t,x)=-i\Hat \rho(\la\delta t) \int_0^t e^{-i(t-s)\Delta_g}v_j(s,x)ds-i\Hat \rho(\la\delta t) \int_0^t e^{-i(t-s )\Delta_g} w_j(s,x)d s$$
where {$w_j$} and $v_j$ are defined in \eqref{eq:def-ujvj}. Let
$$I_j=[(j-1)\la^{-1}\log\la, (j+1)\la^{-1}\log\la],$$
then
\begin{align*}
\int \la\delta\Hat \rho(\la\delta t)
v_j(t) \, e^{-it\la^2}\, dt &=\la\delta\int_{I_j}
e^{-it(\Delta_g+\la^2+i\la/\log\la)} e^{-t\la/\log\la}\Hat \rho(\la\delta t)\\
&\qquad\cdot
\Bigl( \int_0^t \bigl( e^{is\Delta_g}[ \partial_s, \alpha_j] \, \chi_0 e^{-is\Delta_g}f_\la \bigr) \, ds \, \Bigr) \, dt
\\
&=-i(\Delta_g+\la^2+i\la/\log\la)^{-1} \bigl[ R'_{j,v,\la}f  +S_{j,v,\la}f\bigr],
\end{align*}
with
$$R'_{j,v,\la}f = \la\delta\int_{I_j} e^{-it(\Delta_g+\la^2+i\la/\log\la)}  \frac{d}{dt}\bigl(e^{-t\la/\log\la}\Hat \rho(\la\delta t)\bigr)
\Bigl( \int_0^t \bigl( e^{is\Delta_g} [\partial_s,\alpha_j ] \, \chi_0 e^{-is\Delta_g}f_\la \bigr) \, ds \, \Bigr) \, dt,
$$
and
$$S_{j,v,\la}f =\la\delta\int_{I_j} e^{-it\la^2} \Hat \rho(\la\delta t) [\partial_t, \alpha_j ] \chi_0 e^{-it\Delta_g}f_\la \, dt.$$
Similarly, we set
$$\int \la\delta\Hat \rho(\la\delta t)
w_j(t) \, e^{it\la^2}\, dt
={(\Delta_g+\la^2+i\la/\log\la)^{-1} }
\bigl[ R'_{j,w,\la}f  +S_{j,w,\la}f\bigr],$$
where
\begin{multline*}
R'_{j,w,\la}f 
\\
= \la\delta\int_{I_j} e^{-it(\Delta_g+\la^2+i\la/\log\la)}  \frac{d}{dt}\bigl(e^{-t\la/\log\la}\Hat \rho(\la\delta t)\bigr)
\Bigl( \int_0^t \bigl( e^{is\Delta_g} \alpha_j(s)[\Delta_g,\chi_0] \,  e^{-is\Delta_g}f_\la \bigr) \, ds \, \Bigr) \, dt,
\end{multline*}
and
$$S_{j,w,\la}f =\la\delta\int_{I_j} e^{-it\la^2} \alpha_j(t) \Hat \rho(\la\delta t)  [\Delta_g,\chi_0] e^{-it\Delta_g}f_\la \, dt.$$

Let us fix $\chi_1\in C_0^\infty (M)$ such that $\chi_1\equiv 1$ on the support of $\chi_0$. Then, we have $u_j=\chi_1u_j$, and  
since there are $O(\frac{1}{\delta\log\la})$ nonzero $\Hat \rho v_j$ and $\Hat \rho w_j$ summands, by the Cauchy-Schwarz
inequality, we would obtain \eqref{3.4} if we could show
\begin{multline}\label{3.13}
\Bigl(\, \sum_j \|\chi_1(\Delta_g+\la^2+i\la/\log\la)^{-1} R'_{j,v,\la}f \|^2_{L^q} \, \Bigr)^{1/2}
\\+
\Bigl(\, \sum_j \|\chi_1(\Delta_g+\la^2+i\la/\log\la)^{-1} S_{j,v,\la}f \|^2_{L^q} \, \Bigr)^{1/2}
\lesssim \la^{\mu(q)}\delta(\log\la)^{1/2} \, \, \|f\|_{L^q},
\end{multline}
and
\begin{multline}\label{3.14}
\Bigl(\, \sum_j \|\chi_1(\Delta_g+\la^2+i\la/\log\la)^{-1} R'_{j,w,\la}f \|^2_{L^q} \, \Bigr)^{1/2}
\\+
\Bigl(\, \sum_j \|\chi_1(\Delta_g+\la^2+i\la/\log\la)^{-1} S_{j,w,\la}f \|^2_{L^q} \, \Bigr)^{1/2}
\lesssim \la^{\mu(q)}\delta(\log\la)^{1/2} \, \|f\|_{L^2}.
\end{multline}
As we shall see later, the $\chi_1$ cutoff function is only needed to handle the ``$S-$term'' in \eqref{3.14}.

To prove the bounds for the ``$R'$-terms'' 
in \eqref{3.13} and \eqref{3.14}
we shall make use of the following analog of \eqref{3.8}
\begin{equation}\label{3.15}
\|(\Delta_g+\la^2+i\la/\log\la)^{-1}h\|_{L^q(M)} \lesssim \la^{\mu(q)}(\log\la)^{-1/2} \, (\la/\log\la)^{-1}\, \|h\|_{L^2(M)}.
\end{equation}
This follows from the sharp spectral projection estimates in Theorem~\ref{bgsp} and a simple argument using the Cauchy-Schwarz inequality and $L^2$ orthogonality.

{Let $I_j$ be as above.} We then claim that
\begin{equation}\label{3.16}
\|R'_{j,v,\la}f\|_{L^2(M)}\lesssim \la\delta \, (\la/\log\la)^{1/2} \, \|\chi_0 e^{-is\Delta_g}f_\la\|_{L^2(I_j\times M)}.
\end{equation}
If so, by applying \eqref{3.15} and using local smoothing estimates for $M$, we would obtain
\begin{align*}
\bigl(\sum_j \|\chi_1 &(\Delta_g+\la^2+i\la/\log\la)^{-1} R'_{j,v,\la}f \|^2_{L^q}\bigr)^{1/2}
\\
& \lesssim \la^{\mu(q)}\la\delta (\log\la)^{-1/2}(\la/\log\la)^{-1/2}
\|\chi_0 e^{-is\Delta_g}f_\la\|_{L^2_{s,x}(\R\times M)}
\\
&\lesssim \la^{\mu(q)}\delta(\log\la)^{1/2}\|f\|_{L^2},
\end{align*}
which  
gives us the desired bounds for the first term in the left side of \eqref{3.13}.

To prove \eqref{3.16}, note that {$\int_{I_j}e^{t\la/\log\la} |\tfrac{d}{dt}(e^{-t \la/\log\la}\Hat \rho(\la\delta t))| \, dt =O(1)$}.  Thus, by Minkowski's integral inequality,
$$\| R'_{j,v,\la}f \|_{L^2} \lesssim \la\delta\sup_{t\in I_j} \bigl\| \, \int_0^t e^{is\Delta_g} [\partial_s,\alpha_j] \chi_0 e^{-is\Delta_g} f_\la \, ds \, \bigr\|_{L^2(M)},
$$
which leads to \eqref{3.16} by the arguments used to prove the Strichartz estimates for the $v_j$ terms in the previous subsection.

Similarly, repeating arguments used before we obtain
$$\| R'_{j,w,\la}f \|_{L^2} \lesssim \la\delta\la^{-1/2} \|[\Delta_g,\chi_0] e^{-is\Delta_g}f_\la \|_{L^2(I_j\times M)}.$$
Therefore, by \eqref{3.15} and lemma~\ref{lemsp}
\begin{align*}
\bigl(\sum_j \|\chi_1& (\Delta_g+\la^2+i\la/\log\la)^{-1} R'_{j,w,\la}f \|^2_{L^q}\bigr)^{1/2}\\
& \lesssim \la^{\mu(q)}\la\delta(\log\la)^{-1/2}(\la/\log\la)^{-1}\la^{-1/2}\| [\Delta_g,\chi_0]e^{-is\Delta_g}f_\la\|_{L^2(\R\times M)}
\\
&\lesssim \la^{\mu(q)} \la\delta(\log\la)^{-1/2}(\la/\log\la)^{-1}\la^{-1/2}\cdot \la\cdot \la^{-1/2}\|f\|_{L^2},
\end{align*}
which  
means that we also have the desired bounds for the first term in the left side of \eqref{3.14}.

It remains to estimate the second terms in the left sides of \eqref{3.13} and \eqref{3.14},
i.e., the ``$S$-terms''.
First, by \eqref{3.15} and H\"older's inequality, we have
\begin{multline*}
\| (\Delta_g+\la^2+i\la/\log\la)^{-1}S_{j,v,\la}f \|_{L^q} \lesssim \la^{\mu(q)-1} \, (\log\la)^{1/2}  \| S_{j,v,\la}f\|_{L^2}
\\
\lesssim \la^{\mu(q)}\delta(\log\la)^{1/2}(\la/\log\la)^{-1/2}\|\chi_0 [\partial_s,\alpha_j]e^{-is\Delta_g}f_\la\|_{L^2(I_j\times M)}.
\end{multline*}
  Since $[\partial_s,\alpha_j]$ contributes $\la/\log\la$ to the estimates, if we square and sum over $j$ and use local smoothing estimate  \eqref{eq:local-sm-logloss} (with $[0,1]$ replaced by $\mathbb{R}$) in $M$, we obtain that
\begin{align*}
\Bigl(\, \sum_j \|&\chi_1(\Delta_g+\la^2+i\la/\log\la)^{-1} S_{j,v,\la}f \|^2_{L^q} \, \Bigr)^{1/2} 
\\
&\lesssim \la^{\mu(q)}\delta (\log\la)^{1/2}(\la/\log\la)^{-1/2}(\la/\log\la)^{1/2}\|f\|_{L^2}
\\
&=\la^{\mu(q)}\delta(\log\la)^{1/2}  \|f\|_{L^2},
\end{align*}
as desired.

To estimate the ``$S$-term'' in \eqref{3.14}, we
 shall need the following two-sided $L^2\to L^q$ localized resolvent estimate on $M$.

\begin{proposition}\label{keyb}
Let $(M,g)$ be an asymptotically hyperbolic surface with negative curvature,
 $ \chi_1\in C^\infty_0(M)$ with $\chi_1=1$ on $M_{tr}$, and $ \tilde\chi_1\in C^\infty_0(M_\infty)$  supported away from the trapped set.
Then, for 
$2<q<\infty$
\begin{equation}\label{3.15a}
\| \chi_1(\Delta_{ g}+\la^2+i (\log\la)^{-1} \la)^{-1} ( \tilde \chi_1 h)\|_{L^q( M)} \lesssim \la^{\mu(q)-1}  \, 
\|h\|_{L^2(M)}.
\end{equation}
\end{proposition}
We shall postpone the proof to the end of the section. As will become evident in the proof, similar results also hold on asymptotically hyperbolic manifolds with nonpositive curvature in all dimensions.

To use \eqref{3.15a}, we first note that
since $\chi_0\equiv 1$ on $M_{tr}$, $\nabla \chi_0$, and thus $S_{j,w,\la}$ are supported away from the trapped set $M_{tr}$. So, by the local smoothing estimate \eqref{eq:local-sm-nontrap} and Lemma~\ref{lemsp}, we have
\begin{equation}
    \begin{aligned}
\Bigl(\sum_j \|S_{j,w,\la}f\|^2_{L^2(I_j\times M)}\Bigr)^{1/2}&\lesssim \la\delta(\la/\log\la)^{-1/2}\Bigl(\sum_j \|[\Delta_g,\chi_0]e^{-is\Delta_g}f_\la\|^2_{L^2(I_j\times M)}\Bigr)^{1/2}
\\ &\lesssim \la\delta(\log\la)^{\frac12}\|f\|_{L^2}.        
    \end{aligned}
\end{equation}
Therefore, if we use \eqref{3.15a} and the above arguments, we see that the second term in the left side of \eqref{3.14} also
satisfies the desired bound.

Thus, to finish the proof of  \eqref{3.13} and \eqref{3.14}, it remains to prove Propositions~\ref{keya} and \ref{keyb}.  To do so we shall make use of the following easy consequence of the lossless $L^2$-local smoothing bounds.

\begin{lemma}\label{loc}  {Let} $\mu\in [\la/2,2\la]$, $\la\gg 1$, $\delta\in (0,1/2)$ and $\tilde \chi\in C^\infty_0(M_\infty)$. {Then}
\begin{equation}\label{3.17}
\| \one_{[\mu,\mu+\delta)}(\tilde P) (\tilde \chi h)\|_{L^2(\tilde M)}\lesssim \delta^{1/2}\|h\|_{L^2(\tilde M)}.
\end{equation}
and
\begin{equation}\label{3.19}
\| \one_{[\mu,\mu+\delta)}(P) (\tilde \chi h)\|_{L^2( M)}\lesssim \delta^{1/2}\|h\|_{L^2( M)}.
\end{equation}
\end{lemma}

\begin{proof}  Choose $a\in {\mathcal S}(\R)$  satisfying $\supp \,  \Hat a\subset (-1,1)$ and $a(t)\ge 1$ on $[-20,20]$ and let $\tilde \beta=1$ on 
$[1/10,10]$ and supported in $[1/20,20]$.  Then, by orthogonality and duality, for $\mu\in [\la/2,2\la]$ 
\begin{align*}
\|\one_{[\mu,\mu+\delta)} (\tilde P)\tilde \chi \|_{L^2\to L^2}&\le \|a((\la\delta)^{-1}(-\Delta_{\tilde g}-\mu^2)) \, \tilde \beta(\sqrt{-\Delta_{\tilde g}}/\la)\tilde\chi \|_{L^2\to L^2}
\\
&= \bigl\| 
\tilde \chi\, \tilde\beta(\sqrt{-\Delta_{\tilde g}}/\la)  a((\la\delta)^{-1}(-\Delta_{\tilde g}-\mu^2)) \, \|_{L^2\to L^2}.
\end{align*} 
By using  the Cauchy-Schwarz inequality and the lossless local smoothing estimates, we have 
\begin{align*}
 \|\tilde \chi\,& \tilde\beta(\sqrt{-\Delta_{\tilde g}}/\la)  a((\la\delta)^{-1}(-\Delta_{\tilde g}-\mu^2)) \, h\|_{L^2(\tilde M)}
\\
&= \bigl\| 
\int_{-(\la\delta)^{-1}}^{(\la\delta)^{-1}} (\la\delta) \Hat a(\la\delta t) e^{-it\mu^2} \tilde\chi e^{-it\Delta_{\tilde g}} \tilde\beta(\sqrt{-\Delta_{\tilde g}}/\la) h \, dt \|_{L^2(\tilde M)}
\\
&\lesssim (\la\delta)\cdot (\la\delta)^{-1/2}\la^{-1/2}\|h\|_{L^2(\tilde M)}.
\end{align*}
Here
we require the variant of  \eqref{eq:local-sm-nontrap} with $[0,1]$ replaced by $\R$.  This leads to \eqref{3.17}.  A similar argument yields \eqref{3.19}.
\end{proof}

\begin{proof}[Proof of Proposition~\ref{keya}]

We first note that, by adjusting the values of $\delta$ slightly if necessary, proving \eqref{3.10} is equivalent to 
showing that  for all $\la\gg 1$, {$\delta \in ( 0,1/2)$},
\begin{equation}\label{3.10a}
\| (\Delta_{\tilde g}+(\la+i\delta)^2)^{-1} (\tilde \chi_0 h)\|_{L^q(\tilde M)}\lesssim 
    \la^{\mu(q)-1}\|h\|_{L^2(\tilde M)}, \,\, 2<q<\infty.
\end{equation}

Recall that if $\tilde P=\sqrt{-\Delta_{\tilde g}}$, we have the following identity (see e.g., \cite{BSSY})
\begin{equation}\label{resolvent}
    \left(\Delta_{\tilde g}+(\la+i\delta)^2\right)^{-1}=\frac1{i(\la+i\delta)}\int_0^\infty e^{it\la-t\delta}\cos (t\tilde P)\,dt.
\end{equation}

Let us fix $\beta\in C_0^\infty ((1/2, 2))$ satisfying $\sum_{j=-\infty}^\infty \beta(s/2^j)=1$, and 
define 
\begin{equation}\label{tj}
T_j f=\frac1{i(\la+i\delta)}\int_0^\infty \beta(2^{-j}t)e^{it\la-t\delta}\cos (t\tilde P)f\,dt.
\end{equation}
Then it suffices to obtain suitable bounds for the $T_j$ operators.  It is also straightforward to check that the symbol of $T_j$ is 
\begin{equation}\label{symbol}
    T_j(\tau)=\frac1{i(\la+i\delta)}\int_0^\infty \beta(2^{-j}t)e^{it\la-t\delta}\cos (t\tau)f\,dt=O(\la^{-1}2^j(1+2^j|\tau-\la|)^{-N}).
\end{equation}


Note that by \eqref{resolvent} we have $(\Delta_{\tilde g}+(\la+i\delta)^2)^{-1}=\sum_{-\infty}^\infty T_j$.  To prove the
half-resolvent estimates \eqref{3.10a}, it will be natural to separately consider
the contribution of the terms with
$2^j\le 1$, $1\le 2^j\lesssim \log \la$ and $\log\la \lesssim 2^j$.


\noindent(i) $2^j\le 1$.

This case can be handled using the local arguments in 
Bourgain, Shao, Sogge and Yao \cite{BSSY}, as well as related earlier work  of Dos Santos Ferreira, Kenig and Salo~\cite{DKS}, where resolvent estimates on compact manifolds were considered.  The $\tilde\chi_0$ cutoff function is not needed in this case.

First, if $2^j \in  [\la^{-1},1]$, we will show that
\begin{equation}\label{tjbound}
    \|T_j\|_{L^2\to L^q}\lesssim \la^{\mu(q)-1}2^{j/2},\,\,q\ge 6.
\end{equation}
This would  yield the desired result 
$\| \sum_{\la^{-1} <2^j \le 1} T_j\|_{2\to q}=O(\la^{\mu(q)-1})$ for all $q>2$
by interpolating with the trivial $L^2\to L^2$ bound and summing over $j$.

To prove \eqref{tjbound}, as in \cite{BSSY},
by using the Hadamard parametrix for $\cos (t\tilde P)$, it is not hard to show that if $\la^{-1}\le 2^j\le 1$, the kernel of $T_j$ operators satisfies
\begin{equation}
    T_j(x,y)=\begin{cases}
      \la^{-1/2}2^{-j/2} e^{i\la d_{\tilde g}(x,y)}a_\la(x,y),\,\, d_{\tilde g}(x,y)\in [2^{j-2},2^{j+2}] \\
     O( \la^{-1} 2^{-j}), \,\,\,d_{\tilde g}(x,y)\le 2^{j-2},
    \end{cases}
\end{equation}
where $ |\nabla_{x,y}^\alpha a_\la(x,y)|\le C_\alpha d_{\tilde g}(x,y)^{-\alpha}$. Additionally, by the finite propagation speed property of the wave propagator,  $T_j(x,y)=0$ if $d_{\tilde g}(x,y)\ge 2^{j+2}$.
Thus, if $d_{\tilde g}(x,y)\in [2^{j-2},2^{j+2}]$, the bound in \eqref{tjbound} {follows} from the 
 oscillatory integral bounds of H\"ormander \cite{HormanderIII}
and Stein \cite{steinbeijing}, combined with a  scaling argument. And the other case when $d_{\tilde g}(x,y)\le 2^{j-2}$ follows from Young's inequality.

On the other hand, if $2^j\le \la^{-1}$, 
 by integration by parts in $t$-variable once, one can show that the symbol of the operator $\sum_{\{j:2^j\le \la^{-1}\}} T_j(\tilde P)$ satisfies
\begin{equation}\nonumber
   \sum_{\{j:2^j\le \la^{-1}\}} T_j(\tau)=O(\la^{-1}(\la+|\tau|)^{-1}).
\end{equation}
Since we are assuming that $q<\infty$, by Sobolev estimates we  have
$$\|\sum_{\{j:2^j\le \la^{-1}\}} T_j(\tilde P)(\tilde \chi_0 h)\|_{L^q(\tilde M)}
\lesssim \la^{-1}\|\tilde \chi_0 h\|_{L^2(\tilde M)}\lesssim \la^{-1}\| h\|_{L^2(\tilde M)}.
$$



To deal with the two remaining cases corresponding to sums over  $2^j\ge1$,  we shall require the following lemma.
\begin{lemma}\label{kernelhyper}
Let $(\tilde M,\tilde g)$ be defined as in \eqref{tildeM}, which is asymptotically hyperbolic, simply connected and has negative curvature. For $2^j\ge 1$, if $T_j$ is defined as in \eqref{tj}, we have
    \begin{equation}\label{kernelnonlocal}
    \|T_j f\|_{L^\infty(\tilde M)}\lesssim_N (\la^{1/2}e^{-2^{j-4}}+(2^j\la)^{-N} )\|f\|_{L^1(\tilde M)}.
\end{equation}
\end{lemma}
\begin{proof}[Proof of Lemma~\ref{kernelhyper}]
The proof of \eqref{kernelnonlocal} relies heavily on the kernel estimate for the spectral measure of $\tilde P$ established in the work of Chen and Hassell \cite{ChHa}.
Recall that if $dE_{\tilde P}(\mu)$ denote the spectral measure for $\tilde P$, we have
\begin{equation}\label{tja}
T_jf=\frac1{i(\la+i\delta)}\int_0^\infty \int_0^\infty \beta(2^{-j}t)e^{it\la-t\delta}\cos (t\mu) dE_{\tilde P} (\mu) f\, d\mu dt.
\end{equation}

Let us collect several useful facts about the spectral measure $\tilde P$. 
For high energies $\la\gg 1$, by \cite[Theorem 1.3]{ChHa}, we have 
\begin{equation}
    dE_{\tilde P} (\la)(x,y)=\sum_\pm\la e^{\pm i\la d_{\tilde g}(x,y)}b_\pm(\la,x,y)+a(\la,x,y),
\end{equation}
where $d_{\tilde g}(x,y)$ denotes the distance function on $\tilde M$,
\begin{equation}\label{bb}
 \Bigl|\Bigl( \frac{d}{d\la} \Bigr)^j b_\pm(\la,x,y)\Bigr| \lesssim_j \begin{cases}
      \la^{-j}(1+\la d_{\tilde g}(x,y))^{-\frac12},\,\, d_{\tilde g}(x,y)\le 1\\
      \la^{-\frac12-j}e^{-\frac{d_{\tilde g}(x,y)}{2}},\,\, d_{\tilde g}(x,y)\ge 1,
  \end{cases}
\end{equation}
and 
\begin{equation}\label{ab}
 \Bigl|\Bigl( \frac{d}{d\la} \Bigr)^ja(\la,x,y)\Bigr|\lesssim_{j,N} 
      \la^{-N}.
\end{equation}

If we fix $\rho\in C_0^\infty(1/4, 4)$ with $\rho=1$ in (1/2, 2), and define
\begin{equation}\label{tjb}
\tilde T_j f=\frac1{i(\la+i\delta)} \int_0^\infty \int_0^\infty \beta(2^{-j}t)e^{it\la-t\delta}\cos (t\mu)\rho(\mu/\la) dE_{\tilde P} (\mu) f\, d\mu dt,
\end{equation}
then, by  integrating by parts in the $t$ variable, we see that  the symbol of the operator $T_j-\tilde T_j$ is $O\left((2^j(|\tau|+\la))^{-N}\right)$. Thus, by using dyadic Sobolev estimates, it is not hard to show
 \begin{equation}\label{kernelnonlocala}
\|(T_j-\tilde T_j) f\|_{L^\infty(\tilde M)}\lesssim_N(2^j\la)^{-N} \|f\|_{L^1(\tilde M)}. 
\end{equation}

 Consequently, it suffices to show that the operators $\tilde T_j$ satisfy the desired $L^1\to L^\infty$  bound in \eqref{kernelnonlocal}.
If $d_{\tilde g}(x,y)\notin [2^{j-2}, 2^{j+2}]$, then by using \eqref{bb} and \eqref{ab}, and performing integration by parts in the $\mu$ variable enough times, we have 
\begin{equation}\label{c}
|\tilde T_j (x,y)|\le C_{N} \lambda^{-N}2^{-Nj}.
\end{equation}

On the other hand, if $d_{\tilde g}(x,y)\in [2^{j-2}, 2^{j+2}]$, \eqref{kernelnonlocal} follows directly from the pointwise bound in  \eqref{bb} and \eqref{ab}. This completes the proof of \eqref{kernelnonlocal}.
\end{proof}

Using this lemma we can handle the contribution of the $T_j$ terms with: 
\newline \noindent(ii) $2^j\ge C\log\la$.

First we note that, by spectral theorem,
\begin{equation}\label{a1}
    \|T_j \chi_0f\|_{L^2}\lesssim \la^{-1}2^j \|\chi_0f\|_{L^2}\lesssim \la^{-1}2^j \|f\|_{L^2}.
\end{equation}
And by \eqref{kernelnonlocal} and Schwarz's inequality, we also have
\begin{multline}\label{a2}
    \|T_j \chi_0f\|_\infty\lesssim_N (\la^{1/2}e^{-2^{j-4}}+(2^j\la)^{-N} )\|\chi_0f\|_1\\
    \lesssim_N (\la^{1/2}e^{-2^{j-4}}+(2^j\la)^{-N} )\|f\|_2.
\end{multline}
 Thus by  \eqref{a1}, \eqref{a2} and H\"older's inequality, we have
\begin{equation}\label{1a}
    \|T_j \chi_0f\|_{L^q}\lesssim \left(\la^{1/2-3/q}2^{\frac{2j}{q}}e^{-\tfrac{(q-2)2^{j-4}}{q}}+O_{q,N}(\la^{-N})\right) \|f\|_{L^2}.
\end{equation}
Since we are assuming $2^j\ge C\log\la$, we may choose $C$ sufficiently large so that $$2^{\frac{2j}{q}}e^{-\tfrac{(q-2)2^{j-4}}{q}}\lesssim \la^{-2}.$$ 
We can then sum over such $j$ to obtain the desired bound $\| \sum_{2^j\ge C\log \la } T_j\|_{L^2\to L^q}=O(\la^{\mu(q)-1})$.


\bigskip

Our final case involves the $T_j$ with:
\newline
\noindent(iii) $1\le 2^j\le C\log\la$.

To handle the contribution of these terms, we shall first prove that for each fixed $j$ with $2^j\ge1$, we have the uniform bounds
\begin{equation}\label{tjblarge}
   \| T_j\chi_0 f\|_{L^{q}(\tilde M)} \lesssim \la^{\mu(q)-1}  \|f\|_{L^2(\tilde M)}.
\end{equation}
To see this, let us define
\begin{equation}\label{ej}
    E_{\la,j,k}=\one_{[\la+2^{-j}k, \la +(k+1)2^{-j})}( \tilde P).
 \end{equation} 
 By using \eqref{3.7} along with \eqref{3.17}  for $\delta=2^{-j}$,
we have 
\begin{align*}
\|\one_{[\la/2,2\la]}( \tilde P) \, & T_j\chi_0 f\|_{L^{q}(\tilde M)}
\\
&\le \sum_{|k|\lesssim \la 2^{j}}
\|\one_{[\la/2,2\la]}( \tilde P) \, E_{\la,j,k} T_j\chi_0 f\|_{L^{q}(\tilde M)}
\\
&\le  \la^{\mu(q)}2^{-j/2}  \sum_{|k|\lesssim \la2^j}
\|\one_{[\la/2,2\la]}( \tilde P) E_{\la,j,k} T_j\chi_0 f\|_{L^2( \tilde M)}
\\
&\lesssim \la^{\mu(q)}2^{-j/2}\sum_{|k|\lesssim \la2^j} (1+|k|)^{-N} \la^{-1}2^j  \|\one_{[\la/2,2\la]}( \tilde P) E_{\la,j,k} \chi_0f\|_{L^2(\tilde M)}
\\
&\lesssim \la^{\mu(q)-1}  \|f\|_{L^2(\tilde M)},
\end{align*}
using \eqref{3.17} in the last step.
The case when the spectrum is outside $[\la/2,2\la]$ can be handled using Sobolev estimates, as in case (i). Thus, the proof of \eqref{tjblarge} is complete.

In view of \eqref{tjblarge}, it suffices to consider the values of  $j$ such that $C_0\le 2^j\le c_0\log\la$ where $C_0$ is sufficiently large and $c_0$ is sufficiently small. We shall specify the choices of $C_0$ and $c_0$ later in the proof. Furthermore, as shown in the proof of \eqref{kernelnonlocal}, $|T_j(x,y)|=O(\la^{-N})$ if $d_{\tilde g}(x,y)\notin [2^{j-2}, 2^{j+2}]$. Hence, if, as we may, 
we assume $\tilde \chi_0$ is supported in a small neighborhood of some point $y_0$, it suffices to show that 
\begin{multline}\label{n2}
\|\sum_{\{j:C_0\le 2^j\le c_0\log\la\}} T_j(\tilde \chi_0 h)\|_{L^q(S)}\lesssim 
    \la^{\mu(q)-1}\|h\|_{L^2(\tilde M)},
    \\
    \text{where }  \, S=\{x\in \tilde M: \frac {C_0}{4}\le d_{\tilde g}(x,y_0)\le 4c_0\log\la\}.
\end{multline}

To proceed, first note that by \eqref{symbol},
if we fix $\beta \in C^\infty_0((1/4,4))$ with $\beta= 1$ on (1/2,2), it suffices to show 
\begin{equation}\label{n3}
    \|\sum_{\{j:C_0\le 2^j\le c_0\log\la\}}\beta({\tilde{P}}/\la) T_j(\tilde \chi_0 h)\|_{L^q(S)}\lesssim 
    \la^{\mu(q)-1}\|h\|_{L^2(\tilde M)}.
\end{equation}

To prove \eqref{n3}, we need to introduce microlocal cutoffs involving pseudodifferential operators. If we fix {$\delta_0$} with $0<{\delta_0}\ll 1$, 
then since $\tilde M$ has bounded geometry, we can cover the set $S$ by a partition of unity $\{\psi_k\}$, which satisfies
\begin{equation}\label{n4}
1=\sum_k \psi_k(x), \quad \supp\,  \psi_k\subset B(x_k,{\delta_0}),
\end{equation}
with $|\partial_x^j \psi|\lesssim 1$ uniformly in the normal coordinates around $x_k$ for different $k$.  Here  $B(x_k,{\delta_0})$ denotes geodesic balls of radius {$\delta_0$} $\text{with } 
d_{\tilde g}(x_k,x_\ell)\ge {\delta_0} \, \, \text{if } \, k\ne \ell$,
and finitely many
 balls $B(x_k,2{\delta_0})$ overlap.  By a simple volume counting argument, the number of values of $k$ for which $\supp \, \psi_k\cap S\neq \emptyset$ is $O(\la^{Cc_0})$ for some fixed constant $C$. See \eqref{b.0}–\eqref{b.1} in the next section for more  discussions about the properties of manifolds with bounded geometry.

If we extend $\beta\in C_0^\infty ((1/4, 4))$ to an even function by letting $\beta(s)=\beta(|s|)$, then we can choose an even function $\rho\in C_0^\infty$   satisfying $\rho(t)= 1$, $|t|\le \delta_0/4$ and $\rho(t)=0 $, $|t|\ge \delta_0/2$ such that
\begin{equation}\label{n5'}
\begin{aligned}
     \beta(\tilde P/\la)=&(2\pi)^{-1}\int_\R\la\hat \beta(\la t)\cos t\tilde P dt \\
     =&(2\pi)^{-1}\int\rho(t)\la\hat \beta(\la t)\cos t\tilde P dt+ (2\pi)^{-1}\int (1-\rho(t)) \la\hat \beta(\la t)\cos t\tilde P dt\\
     =& B+C.
\end{aligned}
\end{equation}
  It is not hard to check that the symbol of the operator $C$ is $O((1+|\tau|+\la)^{-N})$. Therefore, by Sobolev estimates, we have $ \|C\|_{L^2\to L^q}\lesssim_N \la^{-N}$.
On the other hand, by  using the finite propagation speed property of the wave propagator, one can argue as in the compact manifold case to show that $B$ is a pseudodifferential operator with principal symbol $\beta(p(x,\xi))$,
with $p(x,\xi)$ here being the principal symbol of $\tilde P$. 
See Theorem 4.3.1 in \cite{SFIO2} for more details.

Next, choose $\tilde \psi_k\in C^\infty_0$ with $\tilde \psi_k(y)=1$ for $y\in B(x_k, \tfrac54 \delta_0)$ and
$\tilde \psi_k(y)=0$ for $y\notin B(x_k,\tfrac32 \delta_0)$.  As with the $\psi_k$ we may assume that the $\tilde \psi_k$ have bounded derivatives
in the normal coordinates about $x_k$ by taking $\delta_0>0$ small enough given that $\tilde M$ is of bounded geometry.  Then, if $B(x,y)$ is the kernel of $B$, 
we have $\psi_k(x)B(x,y)=\psi_k(x)B(x,y)\tilde \psi_k(y)+O(\la^{-N})$, and so
%
\begin{equation}\label{n5}
\begin{aligned}
    &\psi_k(x)B(x,y)\\
    &\quad=(2\pi)^{-n}\la^n\int e^{i\la\langle x-y,\xi \rangle} \psi_k(x){\beta(p(x,\xi))}\tilde\psi_k(y)d\xi +R_k(x,y)\\
    &\quad= A_k(x,y)+R_k(x,y).
\end{aligned}
\end{equation}
$R_k$ is a lower order pseudodifferential operator which satisfies 
\begin{equation}\label{rk}
    \|R_k\|_{L^2\to L^q}\lesssim \la^{-1+2(\frac12-\frac1q)},\,\,q\ge 2.
\end{equation}
Moreover, 
$$ R_k(x,y)=0,\,\,\, \text{if } \, x\notin B(x_j,{\delta_0}) \, \,
\text{or } \, \, y\notin B(x_j,3{\delta_0}/2).
$$

{{
Let $$H_p=\frac{\partial p}{\partial \xi}\frac{\partial }{\partial x}-\frac{\partial p}{\partial x}\frac{\partial }{\partial \xi}$$ denote the Hamilton vector field associated with the principal symbol $p(x,\xi)$ of $\tilde P$.
Let $\Phi_t=e^{tH_p}: {T^*M\setminus 0\to T^*M\setminus 0}$ denote the geodesic flow on the cotangent bundle generated by $H_p$.
For each $x_k$, let $\omega_k$ be the unit covector such that $\Phi_{-t}(x_k, \omega_k)=(y_0, \eta_0)$ for some $\eta_0$ and $t=d_{\tilde g}(x_k, y_0)$,
with $y_0$ as in \eqref{n2}. 
We define $a_k(x,\xi)\in C^\infty$ such that in the normal coordinate around $x_k$,}}
\begin{equation}\label{n6}
a_{k}(x,\xi)=0 \, \, 
\text{if } 
\bigl|\tfrac{\xi}{|\xi|_{\tilde g(x)}}-\omega_k\bigr|\ge 2\delta_1,\,\,\,\text{and}\,\,\,a_{k}(x,\xi)=1 \, \, 
\text{if } 
\bigl|\tfrac{\xi}{|\xi|_{\tilde g(x)}}-\omega_k\bigr|\le \delta_1.
\end{equation}
Here $|\xi|_{\tilde g(x)}=p(x,\xi)$, $\delta_1\ll1$ is a fixed small constant that will be chosen later. 
By the proof of Lemma~\ref{propogation} below, we may assume that $\partial^\alpha_x \partial^\gamma_{\xi} a_k=O(1)$ if $p(x,\xi)=1$, independent of $k$, with $\partial_x$ denoting
derivatives in the normal coordinate system about $x_k$.

We finally define the kernel of the microlocal cutoffs $A_{k,0}$ and $A_{k,1}$ as 
\begin{equation}\label{n7}
\begin{aligned}
A_k(x,y)&=A_{k,0}(x,y)+A_{k,1}(x,y)\\
     &=(2\pi)^{-n}\la^n\int e^{i\la\langle x-y,\xi \rangle}\psi_k(x)a_k(x,\xi)\beta((p(x,\xi))\tilde\psi_k(y)d\xi\\
     &\quad+ (2\pi)^{-n}\la^n\int e^{i\la\langle x-y,\xi \rangle}\psi_k(x)(1-a_k(x,\xi))\beta((p(x,\xi))\tilde\psi_k(y)d\xi.
\end{aligned}
\end{equation}
The above operators satisfy
\begin{equation}\label{n10}
\| A_{k,\ell}\|_{L^p(\tilde M)\to L^p(\tilde M)}=O(1), \,\, 1\le p\le \infty,\,\,\ell=0,1.
\end{equation}
This, combined with the support properties of $ A_{k,\ell}$  implies that
\begin{equation}\label{n11}
\| \sum_{k}A_{k,\ell}h\|_{L^p(\tilde M)}\lesssim \|h\|_{L^p(\tilde M)}, \,\, 1\le p\le \infty,\,\,\ell=0,1.
\end{equation}

By \eqref{n4}, \eqref{n5} and \eqref{n7}, to prove \eqref{n3}, it suffices to show 
\begin{equation}\label{n12}
        \|\sum_{\{j:C_0\le 2^j\le c_0\log\la\}}\sum_{k}A_{k,0} T_j(\tilde \chi_0 h)\|_{L^q(S)}\lesssim 
    \la^{\mu(q)-1}\|h\|_{L^2(\tilde M)},
\end{equation}
as well as 
\begin{equation}\label{n13}
        \|\sum_{\{j:C_0\le 2^j\le c_0\log\la\}}\sum_{k}A_{k,1} T_j(\tilde \chi_0 h)\|_{L^q(S)}\lesssim_N
    \la^{-N}\|h\|_{L^2(\tilde M)}.
\end{equation}
The other term involving $R_k$ is more straightforward to handle. More explicitly, by \eqref{rk}, the support property of $R_k$, and \eqref{a1},  we have 
\begin{equation}\label{n12a}
\begin{aligned}
       \|&\sum_{\{j:C_0\le 2^j\le c_0\log\la\}}\sum_{k}R_{k} T_j(\tilde \chi_0 h)\|_{L^q(S)}\\
       &\lesssim \la^{-1+2(\frac12-\frac1q)} \|\sum_{\{j:C_0\le 2^j\le c_0\log\la\}} T_j(\tilde \chi_0 h)\|_{L^2(\tilde M)}\\
       &\lesssim 
    \la^{-1+2(\frac12-\frac1q)} \sum_{\{j:C_0\le 2^j\le c_0\log\la\}}\la^{-1}2^{j}\|\tilde \chi_0 h\|_{L^2(\tilde M)}.
\end{aligned}
\end{equation}
Note that $-1+2(\frac12-\frac1q)<\mu(q)-\frac12$ for all $q\ge 2$. Therefore,  by choosing $c_0$ sufficiently small, the bound in \eqref{n12a} is better than the estimate in \eqref{n3}.

The main reason that \eqref{n13} holds is that, by \eqref{n6}, the microlocal support of the operator $A_{k,1}$ does not propagate to the support of 
$\tilde\chi_0$
  along the backward geodesic flow, which leads to the rapidly decaying term $O(\la^{-N})$. Similarly, the analog of \eqref{n13} holds for the operator $A_{k,0}$ if we replace $T_j$ with its adjoint $T_j^*$, as the microlocal support of the operator $A_{k,0}$ does not propagate to the support of 
$\tilde\chi_0$
  along the forward geodesic flow.  Therefore, one can replace $T_j$ in \eqref{n12} with $T_j-T_j^*$ by introducing a rapidly decaying term.
  
  On the other hand,  one can show that the half localized operator, involving the difference of the resolvent operator $(\Delta_{\tilde g}+(\la+i\delta)^2)^{-1}$ and its adjoint, satisfies the desired bound by a simple argument using the sharp spectral projection bound \eqref{3.7} along with Lemma~\ref{loc} (for a precise statement, see \eqref{n18}). This  allows us to obtain the estimate \eqref{n12} with $T_j$ replaced with $T_j-T_j^*$, which will conclude the proof combined with the other parts.

More explicitly, to prove \eqref{n13}, we shall require the following two lemmas.
\begin{lemma}\label{uniform}
Let $\psi_k$ and $\tilde \psi_k$ be defined as in \eqref{n4} and \eqref{n5}.
    Fix $x\in \supp \,\tilde\psi_k$, let $\omega_0(x)$ be the unit covector such that, 
    if $y_0$ is as in \eqref{n2},
    \begin{equation}\label{omega}
        \Phi_{-t_0}(x,\omega_0(x))=(y_0,\eta_0)\,\,\,\text{for some unit covector} \,\,\eta_0\,\,\, \text{and} \,\,\,t_0=d_{\tilde g}(x,y_0).
    \end{equation}
    If $\supp\,\psi_k \cap S\neq \emptyset$, we have
    $\omega_0(x)$ is uniformly continuous in $x$ over $\supp\,\tilde \psi_k$.
\end{lemma}
\begin{proof}
Note that the covector field $\omega_0(x)=\nabla_x d_{\tilde g}(x,y_0)$ is a $1$-form where $\nabla_x$ is the covariant derivative in $x$ (which is the ordinary differential when it acts on functions). Let $\gamma_{y_0,x}(t)$ be the unit speed geodesic emanating from $y_0$ with $\gamma_{y_0,x}(d_{\tilde g}(x,y_0))=x$. By the second variation formula \cite[(1.17)]{CheegerEbin}, for normalized vector fields $X$ and $Y$ perpendicular to $\gamma'_{y_0,x}(t)$, we have
\begin{equation*}
    \nabla_x  \omega_0(x)(X,Y)=\nabla_x^2 d_{\tilde g}(x,y_0)(X,Y)= \langle J_X'(t),Y\rangle|_{t=d_{\tilde g}(x,y_0)}
\end{equation*}
where $J_X(t)$ is the Jacobi field along $\gamma_{y_0,x}(t)$ such that
$$J_X(0)=0,\quad J_X(d_{\tilde g}(x,y_0))=X.$$
{Since $\tilde{M}$ has negative curvature, $J_X(t)$ is well defined by the above conditions.}

{{By the Hessian comparison theorem \cite[Theorem A]{hesscompare}, since the curvature $K$ of $\tilde M$ is bounded below by $-2$, $\nabla_x^2 d_{\tilde g}(x,y_0)$ is controlled by $\nabla^2_{\tilde x}\tilde{d}(\tilde{x},\tilde{y}_0)$ if $\tilde{d}(\tilde{x},\tilde{y}_0)$ is the distance function on the space form with constant curvature $-2$ and $ d_{\tilde g}(x,y_0)=\tilde{d}(\tilde{x},\tilde{y}_0)$. {More precisely, for any normalized vector fields $X$ on $\tilde{M}$ perpendicular to $\gamma_{y_0,x}$ and $\tilde{X}$ on the space form with constant curvature $-2$ perpendicular to the geodesic connecting $\tilde{x}$ and $\tilde{y}_0$, we have
\[\nabla_x^2 d_{\tilde g}(x,y_0)(X,X)\leq \nabla^2_{\tilde x}\tilde{d}(\tilde{x},\tilde{y}_0)(\tilde{X},\tilde{X}).\]
} By using the second variation formula again, the Hessian of the distance function $\tilde{d}(\tilde{x},\tilde{y}_0)$ on the space form can be computed directly with
\begin{equation*}
    \tilde{J}_{\tilde{X}}(t)= \frac{\sinh\sqrt{2} t}{\sinh \sqrt{2} \tilde{d}(\tilde{x},\tilde{y}_0)}\tilde{X},\quad \tilde{J}'_{\tilde{X}}( \tilde{d}(\tilde{x},\tilde{y}_0))= \sqrt{2} \,\frac{\cosh\sqrt{2} \tilde{d}(\tilde{x},\tilde{y}_0)}{\sinh \sqrt{2} \tilde{d}(\tilde{x},\tilde{y}_0)}\tilde{X}.
\end{equation*}
Since $\tilde{J}'_{\tilde{X}}( \tilde{d}(\tilde{x},\tilde{y}_0))$ is uniformly bounded for $\tilde{d}(\tilde{x},\tilde{y}_0)>1$, {the Hessian matrix $\nabla_x \omega_0(x)$ is uniformly bounded above. On the other hand, a similar comparison with the space form of curvature $-1/2$ or the Euclidean space shows the Hessian matrix $\nabla_x \omega_0(x)$ is uniformly bounded below. Therefore, we conclude}
$\nabla_x \omega_0(x)=\nabla_x^2 d_{\tilde g}( x, y_0)$ is uniformly bounded when $x\in \supp\,\tilde \psi_k$. By definition, if $\Gamma_{ij}^k$ are the Christoffel symbols, we have  $$\nabla_x^2f(\partial_i, \partial_j)=\frac{\partial^2f}{\partial x_i\partial x_j }-\sum_k\Gamma_{ij}^k \frac{\partial f}{\partial x_k}.$$ Since $\tilde M$ has bounded geometry, in the normal coordinates around $x_k$, {$\left|\frac{\partial\,  d_{\tilde g}( x, y_0)}{\partial x_k}\right|+|\Gamma_{ij}^k|\lesssim 1$} if $|x-x_k|\lesssim 1$.  Thus, the standard coordinate derivative of $\omega_0(x)$ is uniformly bounded. This completes the proof of the lemma.}}
\end{proof}

\begin{lemma}\label{propogation}
Let $C_0$ be defined as in \eqref{n2},
assume that $\supp\,\psi_k \cap S\neq \emptyset$ and that $\omega_k$ is as in \eqref{n6}.
For any $\delta_1>0$, we can fix $C_0$ large enough and choose $\delta_0$ in \eqref{n4} sufficiently small such that for any $x\in \supp \,\tilde\psi_k$, if in the normal coordinate around $x_k$, $\left|\frac{\xi}{|\xi|_{{\tilde{g}(x)}}}-\omega_k\right|>\delta_1/10$, then $(x(t),\xi(t))=\Phi_{-t}(x,\xi)$ satisfies
     \begin{equation}\label{xi0}
     d_{\tilde g}(x(t), y_0)\ge 1, \,\,\,\forall\, t\ge0,
 \end{equation}
 $y_0$ as in \eqref{n2}.
 Moreover, if we choose $\delta_1$ to be sufficiently small, then $\left|\frac{\xi}{|\xi|_{{\tilde{g}(x)}}}-\omega_k\right|<10\delta_1$ implies  
  \begin{equation}\label{xi00}
     d_{\tilde g}(x(t), y_0)\ge 1, \,\,\,\forall\, t\le0.
 \end{equation}
\end{lemma}

\begin{proof}
{{We shall first prove \eqref{xi0} by contradiction, and then give the proof of \eqref{xi00} by using \eqref{xi0}. 
Fix $x\in \supp \,\tilde\psi_k$, let $\omega_0(x)$ be defined as in \eqref{omega}.
Suppose there exists some point $y_1$ such that $d_{\tilde g}(y_1,y_0)<1$ and  $\Phi_{-t_1}(x, \omega_1(x))=(y_1,\eta_1)$ for some unit {covectors $\omega_1(x)$ and $\eta_1$} with $\left|\omega_1(x)-\omega_k\right|>\delta_1/10$ and  $t_1=d_{\tilde g}(x,y_1)$. 
By Lemma~\ref{uniform}, we can choose $\delta_0$ in the definition of $\supp\,\tilde \psi_k$ above sufficiently small such that 
$|\omega_0(x)-\omega_k|=|\omega_0(x)-\omega_0(x_k)|\le \delta_1/20$. This implies that 
$\left|\omega_1(x)-\omega_0(x)\right|>\delta_1/20$.}}

{Since $\supp\,\psi_k \cap S\neq \emptyset$} and  $x\in \supp\,\tilde \psi_k$, if $t_0, t_1$ are defined as above, we have $\frac {C_0}8\le t_0,t_1\le 8c_0\log\la$, and $|t_0-t_1|\le 1$  by the triangle inequality.
  If we denote $( y_2, \eta_2)=\Phi_{{-t_0}}(x,\omega_1(x))$, we have $d_{\tilde g}( y_2, y_0)<d_{\tilde g}( y_2, y_1)+d_{\tilde g}( y_1, y_0)<2$. In normal coordinates around $x_k$, we have $|\omega_0(x)|,|\omega_1(x)|\approx 1$. Since the curvature $K$ of $\tilde M$ is bounded {above by $-1/2$}, by using {the Aleksandrov--Toponogov triangle comparison theorem \cite[Theorem 4.1]{karcher}}, it is not hard to show
\begin{equation}\label{xi3}
   {|\omega_1(x)-\omega_0(x)|\lesssim \frac{\sinh(d_{\tilde g}( y_2,y_0){/\sqrt{2}})}{\sinh(t_1/\sqrt{2})}\lesssim e^{-t_1/\sqrt{2}}\lesssim e^{-\frac{1}{16} C_0}.}
\end{equation}
By choosing $C_0$ sufficiently large, this contradicts  the fact that $\left|\omega_1(x)-\omega_0(x)\right|>\delta_1/20$.

To prove \eqref{xi00}, note that by choosing {$\delta_1$} small enough, {$\left|\frac{\xi}{|\xi|_{{\tilde{g}(x)}}}-\omega_k\right|<10\delta_1$ implies that $\left|\frac{\xi}{|\xi|_{{\tilde{g}(x)}}}+\omega_k\right|>\delta_1/10$}. Thus, \eqref{xi00} follows directly from \eqref{xi0}.
\end{proof}

Now we shall give the proof of \eqref{n13}.  We may assume $\supp\, \tilde{\chi}_0$ has diameter less than 1 so Lemma~\ref{propogation} tells us when $x(t)$  intersects $\supp \,\tilde{\chi}_0$.
 By triangle inequality, it suffices to show 
\begin{equation}\label{n14}
        \| A_{k,1}T_j(\tilde \chi_0 h)\|_{L^q(S)}\lesssim_N \la^{-N}\|h\|_{L^2(\tilde M)}, \, \, C_0\le 2^j\le c_0\log\la.
\end{equation}
Note that the volume of the set $S$ is $O(\la^{Cc_0})$ for some constant $C$. To prove \eqref{n14}, it suffices to show 
the following pointwise bound
\begin{equation}\label{n15}
\left|\int_0^\infty \beta(2^{-j}t)e^{it\la-t\delta}   ( A_{k,1}\circ \cos (t\tilde P))(x, y)\tilde \chi_0( y)\,dt \right|\lesssim_N \la^{-N}.
\end{equation}

We shall give the proof of \eqref{n15} using the Hadamard parametrix, as this approach is more easily adaptable to the proof of Proposition~\ref{keyb}. Alternatively, one could also prove \eqref{n15} using kernel estimates for the spectral measure, as in the proof of Lemma~\ref{kernelhyper}.

Since  $\tilde M$ is simply connected with negative curvature, 
we can
use the Hadamard parametrix to express the solution $\cos t\tilde P$  in normal coordinates around $x_k$
  as follows:
\begin{equation}\label{ak13}
\cos t\tilde P( x,
 y)=
\sum_{\nu=0}^N
w_\nu(x, y)W_\nu(t, x, y)
+R_N(t, x, y),
\end{equation}
where $w_\nu \in C^\infty({\mathbb R}^{n}
\times {\mathbb R}^{n})$,
\begin{equation}\label{k14}
W_0(t, x, y)=(2\pi)^{-n}
\int_{{\mathbb R}^{n}} e^{id_{\tilde g}( x,
  y)\xi_1} \cos t|\xi| \, d\xi,
\end{equation}
while for $\nu=1,2,\dots$, $W_\nu(t, x, y)$
is a finite linear combination of Fourier integrals 
of the form
\begin{equation}\label{k15}
\int_{{\mathbb R}^{n}}
e^{id_{\tilde g}( x, y)\xi_1}
e^{\pm it|\xi|}
\alpha_\nu(|\xi|)\, d\xi,
\, \, \text{with } \, \,
\alpha_\nu(\tau)=0, \, \text{for } \, \tau\le 1
\, \, \text{and } \, \, \partial^j_\tau \alpha_\nu(\tau)
\lesssim \tau^{-\nu-j},
\end{equation}
and, if $N_0$ is given, then if $N$ is large enough,
\begin{equation}\label{k16}
|\partial_t^j R_N(t, x, y)|\le C\exp(Ct),
\quad
0\le j\le N_0,
\end{equation}
for a fixed constant $C$. 
And the coefficients 
$w_\nu( x, y)$ satisfy
\begin{equation}\label{k17}
0<w_0(\tilde x,\tilde y)\le 1,
\end{equation}
as well as
\begin{equation}\label{k19}
|\partial^\beta_x w_\nu( x, y)|
\le C\exp(Cr), \, \,
|\beta|, \nu\le N_0, \quad r=d_{\tilde g}( x,
 y),
\end{equation}
for some uniform constant $C$ (depending on 
$\tilde g$ and $N_0$). We also have the similar bound for the distance function
\begin{equation}\label{k19a}
|\partial^\beta_{x,y} d_{\tilde g}( x,
 y)|
\le C\exp(Cr), \, \,
|\beta|\le N_0, \quad r=d_{\tilde g}( x,
 y).
\end{equation}

The facts that we have just recited are well known.
One can see, for instance,  \cite{Berard} or \cite[\S 1.1, \S 3.6]{SoggeHangzhou} for 
background regarding the Hadamard parametrix,
and \cite{SoggeZelditchL4} for a discussion
of properties of $w_0$.

Let us focus on the $\nu=0$ term. The higher order terms can be treated similarly and satisfy better bounds, and the error term involving $R_N$ certainly satisfies desired bound by using  \eqref{n10},\eqref{k16}, and an integration by parts argument in the $t$ variable. The kernel of the main term in \eqref{20a} is 
\begin{equation}\label{mk}
    \begin{aligned}
        K( x, y)=(2\pi)^{-2n}\la^{2n}&\iiiint \beta(2^{-j}t)e^{it\la-t\delta} e^{i\la\langle  x- z,\xi \rangle} \psi_k(x)(1-a_k(x,\xi))\beta((p(x,\xi))\\ 
        &\,\,\cdot\tilde\psi_k(z)\tilde\chi_0( y)
        \cdot w_0( z, y) e^{i\la d_{\tilde g}( z,
 y)\eta_1} \cos (t\la|\eta|) dzd\xi d\eta dt.
    \end{aligned}
\end{equation}
We can replace $\cos (t\la|\eta|)$ by $e^{-it\la|\eta|}$ since the term involving 
$e^{it\la|\eta|}$ is rapidly decreasing through integration by part in the $t$-variable. A similar integration by parts argument in the $z, \eta$ variables also shows that we may assume ${\eta_1}/{p(x,\xi)}\in [1-\delta_2, 1+\delta_2]$ for some sufficiently small $\delta_2$.

{We claim that} we have for $z\in \supp\,\tilde \psi_k$
and $y\in \supp\,\tilde\chi_0$,
\begin{equation}\label{eq:k9a}
| \nabla_{ z}d_{\tilde g}( z,
 y)-\omega_k|\le {\frac{\delta_1}{10}}.
\end{equation}
 This is because
$
\Phi_{-t}(z,\nabla_{ z}d_{\tilde g}( z,
 y))=( y,\xi_0)
$
for some $\xi_0$ and $t=d_{\tilde g}( z,
 y)$, which by \eqref{xi0} implies $y\notin \supp\,\tilde \chi_0 $ if $\left|{\nabla_{ z}d_{\tilde g}( z,
 y)} -\omega_k\right|> \delta_1/10$. 

 Recall that ${\eta_1}/{p(x,\xi)}\in [1-\delta_2, 1+\delta_2]$. By choosing $\delta_2$ small enough, \eqref{eq:k9a} implies that 
for any $(x, \xi)\in \supp \,a_{k,1}(x,\xi)$, we have 
$$
| \eta_1\nabla_{ z}d_{\tilde g}( z,
 y)-\xi|\ge {\frac{\delta_1}{2}}.$$ Hence a simple 
 integration by parts argument in the $z$
variable yields that 
  the kernel in \eqref{mk} is $O(\la^{-N})$, which completes the proof of \eqref{n13}.

Now we give the proof of  \eqref{n12}. By \eqref{n11} and our previous results for the operators $T_j\tilde \chi_0$ when $2^j\le C_0$ and $2^j\ge c_0\log\la$, proving \eqref{n12} is equivalent to showing that 
\begin{equation}\label{n16}
        \|\sum_{k}A_{k,0}(\Delta_{\tilde g}+(\la+i\delta)^2)^{-1} (\tilde \chi_0 h)\|_{L^q(S)}\lesssim 
    \la^{\mu(q)-1}\|h\|_{L^2(\tilde M)}.
\end{equation}

To prove \eqref{n16}, it suffices to show 
\begin{equation}\label{n17}
    \|\sum_{k}A_{k,0} (\Delta_{\tilde g}+(\la-i\delta)^2)^{-1} (\tilde \chi_0 h)\|_{L^q(S)}\lesssim 
    \la^{\mu(q)-1}\|h\|_{L^2(\tilde M)},
\end{equation}
as well as 
\begin{equation}\label{n18}
    \|\sum_{k}A_{k,0} \big((\Delta_{\tilde g}+(\la+i\delta)^2)^{-1}-(\Delta_{\tilde g}+(\la-i\delta)^2)^{-1} \big) (\tilde \chi_0 h)\|_{L^q(S)}\lesssim 
    \la^{\mu(q)-1}\|h\|_{L^2(\tilde M)}.
\end{equation}

Note that if we define
$E_{\la,k}=\one_{[\la+k\delta, \la +(k+1)\delta)}( \tilde P)$, then
the symbol of the operator
$$
E_{\la,k} \big((\Delta_{\tilde g}+(\la+i\delta)^2)^{-1}-(\Delta_{\tilde g}+(\la-i\delta)^2)^{-1} \big)
$$
is $O\left((\la\delta)^{-1}(1+|k|)^{-2}\right )$. Thus \eqref{n18} can be proved using the same arguments as in the proof of \eqref{tjblarge}.

To prove \eqref{n17}, note that by taking the complex conjugate of both side of \eqref{resolvent}, we have 
\begin{equation}
    \left(\Delta_{\tilde g}+(\la-i\delta)^2\right)^{-1}=\frac i{(\la-i\delta)}\int_0^\infty e^{-it\la-t\delta}\cos (t\tilde P)\,dt.
\end{equation}
As in \eqref{tj}, if we define 
\begin{equation}\label{tj1}
\bar T_j f=\frac i{(\la-i\delta)}\int_0^\infty \beta(2^{-j}t)e^{-it\la-t\delta}\cos (t\tilde P)f\,dt,
\end{equation}
then the above arguments implies that the analog of \eqref{n17}, involving the operators $\bar T_j\tilde \chi_0$ for $2^j\le C_0$ and $2^j\ge c_0\log\la$, satisfies the desired bound. Thus, it suffices to show 
\begin{equation}\label{n19}
        \|\sum_{\{j:C_0\le 2^j\le c_0\log\la\}}\sum_{k}A_{k,0} \bar T_j(\tilde \chi_0 h)\|_{L^q(S)}\lesssim_N 
    \la^{-N}\|h\|_{L^2(\tilde M)}.
\end{equation}
By applying the Hadamard parametrix and using \eqref{xi00}, the proof of \eqref{n19}  follows similarly to that of \eqref{n13}. Hence, we omit the details here.
\end{proof}

\begin{proof}[Proof of Proposition~\ref{keyb}]
As in the proof of Proposition~\ref{keya}, it suffices to show the following equivalent version of \eqref{3.15a},
\begin{equation}\label{main}
    \left\|\chi_1\left(\Delta_g+(\la+i(\log\la)^{-1})^2\right)^{-1}\tilde\chi_1\right\|_{L^2\to L^q}\lesssim \la^{\mu(q)-1}.
\end{equation}
And if we fix $\beta \in C^\infty_0((1/4,4))$ with $\beta= 1$ in (1/2,2), it suffices to show 
\begin{equation}\label{p2.5}
    \left\|\chi_1\left(\Delta_g+(\la+i(\log\la)^{-1})^2\right)^{-1}\beta(P/\la)\tilde\chi_1\right\|_{L^2\to L^q}\lesssim \la^{\mu(q)-1},
\end{equation}
since the analogous estimate involving $(I-\beta(P/\la))$
follows from Sobolev estimates and does not require the $\chi_1$ and $\tilde\chi_1$ cutoff functions.

 We shall need the following lemma which is analogous to Lemma~\ref{propogation}:
\begin{lemma}
    There exist zero-order pseudodifferential operators $A_\pm$ {with compactly supported Schwartz kernel} such that 
{
\begin{equation}\label{eq:lem2.8}
\beta(P/\la)\tilde\chi_1=A_++ A_-+R,
\end{equation} 
}with
$ \|R\|_{L^2\to L^2}\lesssim\la^{-1}.
$
{In local coordinates, $A_{\pm}$ is of the form
\begin{equation}
    A_{\pm} u(x) = (2\pi)^{-n}\la^n\int\int e^{i\la\langle x-y,\xi \rangle} A_{\pm}(x,y,\xi)u(y) dyd\xi,\quad A_{\pm}(x,y,\xi) \in C_0^{\infty}(T^*M).
\end{equation}
}
For all $(x,y,\xi)\in \supp\, A_+(x, y,\xi)$, if $(x(t),\xi(t))=\Phi_t(x,\xi)$, we have 
\begin{equation}\label{6a}
   \dist( x(t), \supp\,\chi_1)\ge 1  \,\,\,\text{for}\,\,\,t \ge C,
\end{equation}
for some sufficiently large constant C. Similarly, for all $(x, y,\xi)\in \supp\, A_-(x,y,\xi)$,  we have 
\begin{equation}\label{7a}
   \dist( x(t), \supp\,\chi_1)\ge 1 \,\,\,\text{as}\,\,\,t\le -C.
\end{equation}
\end{lemma}
\begin{proof}
{{
As in \eqref{n5'}, if we extend $\beta$ to be an even function, then we can write $\beta(P/\la)=B+C$ where $ \|C\|_{L^2\to L^2}\lesssim_N \la^{-N}$, and $B$ is a pseudodifferential operator with principal symbol $\beta(p(x,\xi))$,
with $p(x,\xi)$ here now being the principal symbol of $P$. 

Next, choose $\psi\in C_0^\infty{(M)}$ with $\psi=1$ in a neighborhood of the support of $\tilde\chi_1$  and $\psi=0$ on $M_{tr}$. Without loss of generality, we may assume both $\psi$ and $\tilde \chi_1$ are supported in a sufficiently small neighborhood of some fixed point $y_0$. Then, in normal coordinates around $y_0$, if $B(x,y)$ is the {Schwartz} kernel of $B$, 
we have $B(x,y)\tilde\chi_1(y)=\psi(x)B(x,y)\tilde \chi_1(y)+O(\la^{-N})$. {Since $B$ has principal symbol $\beta(p(x,\xi))$},
\begin{equation}\label{2.95a}
    B(x,y)\tilde\chi_1(y)=(2\pi)^{-n}\la^n\int e^{i\la\langle x-y,\xi \rangle} \psi(x){\beta(p(x,\xi))}\tilde\chi_1(y)d\xi +R(x,y),
\end{equation}
where $R$ is a lower order {pseudodifferential} operator which satisfies $ \|R\|_{L^2\to L^2}=O(\la^{-1})
$.
}}
{Let $S=\{(x,\xi)\in S^*M: x\in \supp\,\psi \}$. 
Since $\psi(x)=0$ on $M_{tr}$ and $\Gamma_+\cap \Gamma_- \subset M_{tr}$, the two sets $\Gamma_+\cap S$ and $\Gamma_-\cap S$ are disjoint. Since $\Gamma_\pm$ are closed and $S$ is compact, 
there exists $\phi_{\pm}\in C_0^{\infty}(S^*M)$ subordinate to the open cover $S\subset (U\setminus \Gamma_-)\cup (U\setminus \Gamma_+)$ where $U$ is a small neighbourhood of $S$, such that
\begin{equation}
    \phi_+(x,\xi) +\phi_-(x,\xi) =1,\quad (x,\xi)\in S,
\end{equation}
with $\supp\,\phi_+ \cap \Gamma_-=\emptyset$ 
 and $\supp\,\phi_- \cap \Gamma_+=\emptyset$. If we define the operators $A_\pm$ by 
\begin{equation}\label{10}
   A_\pm f(x)=(2\pi)^{-n}\la^n\int e^{i\la\langle x-y,\xi \rangle} \phi_{\pm}(x,\xi/|\xi|_g)\psi(x)\beta(p(x,\xi))\tilde\chi_1(y) f(y)d\xi dy,
\end{equation}
then we obtain \eqref{eq:lem2.8}. Moreover, by \eqref{eq:convex}, $(x,\xi)\in \supp \, \phi_{\pm}$ satisfies \eqref{6a} and \eqref{7a} for sufficiently large $C$, respectively.
}
\end{proof}
For later use, it is straightforward to check that the $A_\pm$ operators  satisfy 
\begin{equation}\label{bbound}
    \|A_\pm\|_{L^p\to L^p}=O(1),\,\,\,\forall\,\,\,1\le p\le \infty.
\end{equation}

To prove \eqref{p2.5} it suffices to show 
\begin{equation}\label{01}
    \left\|\chi_1\left(\Delta_g+(\la+i(\log\la)^{-1})^2\right)^{-1}\circ A\right\|_{L^2\to L^q}\lesssim \la^{\mu(q)-1}, \quad A=A_+, A_-.
\end{equation}
{{
The other term involving $R$ is more straightforward to handle and does not require the $\chi_1$ cut-off function on the left. More explicitly, by \eqref{3.15} along with the fact that $\|R\|_{L^2\to L^2}\lesssim\la^{-1}$, we have
\begin{equation}\label{n12b}
\begin{aligned}
       \left\|\left(\Delta_g+(\la+i(\log\la)^{-1})^2\right)^{-1} R \,(h)\right\|_{ L^q(M)}&\lesssim \la^{\mu(q)-1}(\log\la)^{\frac12}\|R( h)\|_{L^2( M)}\\
       &\lesssim 
   \la^{\mu(q)-2}(\log\la)^{\frac12} \| h\|_{L^2( M)},
\end{aligned}
\end{equation}
which is better than the required estimate in \eqref{p2.5}.}}

Let us first prove \eqref{01} for $A=A_+$.  The main strategy in the proof
is similar to what was used in the  proof of Proposition~\ref{keya}. If we define $T_j$ as in \eqref{tj}, it is natural to separately consider
the contribution of the terms with
$2^j\lesssim  1$, $1\lesssim  2^j\lesssim \log \la$ and $\log\la \lesssim 2^j$.

\noindent(i) $2^j\le 10C$ for $C$ as in \eqref{7a}.

One can directly apply the arguments from case (i) in the proof of Proposition~\ref{keya} to handle this case. There is no need to make use of $\chi_1$ and $A$ operator here.

\noindent(ii) $2^j\ge c_0\log\la$ for some small enough $c_0$.

Let
$E_{\la,k}=\one_{[\la+k/\log\la, \la +(k+1)/\log\la)}( P).$ 
Then by integration by parts in $t$-variable, the symbol of 
$$S_k=\frac1{i(\la+i(\log\la)^{-1})}E_{\la,k}\int_0^\infty e^{it\la-t/\log\la}\cos (tP)\,\sum_{2^j\ge c_0\log\la}\beta(2^{-j}t)\,dt
$$
is $O\left(\la^{-1}\log\la (1+|k|)^{-N}\right )$. Thus, by the sharp spectral projection bound in \eqref{i.7} and \eqref{3.19} with $\delta=(\log\la)^{-1}$
\begin{align*}
\|&\one_{[\la/2,2\la]}( P) \, 
 \sum_{|k|\lesssim \la\log\la}(
 S_k\circ A h)\|_{L^q( M)}
\\
&\le \sum_{|k|\lesssim \la\log\la}
\|\one_{[\la/2,2\la]}( P) \,(S_k\circ A h)\|_{L^q( M)}
\\
&\le  \la^{\mu(q)}(\log\la)^{-1/2}  \sum_{|k|\lesssim \la\log\la}
\| \one_{[\la/2,2\la]}( P) \, (S_k\circ A h)\|_{L^2( M)}
\\
&\lesssim \la^{\mu(q)}(\log\la)^{-1/2}\sum_{|k|\lesssim \la\log\la} (1+|k|)^{-N} \la^{-1}\log\la  \|\one_{[\la/2,2\la]}( P) E_{\la,k} \circ (A h)\|_{L^2( M)}
\\
&\lesssim \la^{\mu(q)-1}  \|h\|_{L^2( M)},
\end{align*}
 using \eqref{3.19}, \eqref{10} and \eqref{bbound} in the last step.
The case when the spectrum is outside {$[\la/2,2\la]$} can be handled using Sobolev estimates and satisfies better bounds.

\noindent(iii) $10C \le 2^j\le c_0\log\la$ for $C$ as in \eqref{7a}.

This is the case where we require compact cutoff functions on both sides. 
By duality, it suffices to show that
 the operator
\begin{equation}\label{20}
  T= \sum_{\{j: 10C\le 2^j\le c_0\log\la\}}\frac i{(\la-i(\log\la)^{-1})}\int_0^\infty \beta(2^{-j}t)e^{-it\la-t/\log\la}   A\circ \cos (tP)\chi_1\,dt
\end{equation}
satisfies the same $L^{q'}\to L^2$ mapping bound as in  \eqref{01}.

To proceed, since $(M,g)$ has nonpositive sectional curvature, we can use the 
Cartan-Hadamard theorem to lift the calculation up to the universal
cover of $(M,g)$ using the formula
(see e.g., \cite[(3.6.4)]{SoggeHangzhou})
\begin{equation}\label{k.3}
    (\cos t\sqrt{-\Delta_{g}})(x,y)
=\sum_{\alpha\in \Gamma}
(\cos t\sqrt{-\Delta_{\tilde g}})(\tilde x,\alpha(\tilde y)).
\end{equation}
Here $(\Rn,\tilde g)$ is the universal cover of 
$(M,g)$, with $\tilde g$ now being the Riemannian metric
on $\Rn$ obtained by pulling back the metric $g$
via the covering map.   Also, $\Gamma: \Rn\to \Rn$ are the 
deck transformations, and $\tilde x, \tilde y \in D$ with $D\simeq M$ being a Dirichlet fundamental domain.

Note that by finite propagation speed,  $\cos (t\sqrt{-\Delta_{\tilde g}})(\tilde x,\alpha(\tilde y))=0$ if $d_g(\tilde x,\alpha(\tilde y))>|t|$.
Thus, for each fixed $\tilde x$, by using a simple volume counting argument using the fact that the injectivity radius is positive along with the bounded geometry of $(M,g)$,  the number of deck transformations $\alpha$ such that $d_g(\tilde x,\alpha(\tilde y))\lesssim c_0\log\la$ is $O(\la^{C_Mc_0})$. The proof of this, along with additional properties of manifolds with bounded geometry, will be provided in the next section (see \eqref{b.0}–\eqref{b.1} and the discussion following \eqref{k.7}).

Therefore, it suffices to show that for each fixed $\alpha$, we have
\begin{equation}\label{20a}
\int_0^\infty \beta(2^{-j}t)e^{-it\la-t/\log\la}   (A\circ \cos (t\sqrt{-\Delta_{\tilde g}}))(\tilde x,\alpha(\tilde y))\chi_1(\tilde y)\,dt \lesssim_N \la^{-N}.
\end{equation}
Here, we slightly abuse the notation by identifying $\chi_1$ with a compactly supported function on the fundamental domain $D$, and $A$ here denotes the lift of the operator on $(M,g)$
to $(\Rn,\tilde g)$ via the covering map. By applying the Hadamard parametrix and using \eqref{6a}, \eqref{20a} follows from the same arguments as in the proof of \eqref{n15}. We omit the details here.

This finishes the proof of \eqref{01} if $A=A_+$. 

Similarly, if  $A=A_-$, we can use the arguments in the proof of Proposition~\ref{keya} to show that 
\begin{equation}\label{01a}
    \left\|\chi_1\left(\Delta_g+(\la-i(\log\la)^{-1})^2\right)^{-1}\circ A\right\|_{L^2\to L^q}\lesssim \la^{\mu(q)-1},
\end{equation}
as well as 
\begin{multline}\label{b-}
\left\|\chi_1\left(\left(\Delta_g+(\la+i(\log\la)^{-1})^2\right)^{-1}-\left(\Delta_g+(\la-i(\log\la)^{-1})^2\right)^{-1}\right)\circ A\right\|_{L^2\to L^q}
\\
\lesssim \la^{\mu(q)-1}.
\end{multline}
These two inequalities yield \eqref{01} with $A=A_-$.

By repeating the above arguments, \eqref{01a} is a consequence of 
\begin{equation}\label{20b}
\int_0^\infty \beta(2^{-j}t)e^{it\la-t/\log\la}   (A\circ \cos (t\sqrt{-\Delta_{\tilde g}}))(\tilde x,\alpha(\tilde y))\chi_1(\tilde y)\,dt \lesssim_N \la^{-N}.
\end{equation}
This follows from the same arguments as in the proof of \eqref{n15},  utilizing Hadamard parametrix and \eqref{7a}.

On the other hand,
if we define
$E_{\la,k}=\one_{[\la+k/\log\la, \la +(k+1)/\log\la)}( P)$, then
the symbol of the operator
$$
E_{\la,k}\left(\left(\Delta_g+(\la+i(\log\la)^{-1})^2\right)^{-1}-\left(\Delta_g+(\la-i(\log\la)^{-1})^2\right)^{-1}\right)
$$
is $O\left(\la^{-1}\log\la (1+|k|)^{-2}\right )$. Thus, \eqref{b-} can be proved using the same arguments as in case (ii) of the proof of Proposition~\ref{keyb}.
\end{proof}

%

\newsection{Manifolds of bounded geometry}

As we mentioned before, using the assumption of bounded geometry, we shall be able to  modify the arguments that were
used in \cite{BHSsp}, \cite{SBLOg}, \cite{HSp} and \cite{HSst} to obtain \eqref{i.2}, \eqref{i.6} and \eqref{i.7} in the special
case where $(M,g)$ was a {\em compact} Riemannian manifold.  The basic facts that will allow us to carry out the
local harmonic analysis for general manifolds of bounded geometry can be found, for instance, in 
Chapter 2 of Eldering~\cite{normhyp}.  In addition to extending the local harmonic analysis that was used in these
earlier works, we shall also need to show that the global kernel estimates in \cite{BSTop} and the aforementioned earlier
works hold for manifolds of bounded geometry and nonpositive curvature.  As we shall see in the end of the section,
like in the earlier works, we can do this by lifting the calculations up to the universal cover and exploiting the fact that
$\Inj(M)>0$ if $M$ is of bounded geometry.  After possibly multiplying the metric by a constant, we may also assume
that $\Inj(M)> 1$, as we shall do throughout this section.

Let us quickly review facts that we shall require for our arguments.  First, there is a $\delta=\delta(M)>0$ so that the coordinate
charts given by the exponential map are defined on all geodesic balls $B(x,2\delta)$, $x\in M$, of radius $2\delta>0$.
Furthermore, in the resulting normal coordinates, the Riemannian distance, $d_g(x_1,x_2)$, is comparable to
$|\exp^{-1}_x(x_1)-\exp^{-1}_x(x_2)|$, independent of $x\in M$.  
Additionally, derivatives of the transition maps from these coordinates are also uniformly bounded.
(See Proposition 2.5 and Lemma 2.6 in  \cite{normhyp}.)

Furthermore, there is a uniformly locally finite cover by geodesic balls.  
By this we mean that there is a $\delta(M)>0$ so that whenever $\delta\in (0,\delta(M)]$ there is a countable
covering by geodesic balls
\begin{equation}\label{b.0}
M=\bigcup_{j}B(x_j,\delta) \quad \text{with } 
d_g(x_j,x_k)\ge \delta \, \, \text{if } \, j\ne k.
\end{equation}
Furthermore, assuming that $\delta(M)$ is small enough, for $\delta$ as above, we can
assume that the covering also satisfies
\begin{equation}\label{b.00}
\text{Card} \bigl\{j: \, B(x_j,2\delta)\cap B(x,2\delta)\ne \emptyset \bigr\}\le C_0, \, \, 
\, \forall \, \, x\in M,
\end{equation}
 for a uniform constant $C_0=C_0(M)<\infty$. 
(See \cite[Lemma 2.16]{Borthwick}.)

From this one also sees that we can
also choose $\delta=\delta(M)>0$ small enough so that there is a $C^\infty$ partition of unity $\{\psi_i\}$,
\begin{equation}\label{b.1}
1=\sum_j \psi_j(x), \quad \supp \psi_j\subset B(x_j,\delta),
\end{equation}
with uniform control of each derivative all of the $\{\psi_j\}$ in the normal coordinates described above.
(See Lemma 2.17 and 
Definition 2.9 in \cite{normhyp}.)

Using our assumption of bounded geometry, as we shall describe shortly, we can also construct a microlocal partition of unity involving 
pseudodifferential operators supported in the $\delta$-balls in the above covering.  It will be convenient, as in
the compact manifold case, to use such microlocal cutoffs for the local harmonic analysis that we shall
require in the proof of Theorems~\ref{bgst} and \ref{bgsp}.  These operators will, in effect, give us a
second microlocalization needed to apply bilinear harmonic analysis.

\noindent{\bf 3.1. Log-scale spectral projection estimates}

Let us first show how we can adapt the proofs for the compact manifold case treated
in \cite{HSp} to prove Theorem~\ref{bgsp} since the second microlocalization and
the resulting arguments is a bit more straightforward than what is needed for the
Strichartz estimates in Theorem~\ref{bgst}.

Before describing these pseudodifferential cutoffs, let us introduce another local operator which we shall require.  To do so, let us
fix, following \cite{HSp},  $\rho\in {\mathcal S}(\R)$ satisfying
\begin{multline}\label{b.3}
\rho(0)=1, \quad \Hat \rho(t)=0, \, \, t\notin \delta_1\cdot [1-\delta_2,1+\delta_2]=[\delta_1-\delta_1\delta_2, \, 
\delta_1+\delta_1\delta_2], 
\\
\text{with } \, 0
<\delta_1, \delta_2 <\tfrac12 \min(\Inj(M),1)
\end{multline}
 as above, with $\delta_1,\delta_2$
  to be specified later when in order to apply 
bilinear oscillatory integral results from \cite{LeeBilinear} and \cite{TaoVargasVega}, just as was done
in \cite{HSp}.    We then define the ``local'' operators
\begin{equation}\label{b.4}
\sigma_\la =\rho(\la-P)+\rho(\la+P).
\end{equation}
We call these local since their kernels satisfy
\begin{equation}\label{b.5}
\sigma_\la(x,y)=0 \quad \text{if } \, \, d_g(x,y)>r, \, \, 
r= \delta_1(1+\delta_2)<\delta/2.
\end{equation}
This follows from the fact that, since $\Hat \rho$ has support as in \eqref{b.3} we have, by Euler's formula,
\begin{equation}\label{b.5'}
\sigma_\la =\pi^{-1}\int_0^{\delta_1(1+\delta_2)} \Hat \rho(t) \, e^{i\la t}\,  \cos(tP)\, dt.
\end{equation}
Finally, by finite propagation speed, $(\cos (tP))(x,y)=0$ if $d_g(x,y)>|t|$, which, along with the preceding
identity, yields \eqref{b.5}.

We also note that
by \eqref{ii.3} and orthogonality we have
\begin{equation}\label{b.6}
\| \sigma_\la\|_{L^2\to L^q}=O(\la^{\mu(q)}),
\end{equation}
with $\mu(q)$ as in \eqref{ii.2}.
We shall also consider the ``global'' smoothed out spectral projection operators
\begin{equation}\label{b.7}
\rho_\la =\rho(T(\la-P))
, \quad T=c_0\log\la,
\end{equation}
where $c_0>0$ shall be fixed later.  We then conclude that, in order to prove
Theorem~\ref{bgsp} it suffices to show that if all of the sectional curvatures
of $M$ are nonpositive then we have
\begin{equation}\label{b.8.0}
\| \rho_\la \|_{L^2\to L^q}\lesssim
\begin{cases} \la^{\mu(q)}(\log\la)^{-1/2}, \quad \text{if } \, \, q>q_c
\\
\bigl(\la(\log\la)^{-1}\bigr)^{\mu(q)}, \quad \text{if } \, \, q\in (2,q_c],
\end{cases}
\end{equation}
while, if all the sectional curvatures are all pinched below $-\kappa^2_0$ with $\kappa_0>0$, we have the stronger estimates
\begin{equation}\label{b.9.0}
\| \rho_\la \|_{L^2\to L^q}\le C_q \, \la^{\mu(q)}(\log\la)^{-1/2}, \, \, q\in (2,\infty].
\end{equation}
For later use, we note that
\begin{equation}\label{b.10}
\|(I-\sigma_\lambda)\circ \rho_\lambda\|_{L^2\to L^q}
\le C \lambda^{\mu(q)} (\log \lambda)^{-1}, \qquad q\in (2,\infty].
\end{equation}
Indeed, the symbol of the operator satisfies
$$
\bigl|((I-\sigma_\lambda)\circ \rho_\lambda ) (\tau)\bigr|
\lesssim_N \frac1T (1+|\tau-\lambda|)^{-N}.
$$
Therefore, \eqref{b.10} follows from \eqref{ii.3} together with a simple orthogonality argument.

As in earlier works, the task of obtaining the bounds in \eqref{b.8.0} and \eqref{b.9.0}
naturally splits into three cases: $q=q_c$, $q\in (q_c,\infty]$ and $q\in (2,q_c)$.  Handling
the critical exponent is the most difficult.  So, we shall first prove the special case
of these to bounds for this exponent:
\begin{multline}\label{b.8.1}
\|\rho_\la\|_{L^2\to L^q_c}\lesssim \bigl(\la(\log\la)^{-1}\bigr)^{\mu(q_c)},
\\
\text{if all of the sectional curvatures of } \, M \, \, \text{are nonpositive},
\end{multline}
and the stronger results
\begin{multline}\label{b.9.1}
\|\rho_\la\|_{L^2\to L^q_c}\lesssim \la^{\mu(q_c)}(\log\la)^{-1/2},
\\
\text{if all of the sectional curvatures of } \, M \, \, \text{are }  \le -\kappa_0^2, \, \,
\, \text{some } \, \kappa_0>0.
\end{multline}


Let us now describe the microlocal operators which will  be utilized to give us our
very useful second microlocalization.
 We shall use the fact that for each fixed $j$, we can write
\begin{equation}\label{b.12}
\psi_j(x) (\sigma_\la h)(x)=\sum_{\ell=1}^K \bigl(A_{j,\ell} \circ\sigma_\la\bigr)(h)(x)+R_jh(x),
\end{equation}
where the  $A_{j,l}$ are pseudodifferential operators in a bounded subset of $S^0_{1,0}$
and the kernels of the above operators satisfy
\begin{equation}\label{b.13}
A_{j,\ell}(x,y), \, \, R_j(x,y)=0, \quad \text{if } \, x\notin B(x_j,\delta) \, \,
\text{or } \, \, y\notin B(x_j,3\delta/2).
\end{equation}
Furthermore, in the normal coordinate system about $x_j$ described above 
we may assume that
\begin{equation}\label{b.14}
A_{j,\ell}(x,y)=(2\pi)^{-n} \int_{\Rn}e^{i\langle x-y,\xi\rangle} a^\la_{j,\ell}(x,y,\xi) \, d\xi
\end{equation}
 with $a^\la_{j,\ell}=0$ 
if $\xi/|\xi|$ is outside of a $O(K^{-1/(n-1)})$ neighborhood of some
$\xi_{j,\ell}\in \mathbb{S}^{n-1}$ and
$a^\la_{j,\ell}=0$ if $|\xi|\notin [c_1\la, \la/c_1]$, with
$c_1\in (0,1)$ independent of $j$.
  So, if $K$ is large enough, we may assume that
\begin{equation}\label{aas}
a^\la_{j,\ell}(x,y,\xi)=0 \, \, 
\text{when } \, x\notin B(x_j,\delta), \,
y \notin B(x_j,2\delta), \, \, \text{or } \, 
\bigl|\tfrac{\xi_{j,\ell}}{|\xi_{j,\ell|}}-\tfrac{\xi}{|\xi|}\bigr|\ge \delta,
\end{equation}
 and also that this symbol satisfies the natural size estimates corresponding to
 these support properties
$$\partial^{\alpha_1}_{x,y}\partial^{\alpha_2}_\xi a^\la_{j,\ell} =O_\delta(\la^{-|\alpha_2|}).$$
In addition to \eqref{b.13}, we may assume by fixing $c_1>0$ small enough that we have the uniform
bounds
\begin{equation}\label{b.15}
R_j(x,y)=O(\la^{-N}), \quad N=1,2,\dots .
\end{equation}
The above dyadic pseudodifferential operators satisfy
\begin{equation}\label{b.16}
\| A_{j,\ell}\|_{L^p(M)\to L^p(M)}=O(1), \quad 1\le p\le \infty.
\end{equation}

One constructs the above microlocal cutoffs $A_{j,\ell}$ using standard arguments
from the theory of pseudodifferential operators.  The resulting symbols can have
dyadic support $|\xi|\approx \la$, just as in the case of compact manifolds treated
in \cite{HSp}, since the left side of \eqref{b.12} involves $\sigma_\la = \sigma(\la-P)$.
Furthermore, using the fact that we are assuming that $M$ is of bounded geometry,
by the above discussion, the implicit constants in the above description of this
second microlocalization
 can be chosen to be independent 
of $j$ if $\delta\le \delta(M)$  and $K$ in \eqref{b.12}
are fixed.

We also note that, due to our assumptions,
we can assume that the symbols $a^\la_{j,\ell}$ vanish
outside of a small conic neighborhood of $(x_j,\xi_{j,\ell})$ 
by choosing $\delta\le 
\delta(M)$ to be small and $K$ to be large. As in the 
compact manifold case, this will be useful when we
need to use our local harmonic analysis.

One consequence of this, \eqref{b.00}, \eqref{b.13} and \eqref{b.16} is that if we fix
$\ell_0\in \{1,\dots,K\}$, then
these dyadic operators satisfy
\begin{equation}\label{b.17}
A=A_{\ell_0} = \sum_{j}A_{j,\ell_0}\in S^0_{1,0}
\quad \text{and } \, 
\|A\|_{L^p(M)\to L^p(M)}=O(1), \, \, 1\le p\le \infty.
\end{equation}
These microlocal cutoffs will play the role of the ``$B$'' operators 
that were used for compact manifolds in 
in \cite{BHSsp}, \cite{SBLOg} and \cite{HSp}.
Note also that by \eqref{b.00} and \eqref{b.13} we also have
\begin{equation}\label{b.18}
\|Rh\|_{L^q(M)} \le C_{p,q,M} \, \|h\|_{L^p(M)}, \, \, 1\le p\le q\le \infty, \, \,
\text{if }, \, \, R=\sum_j R_j.
\end{equation}

Note that, in view of \eqref{b.1}, \eqref{b.10}, \eqref{b.12}, \eqref{b.17} and \eqref{b.18}, in order
to prove \eqref{b.8.1} and \eqref{b.9.1}, it suffices to
prove that if all the sectional curvatures of
$M$ are nonpositive
\begin{equation}\label{b.20}
\| A \sigma_\la \rho_\la \|_{L^2\to L^q_c}\lesssim
\bigl(\la(\log\la)^{-1}\bigr)^{\mu(q_c)}, 
\end{equation}
while, if all the sectional curvatures are all pinched below $-\kappa^2_0$ with $\kappa_0>0$, we have
\begin{equation}\label{b.21}
\| A\sigma_\la \rho_\la \|_{L^2\to L^q_c}\lesssim \, \la^{\mu(q_c)}(\log\la)^{-1/2}.
\end{equation}

Using \eqref{b.00} and \eqref{b.13}, we have for $q\ge 2$
$$|A\sigma_\la \rho_\la f(x)|\le C\|A_{j,\ell_0}(\sigma_\la \rho_\la f)(x)\|_{\ell^q_j}.$$
Consequently, if we consider the vector-valued operators
\begin{equation}\label{b.211}
\A h=(A_{1,\ell_0}h, \, A_{2,\ell_0}h,\dots)
\end{equation}
we have
\begin{equation}\label{b.22}
\|A\sigma_\la \rho_\la f\|_{L^q(M)}\lesssim
\| \A(\sigma_\la \rho_\la f)\|_{L^q_x \ell^q_j(M\times \N)}, \quad q\in [2,\infty].
\end{equation}
Note for later use that by \eqref{b.00}, \eqref{b.16} and \eqref{b.17} we also have 
\begin{equation}\label{b.23}
\| \A h\|_{L^p_x\ell^p_j}\le C\|h\|_{L^p(M)}, \quad 1\le p\le \infty.
\end{equation}

In view of \eqref{b.22}, in order to prove \eqref{b.20} and \eqref{b.21}, it suffices to show that when
all of the sectional curvatures of $M$ are nonpositive
we have
\begin{equation}\label{b.24}
\| \A \sigma_\la \rho_\la f\|_{L^{q_c}_x\ell^{q_c}_j} \lesssim 
\bigl(\la(\log\la)^{-1}\bigr)^{\mu(q_c)} \|f\|_{L^2(M)},
\end{equation}
while, if all the sectional curvatures are all pinched below $-\kappa^2_0$ with $\kappa_0>0$, we have
\begin{equation}\label{b.25}
\| \A\sigma_\la \rho_\la f \|_{L^{q_c}_x\ell^{q_c}_j}  \lesssim  \la^{\mu(q_c)}(\log\la)^{-1/2} \|f\|_{L^2(M)}.
\end{equation}

The operators $\A \sigma_\la \rho_\la$ play the role of the $\tilde \rho_\la$ operators in \cite{SBLOg} and \cite{HSp}.
We are introducing this vector-valued approach to easily allow us to only have to carry out the local bilinear harmonic analysis
in individual coordinate patches coming from the geodesic normal coordinates in the balls $B(x_j,2\delta)$ mentioned
before.  In the compact case treated by two of us and coauthors, this was not necessary since $M$ could be covered
by finitely many balls of sufficiently small radius on which the bilinear analysis could be carried out.

In proving these two estimates we may, of course, assume,
as we shall throughout this section, that
\begin{equation}\label{b.norm}
\|f\|_{L^2(M)}=1.
\end{equation}
Then, similar to the compact manifold case, let us
define vector-valued sets
\begin{equation}\label{b.26}
\begin{split}
A_+&= \{(x,j): \, |(\A \sigma_\la \rho_\la f)
(x,j)|\ge \la^{\frac{n-1}4+\frac18}\}
\\
A_-&=\{(x,j): \, |(\A \sigma_\la \rho_\la f)
(x,j)|< \la^{\frac{n-1}4+\frac18}\}.
\end{split}
\end{equation}
Recall here that
\begin{equation}\label{b.27}
(\A \sigma_\la\rho_\la f)(x,j)=A_{j,\ell_0}
\sigma_\la \rho_\la f(x).
\end{equation}

In order to prove \eqref{b.24} and \eqref{b.25},
it suffices to show that
we have the following two results.  First, for
all complete manifolds of bounded 
geometry and  nonpositive 
sectional curvatures, we have for $\la \gg 1$
the large height estimates
\begin{multline}\label{b.31}
\|\A \sigma_\la\rho_\la f\|_{L^{q_c}_x
\ell_j^{q_c}(A_+)}
\lesssim \la^{\mu(q_c)}T^{-1/2}, 
\\ \text{if } T=c_0\log\la,
\, \, 
\text{with }c_0=c_0(M)>0 \, 
\, \text{sufficiently small.}
\end{multline}
The remaining estimate, for small heights, 
which would yield the above
desired bounds for $q=q_c$ then would be the following
for $T$ as above
\begin{multline}\label{b.32}
\|\A \sigma_\la\rho_\la f\|_{L^{q_c}_x
\ell_j^{q_c}(A_-)}
\\
\lesssim
\begin{cases}
\bigl(\la T^{-1}\bigr)^{\mu(q_c)}, 
\, \, \text{if all the sectional curvatures of } 
M \, \text{are nonpositive}
\\
\la^{\mu(q_c)}T^{-1/2}, 
\, \, \text{if all the sectional curvatures of } 
M \, \text{are } \le -\kappa_0^2, \,
\text{some }\kappa_0>0,
\end{cases}
\end{multline}
with $T$ in \eqref{b.7} as in the preceding inequality.

In order to prove \eqref{b.31} and also the estimates
in Theorem~\ref{bgsp} for $q>q_c$ we shall require
the following lemma.

\begin{lemma}\label{Gl}
Let $\Psi=|\rho|^2$ and fix 
$$a\in C^\infty_0((-1,1)) \, \, 
\text{satisfying } \, \, a(t)=1, \, \, 
|t|\le 1/2.$$
Then, if $G_\la=G_{\la,T}$ is defined by
\begin{equation}\label{b.33}
G_\la=G_\la(P)
=\frac1{2\pi}\int_{-\infty}^\infty
(1-a(t))
\, T^{-1}\Hat \Psi(t/T) \,
e^{it\la} \, e^{-itP} \, dt, 
\end{equation}
we have for $c_0=c_0(M)>0$ sufficiently small
and $\la\gg 1$
\begin{equation}\label{b.34}
\|G_\la\|_{L^1(M)\to L^\infty(M)}= 
O(\la^{\frac{n-1}2}\exp(C_MT)),
\, \, 1\le T\le c_0\log\la,
\end{equation}
assuming that $M$ is of bounded geometry and that
all of its sectional curvatures are nonpositive.
\end{lemma}

Note that if $L_\la=L_{\la,T}$ is given by
\begin{equation}\label{b.33.2}
L_\la =L_\la(P)=\frac1{2\pi}\int_{-\infty}^\infty
a(t) \, T^{-1}\Hat \Psi(t/T) \, e^{it\la}
e^{-itP} \, dt,
\end{equation}
then
\begin{equation}\label{b.33.3}
G_\la+L_\la =\Psi(T(\la-P))=\rho_\la\rho^*_\la.
\end{equation}
Furthermore, it is simple to use \eqref{ii.3}
and a simple orthogonality argument to see
that if $q_c'$ is the dual exponent for $q_c$ then
\begin{equation}\label{b.33.4}
\|L_\la\|_{L^{q_c'}(M)\to L^{q_c}(M)}=O(T^{-1}
\la^{2\mu(q_c)})=O(T^{-1}\la^{2/q_c}).
\end{equation}

We shall postpone the proof of this lemma until the 
end of this section.   Let us now see how we can use it
along with the local estimate \eqref{b.6} to prove
the large height estimates.  As two of us did for compact
manifolds, we shall rely on a variant of an argument of 
Bourgain~\cite{BourgainBesicovitch}, along
with
\eqref{b.34} and \eqref{b.33.4}.

\begin{proof}[Proof of \eqref{b.31}]
This just follows from the proof of (2.18) in \cite{HSp}; however, we shall give the argument here for the sake
of completeness.  We shall be assuming
here that, as in \eqref{b.31}, $T=c_0\log\la$, with
$c_0>0$ to be specified in a moment.

We first note that, by \eqref{b.10}, \eqref{b.211} and \eqref{b.norm}
\begin{equation}\label{b.35}
\|\A \sigma_\la \rho_\la f\|_{L^{q_c}_x\ell^{q_c}_j(A_+)}
\le \|\A \rho_\la f\|_{L^{q_c}_x\ell^{q_c}_j(A_+)} + C\la^{1/q_c}/\log\la,
\end{equation}
since, by \eqref{ii.2},
$$\mu(q_c)=1/q_c.$$
As a result, we would obtain \eqref{b.31} if we could show that
\begin{equation}\label{b.36}
\|\A \rho_\la f\|_{L^{q_c}_x\ell^{q_c}_j(A_+)}
\le C\la^{1/q_c}(\log\la)^{-1/2}+\tfrac12 \|\A\sigma_\la \rho_\la f\|_{L^{q_c}_x\ell^{q_c}_j(A_+)}.
\end{equation}

To prove this, similar to what was done in \cite{HSp},
choose $g=g(x,j)$ vanishing outside $A_+$ so that
\begin{equation}\label{b.41}
\|g\|_{L^{q_c'}_x\ell^{q_c'}_j(A_+)}=1
\, \, 
\text{and } \, 
\|\A \rho_\la f\|_{L^{q_c}_x\ell^{q_c}_j(A_+)}
=\sum_j \int 
(\A \rho_\la f)(x,j) \,
\overline{\1_{A_+}(x,j) \cdot g(x,j)} \, dx.
\end{equation}

Then, similar to (3.4) in \cite{HSp},
using \eqref{b.norm} and \eqref{b.33.3} we find that
\begin{align*}
\|\A \rho_\la f\|^2_{L^{q_c}_x\ell^{q_c}_j(A_+)}
&=\Bigl(\, \int_M
f(x)\cdot \overline{\bigl(\rho_\la^*\A^*(
\1_{A_+}\cdot g)\bigr)(x)} \, dx\, \Bigr)^2
\\
&\le
\int_M \bigl| \bigl(\rho_\la^*\A^*(
\one_{A_+}\cdot g)\bigr)(x) \bigr|^2 \, dx
\\
&=\sum_j
\int_M \bigl(\bigl( \A\circ \Psi(T(\la-P))\circ \A^*
\bigr)(\1\cdot g)\bigr)(x,j) 
\cdot 
\overline{(\1_{A_+}\cdot g)(x,j)}\, dx
\\
&=\sum_j \int_M \bigl(\bigl( \A\circ L_\la\circ \A^*
\bigr)(\1\cdot g)\bigr)(x,j) 
\cdot 
\overline{(\1_{A_+}\cdot g)(x,j)}\, dx
\\
&+\sum_j \int_M \bigl(\bigl( \A\circ G_\la\circ \A^*
\bigr)(\1\cdot g)\bigr)(x,j) 
\cdot 
\overline{(\1_{A_+}\cdot g)(x,j)}\, dx
\\
&\qquad= I+II.
\end{align*}

By \eqref{b.23}, \eqref{b.33.4} and \eqref{b.41} and
H\"older's inequality, we have
\begin{align*}
|I|&\le
\|(\A L_\la \A^*)(\1_{A_+}\cdot g)
\|_{L^{q_c}_x\ell_j^{q_c}}
\cdot \|\1_{A_+}\cdot g
\|_{L^{q_c'}_x\ell^{q_c'}_j}
\\
&\lesssim \|L_\la \A^*(\1_{A_+}\cdot g)
\|_{L^{q_c}_x}\cdot 1
\\
&\lesssim T^{-1}\la^{2/q_c}\|\A^*(\1_{A_+}\cdot
g)\|_{L^{q_c'}_x}
\\
&\lesssim T^{-1}\la^{2/q_c}\|\1_{A_+}\cdot g
\|_{L^{q_c'}_x\ell^{q_c'}_j}=T^{-1}\la^{2/q_c}.
\end{align*}

To estimate $II$, note that, by \eqref{b.23},
$\|\A^*\|_{L^1_x\ell^1_j\to L^1_x}, 
\, \|\A\|_{L^\infty_x\to L^\infty_x\ell_j^\infty}
=O(1)$.
Also, if $c_0>0$ is chosen small enough, then,
by \eqref{b.34} we have
$$\|G_\la\|_{L^1(M)\to L^\infty(M)}=O(\la^{\frac{n-1}2
+\frac18}).$$
 Therefore, we have, by the above argument,
\begin{align*}
|II|&
\le
\|(\A G_\la \A^*)(\1_{A_+}\cdot g)\|_{L^\infty_x
\ell^\infty_j}\|\1_{A_+}\cdot g
\|_{L^1_x\ell_j^1}
\\
&
\lesssim \|( G_\la \A^*)(\1_{A_+}\cdot g)\|_{L^\infty_x}\|\1_{A_+}\cdot g
\|_{L^1_x\ell_j^1}
\\
&\le C
\la^{\frac{n-1}2+\frac18}\|\A^*(\1_{A_+}\cdot g)
\|_{L^1_x}\|\1_{A_+}\cdot g\|_{L^1_x\ell^1_j}
\\
&\le C'\la^{\frac{n-1}2+\frac18} \|\1_{A_+}\cdot g\|_{L^1_x\ell^1_j}^2
\\
&\le C' \la^{\frac{n-1}2+\frac18}
\| g\|^2_{L^{q_c'}_x\ell^{q_c'}_j(A_+)}
\cdot \|\1_{A_+}\|^2_{L^{q_c}_x\ell^{q_c}_j}
\\
&=C' \la^{\frac{n-1}2+\frac18} \|\1_{A_+}\|^2_{L^{q_c}_x\ell^{q_c}_j}.
\end{align*}
But, by the definition of $A_+$ in \eqref{b.26}
$$\|\1_{A_+}\|^2_{L^{q_c}_x \ell^{q_c}_j}
\le \bigl(\la^{\frac{n-1}4+\frac18}\bigr)^{-2}
\|\A \sigma_\la \rho_\la f\|^2_{L^{q_c}_x\ell^{q_c}_j
(A_+)}.$$
So, assuming, as we may that $\l\gg 1$ is large,
we have
$$|II|\le C\la^{-1/8}
\|\A\sigma_\la \rho_\la f\|^2_{L^{q_c}_x\ell^{q_c}_j
(A_+)}
\le \tfrac12 \|\A\sigma_\la \rho_\la f\|^2_{L^{q_c}_x\ell^{q_c}_j
(A_+)}.$$

The estimates for $I$ and $II$ yield \eqref{b.31}.
\end{proof}

Let us also see how we can use Lemma~\ref{Gl} to
obtain the bounds in Theorem~\ref{bgsp} for 
$q>q_c$, which extend the results for compact
manifolds of Hassell and Tacy~\cite{HassellTacy}.

\begin{proof}[Proof of $q>q_c$ bounds]
Let us now prove the estimates in \eqref{i.6} for
$q>q_c$.  For a given such $q$, it suffices to show
that
$$\|\rho_\la \|_{L^2\to L^q}\lesssim T^{-1/2}\la^{\mu(q)},
$$
with $T=c_q\log\la$, $c_q=c_q(M)>0$ sufficiently small.
This in turn is equivalent to showing that
\begin{equation}\label{b.42}
\|\Psi(T(\la-P))\|_{L^{q'}\to L^q}\lesssim T^{-1}
\la^{2\mu(q)}, \, \, q>q_c,
\end{equation}
with $\Psi=|\rho|^2$, as above.

If, as in \eqref{b.33.3}, $\Psi(T(\la-P))=L_\la + G_\la$,
it is straightforward to check that \eqref{ii.3} yields
$$\|L_\la \|_{L^{q'}\to L^{q}}\lesssim T^{-1}\la^{2\mu(q)}.$$
Furthermore, by \eqref{b.34} and orthogonality, we have
$$\|G_\la\|_{L^2\to L^2}=O(1).$$
If we interpolate between this estimate and \eqref{b.34}
we obtain for $T\lesssim c_0 \log\la$ as above
$$\|G_\la\|_{L^{q'}\to L^q}=O(\la^{\frac{(n-1)(q-2)}{2q}}
\exp(C_MT)), \quad q>2.$$
Once checks that 
$\tfrac{ (n-1)(q-2)}{2q}
<
2\mu(q)
=2n(\tfrac12-\tfrac1q )-1$ 
if 
$q>q_c=\tfrac{2(n+1)}{n-1}$.  As a result, for such
an exponent, we have, for such $q$,
$\|G_\la\|_{L^{q'}\to L^{q}}=O(\la^{2\mu(q)-\e_q})$, some
$\e_q>0$, if, as we may assume $\la\gg 1$ and
$T=c_q\log\la$ with $c_q>0$ sufficiently small.

Since this and the above bound for $L_\la$ yields
\eqref{b.42}, the proof of the spectral projection
estimates in Theorem~\ref{bgsp} for $q>q_c$ is complete.
\end{proof}

Next, we note that we would complete the proof of the spectral projection
estimates in Theorem~\ref{bgsp} for $q=q_c$ if we could prove the low height
estimates \eqref{b.32} which, unlike \eqref{b.31}, differ depending on the 
curvature assumptions.  For this we shall need to use local bilinear harmonic
analysis which is a variable coefficient variant of that in 
Tao, Vargas and Vega~\cite{TaoVargasVega}
and relies on bilinear oscillatory integral estimates of Lee~\cite{LeeBilinear}.
Since the microlocal cutoffs in \eqref{b.17} arise from
the partition of unity in \eqref{b.12} corresponding to the balls $\{B(x_j,\delta)\}$
whose doubles have finite overlap, we shall be able carry out this analysis
in each ball $B(x_j,2\delta)$ using geodesic normal coordinates about the center.
Since, as we pointed out earlier, our assumption of bounded geometry ensures
bounded transition maps and uniform bounds on derivatives of the metric, we shall
be able to localize to individual balls.   As a result, we just need to repeat the arguments
in the earlier work of two of use \cite{HSp} for compact manifolds, which also reduced
to bilinear analysis in a fixed coordinate chart.

Just as in the earlier works for compact manifolds, \cite{BHSsp}, \cite{SBLOg}, \cite{HSp},
to prove \eqref{b.32}, besides \eqref{b.12}, we shall need to use a second microlocalization, which
involves localizing in $\theta\ge\la^{-1/8}$ neighborhoods of geodesics in a fixed coordinate chart.  To describe
this, let us fix $j$ in \eqref{b.12}, as well as $\ell_0\in \{1,\dots,K\}$ and consider the resulting
pseudodifferential cutoff, $A_{j,\ell_0}$, which is a summand in \eqref{b.17}.  Its symbol then satisfies
the conditions in \eqref{aas}.  The resulting geodesic normal coordinates on $B(x_j,2\delta)$
vanish at $x_j$.  We then have that the metric $g_{jk}$ satisfies $g_{jk}(y)=\delta^k_j+O((d_g(x_j,y))^2)$.  We may also
assume that $\xi_{j,\ell_0}=(0,\dots,0,1)$.  Since we are fixing
$j$ and $\ell_0$ for now, analogous to \cite{HSp}, let us simplify the notation a bit by letting
\begin{equation}\label{b.43}
\tilde \sigma_\la =A_{j,\ell_0}\sigma_\la,
\end{equation}
which is analogous to (2.10) in \cite{HSp}.

For dyadic $\theta\ge \la^{-1/8}$, the additional microlocal cutoffs that we require correspond to
$\theta$-nets of geodesics,
$\{\gamma_\nu\}$, 
in $S^*M$ passing through points $(y,\eta)$ near $(0,(0,\dots,0,1))$.   To define them, fix a function
$b\in C^\infty_0({\mathbb R}^{2(n-1)})$ supported in $\{z: \, |z_i|\le 1, \, 1\le i\le 2(n-1)\}$ satisfying
\begin{equation}\label{b.44}
\sum_{k\in {\mathbb Z}^{2(n-1)}} b(z-k)\equiv 1.
\end{equation}
To use this, let
$$\Pi =\{ y: \, y_n=0\}$$
be the points in $\Omega=B(x_j,2\delta)$ whose last coordinate vanishes.  Also let
$y'=(y_1,\dots, y_{n-1})$ and $\eta'=(\eta_1,\dots, \eta_{n-1})$ denote the first
$(n-1)$ coordinates of $y$ and $\eta$, respectively, with $(y,\eta)\in S^*\Omega$. 
 We shall always have $\theta\in [\la^{-1/8},1]$, and,
$\la^{-1/8}$, here, of course, is related to the height decomposition \eqref{b.26}.

To construct the cutoffs associated to the $\theta$-net of geodesics in $S^*M$ that we require, let us first
set
\begin{equation}\label{map0}
b^\theta_k(y',\eta')=b(\theta^{-1}(y',\eta')-k)\in C^\infty_0({\mathbb R}^{2(n-1)}),
\end{equation}
so that $\sum_{k\in {\mathbb Z}^{2(n-1)}} b^\theta_k(y',\eta')=1$.  We then note that the map
\begin{equation}\label{map1}
(t,x',\eta)\to \Phi_t(x',0,\eta)\in S^*\Omega, \quad (x',0,\eta)\in S^*\Omega
\end{equation}
is a diffeomorphism from a neighborhood of $x'=0$, $t=0$ and $(0,\dots,0,1)$ to a neighborhood of $(0,(0,\dots,0,1))\in S^*\Omega$.  

Next, write the inverse of \eqref{map1} as
$$S^*\Omega \ni (x,\omega)\to (\tau(x,\omega),\Psi(x,\omega), \Theta(x,\omega))\in \R \times \{y'\in \R^{n-1}\}\times S^*_{(\Psi(x,\omega),0)}M.$$
Thus, the unit speed geodesic passing through $(x,\omega)\in S^*\Omega$ arrives at the hyperplane  $\Pi$ where $y_n=0$ at $(\Psi(x,\omega),0)\in \Pi$,
has covector $\Theta(x,\omega)\in S^*_{(\Psi(x,\omega),0)}\Omega$ there, and $\tau(x,\omega)=d_g(x,(\Psi(x,\omega),0))$ is the geodesic distance
between $x$ and the point $(\Psi(x,\omega),0)$ on this hyperplane.  We shall also let $\tilde \Theta(x,\omega)$ denote the first $(n-1)$-coordinates
of the covector $\Theta(x,\omega)$, meaning that $\tilde\Theta(x,\omega)=\eta'$ if $\Theta(x,\omega)=(\eta',\eta_n)\in S^*_{(\Psi(x,\omega),0)}\Omega$.

We can now define the microlocal cutoffs that we shall use.  For $(x,\xi)\in T^*\Omega$ in a conic neighborhood of $(0,(0,\dots,0,1))$, if
$b^\theta_k$ is as in \eqref{map0}, we define
\begin{equation}\label{b.45}
q^\theta_k(x,\xi)=b^\theta_k(\Psi(x,\xi/p(x,\xi)), \tilde \Theta(x,\xi/p(x,\xi))),
\end{equation}
with $p(x,\xi)$ being the principal symbol of $P=\sqrt{-\Delta_g}$.

Note that
\begin{equation}\label{b.46}
q^\theta_k(\Phi_s(x,\xi))=q^\theta_k(x,\xi).
\end{equation}
Furthermore, 
 if $(y_\nu',\eta'_\nu)=\theta k=\nu$ and $\gamma_\nu$ is the geodesic
in $S^*\Omega$ passing through $(y'_\nu,0,\eta_\nu)\in S^*\Omega$ with 
$\eta_\nu\in S^*_{(y',0)}\Omega$ having $\eta_\nu'$ as its first $(n-1)$ coordinates
and $\eta_n>0$ then
\begin{equation}\label{b.48}
q^\theta_k(x,\xi)=0 \, \, 
\text{if } \, \text{dist }((x,\xi),\gamma_\nu)\ge C_0\theta, \, \, \nu=\theta k
\end{equation}
for a uniform constant $C_0$.  Also, $q_k^\theta$ satisfies the estimates
\begin{equation}\label{b.49}
|\partial_x^\sigma \partial_\xi^\gamma q^\theta_k(x,\xi)|\lesssim \theta^{-|\sigma|-|\gamma|},
\quad \text{if } \, \, p(x,\xi)=1,
\end{equation}
related to this support property.

Next, fix $\tilde \psi\in C^\infty_0$ supported in $|x|<3\delta/2$ which equals one when
$|x|\le 5\delta/4$.  Additionally, fix $\tilde{\tilde \psi} \in C^\infty_0$ supported in 
$|x|<2\delta$ which equals one when $|x|\le 3\delta/2$.
Also, fix 
$\tilde \beta \in C^\infty_0((0,\infty))$ so that $\tilde \beta(p(x,\xi)/\la)$ equals one in a neighborhood
of the $\xi$ support of $a^\la_{j,\ell_0}$.  We then define the compound symbols $Q^\theta_\nu
=Q^\theta_{j,\ell_0,\nu}$ and associated operators by
\begin{multline}\label{b.50}
Q^\theta_\nu(x,y,\xi)=\tilde \psi(x) \tilde{\tilde \psi}(y) q^\theta_k(x,\xi) \, 
\tilde \beta(p(x,\xi)/\la), \, \,
\nu=\theta k\in \theta\cdot {\mathbb Z}^{2(n-1)}, \, \, \text{and}
\\
Q^\theta_\nu h(x) = (2\pi)^{-n}
\iint e^{i(x-y)\cdot \xi}   Q^\theta_\nu(x,y,\xi)\,  h(y) \, d\xi dy.
\end{multline}
It follows that these dyadic pseudodifferential operators belong to a bounded subset
of $S^0_{7/8,1/8}$ due to our assumption that $\theta\in [\la^{-1/8},1]$.  We have constructed these
operators so that for small enough $\delta_0>0$ we have
\begin{multline}\label{b.51}
Q^\theta_\nu(x,y,\xi)=Q^\theta_\nu(z,y,\eta), \, (z,\eta)=\Phi_t(x,\xi), 
\\ \text{if } \, \text{dist }((x,\xi), \, \text{supp }A_{j,\ell_0})\le \delta_0
\, \text{and } \, \, |t|\le 2\delta_0.
%
\end{multline}

The compound symbol involves the cutoff $\tilde{\tilde \psi}(y)$ which equals one on a neighborhood 
of the $x$-support of $A_{j,\ell_0}$ as well as the support of $\tilde \psi$.
We use cutoffs in both variables
since $M$ is not assumed to be compact and we want to avoid
issues at infinity.   This symbol in \eqref{b.51} vanishes when either $x$ or $y$ is outside the $2\delta$-ball about
the origin in our coordinates for $\Omega$.  By \eqref{b.5'} \eqref{b.13} and \eqref{b.43}, we can fix $\delta_1$ in \eqref{b.3} small enough so that
we also have, analogous to (2.41) in \cite{HSp},
\begin{multline}\label{b.52}
\tilde \sigma_\la = \sum_\nu \tilde \sigma_\la Q^{\theta_0}_\nu + R, \quad R=R_{\la,j,\ell_0},  \, \, \theta_0=\la^{-1/8},
\, \,  \tilde \sigma_\la =A_{j,\ell_0}\sigma_\la, 
\\
\text{where } \, \, R(x,y)=O(\la^{-N}), \forall \, N \, \, 
\text{and } R(x,y)=0, \, \, \text{if } x\notin B(x_j,2\delta) \, \,
\text{or } y\notin B(x_j,2\delta),
\end{multline}
with bounds for the remainder kernel independent of $j$.

Let us now point out straightforward but useful properties of our operators.  First, by \eqref{b.13}, \eqref{b.52}
and the support properties of $\tilde \psi$, $\tilde{\tilde \psi}$, we have
\begin{multline}\label{b.54}
\tilde \sigma_\la Q^{\theta_0}_\nu h=\1_{B(x_j,2\delta)} \cdot \tilde \sigma_\la Q^{\theta_0}_\nu\bigl(
\1_{B(x_j,2\delta)} \cdot h\bigr), \, \, Q^{\theta_0}_\nu =Q^{\theta_0}_{j,\ell_0,\nu}
\\
\text{and } \, \, Rh=\1_{B(x_j,2\delta)}\cdot R \bigl(
\1_{B(x_j,2\delta)} \cdot h\bigr), \, \, R=R_{\la,j,\ell_0}.
\end{multline}

Also, we have the uniform bounds
\begin{equation}\label{b.55}
\begin{split}
\|Q^{\theta_0}_\nu h\|_{\ell^q_\nu L^q(M)}\lesssim \|h\|_{L^q(M)}, \, \, 2\le q\le \infty
\\
\bigl\| \sum_{\nu'}(Q^{\theta_0}_\nu)^*H(\nu', \, \cdot \, )\|_{L^p(M)}\lesssim 
\|H\|_{\ell^p_{\nu'}L^p(M)}, \, \, 1\le p\le 2.
\end{split}
\end{equation}
The second estimate follows via duality from the first.  The first one is (2.33) in Lemma~2.2 of \cite{HSp}.  By interpolation,
one just needs to verify that the estimate holds for the two endpoints, $p=2$ and $p=\infty$.  The former follows via an almost
orthogonality argument, and the latter from the fact that, if $Q^{\theta_0}_\nu(x,y)$ denotes the Schwartz kernel of the operator $Q^{\theta_0}_\nu$, then  we have the uniform bounds
$$\sup_{x\in B(x_j,2\delta)}\int_{B(x_j,2\delta)}|Q^{\theta_0}_\nu(x,y)| \, dy\le C.$$
See \cite{HSp} for more details.

Note that if we use \eqref{b.55} along with \eqref{b.50} and the finite overlap of the balls $\{B(x_j,2\delta)\}$ we obtain
for our fixed $\ell_0=1,\dots,K$
\begin{equation}\label{b.555}
\begin{split}
\bigl(\sum_{j,\nu}\|Q^{\theta_0}_{j,\ell_0,\nu} h\|_{ L^q(M)}^q\bigr)^{1/q}\lesssim \|h\|_{L^q(M)}, \, \, 2\le q\le \infty
\\
\bigl\| \sum_{j',\nu'}(Q^{\theta_0}_{j',\ell_0,\nu'})^*H(\nu',j', \, \cdot \, )\|_{L^p(M)}\lesssim 
\|H\|_{\ell^p_{\nu'}L^p(M)}, \, \, 1\le p\le 2.
\end{split}
\end{equation}

In addition to this inequality and \eqref{b.10} we shall require another that follows almost directly from a result
in \cite{HSp}.  Specifically, we require the following commutator bounds
\begin{equation}\label{comm}
\bigl\| (A_{j,\ell_0}\sigma_\la Q^{\theta_0}_{j,\ell_0,\nu}-A_{j,\ell_0}Q^{\theta_0}_{j,\ell_0,\nu}\sigma_\la)h\|_{L^q(M)}
\le C_q \la^{\mu(q)-1/4}\|h\|_{L^2(B(x_j,2\delta))},
\end{equation}
with $\mu(q)$ is as in \eqref{ii.2},
assuming that $\delta$, as well as $\delta_1$ in \eqref{b.3} are fixed small enough.

To see this, let $\tilde A_{j,\ell_0}$ be a 0-order pseudo-differential operator with symbol $\tilde a_{j,\ell_0}^\la(x,y,\xi)$ supported in $|\xi|\approx \la$ and which equals one in 
the support of the  symbol $ a_{j,\ell_0}^\la(x,y,\xi)$ of the $ A_{j,\ell_0}$ operator, then it is not hard to see that
\begin{equation}\label{i.5a}
\| A_{j,\ell_0}-\tilde A_{j,\ell_0}  A_{j,\ell_0}\|_{L^p_x\to L^p_x}=O(\la^{-N})
\, \, \forall \, N\, \, \text{if } \,
\, 1\le p\le \infty.
\end{equation}
And by using the fact that 
 the kernel $\tilde A_{j,\ell_0} (x,y)$ is $O(\la^{n}(1+\la|x-y|)^{-N})$ and Young's inequality, we also have 
\begin{equation}\label{i.5b}
\| \tilde A_{j,\ell_0} \|_{L^2_x\to L^p_x}=O(\la^{n(\frac12-\frac1p)})\,\,\, \text{if } \,
\, 2\le p\le \infty.
\end{equation}
Thus to prove \eqref{comm} it suffices to show
 \begin{equation}\label{comma}
\bigl\| (A_{j,\ell_0}\sigma_\la Q^{\theta_0}_{j,\ell_0,\nu}-A_{j,\ell_0}Q^{\theta_0}_{j,\ell_0,\nu}\sigma_\la)h\|_{L^2(M)}
\le C_q \la^{-3/4}\|h\|_{L^2(B(x_j,2\delta))}
\end{equation}
since $n(\frac12-\frac1p)-\frac34\le \mu(q)-\frac14$ for $q\le q_c$.
 This follows from the 
proof of (2.39) in \cite{HSp} since, by \eqref{aas}, $A_{j,\ell_0}f$ vanishes outside
$B(x_j,2\delta)$ and the two operators in \eqref{comma} vanish when acting on functions vanishing on
$B(x_j,2\delta)$.  This, just as in \cite{HSp}, allows one to prove \eqref{comma},
exactly as in \cite{HSp},  by just working in a coordinate chart
($B(x_j,2\delta)$ here) and, to obtain the inequality
using \eqref{b.46}, \eqref{b.51} and  Egorov's theorem related to the properties of the half wave operator $e^{itP}$ in this local coordinate.

Next, as in \cite{BHSsp} and \cite{HSp}, we note that we can write for $\theta_0$ and $\tilde \sigma_\la$ as in \eqref{b.43}
\begin{equation}\label{b.56}
\bigl(\tilde \sigma_\la h\bigr)^2 =
\sum_{\nu,\nu'}
\bigl(\tilde \sigma_\la Q^{\theta_0}_\nu h\bigr)\cdot 
\bigl(\tilde \sigma_\la Q^{\theta_0}_{\nu'}h\bigr) +O(\la^{-N}\|h\|^2_{L^2(B(x_j,2\delta))}), \, \, \forall \, N.
\end{equation}

Note that the $\nu=\theta_0\cdot {\mathbb Z}^{2(n-1)}$ index a $\la^{-1/8}$-separated lattice in ${\mathbb R}^{2(n-1)}$.   As
in earlier works, to be able to apply bilinear oscillatory integral results, we need to organize the pairs $(\nu,\nu')$ in
the above sum.  As in \cite{TaoVargasVega}, we first consider dyadic cubes $\tau^\theta_\mu$ in 
${\mathbb R}^{2(n-1)}$ of sidelength $\theta=2^k\theta_0=2^k\la^{-1/8}$, with $\tau^\theta_\mu$ denoting translation of the cube
$[0,\theta)^{2(n-1)}$ by $\mu=\theta\cdot {\mathbb Z}^{2(n-1)}$.  We then say two such cubes are {\em close} if they are not
adjacent but have adjacent parents of sidelength $2\theta$.  In this case we write $\tau^\theta_\mu \sim \tau^\theta_{\mu'}$.  Note that
close cubes satisfy $\text{dist }(\tau^\theta_\mu,\tau^\theta_{\mu})\approx \theta$ and also that each fixed cube has
$O(1)$ cubes that are ``close'' to it.  Moreover, as was noted in \cite{TaoVargasVega}, any distinct points
$\nu,\nu'\in {\mathbb R}^{2(n-1)}$ must lie in a unique pair of cubes in this Whitney decomposition of ${\mathbb R}^{2(n-1)}$.
Consequently, there must be a unique triple $(\theta=2^k\theta_0, \mu,\mu')$ such that
$(\mu,\mu')\in \tau^\theta_\mu\times \tau^\theta_{\mu'}$ and $\tau^\theta_{\mu}\sim \tau^\theta_{\mu'}$.

We also note that if, as we shall, we fix the $\delta$ occurring in the construction of the $\{A_{j,\ell}\}$ to be small enough
then we only need to consider $\theta=2^k\theta_0\ll 1$ when dealing with the bilinear sum in \eqref{b.56}.

Based on these observations, we can organize the sum in \eqref{b.56} as follows
\begin{multline}\label{b.57}
\sum_{\{k\in {\mathbb N}: \, k\ge 10 \, \, \text{and } \, 
\theta=2^k\theta_0\ll 1\}}
\sum_{\{(\mu, \mu'): \, \tau^\theta_\mu
\sim \tau^\theta_{ \mu'}\}}
\sum_{\{(\nu, \nu')\in
\tau^\theta_\mu\times \tau^\theta_{ \mu'}\}}
\bigl(\tilde \sigma_\la
Q^{\theta_0}_\nu h\bigr) 
\cdot \bigl(\tilde \sigma_\la
Q^{\theta_0}_{ \nu'} h\bigr)
\\
+\sum_{(\nu, \nu')\in \Xi_{\theta_0}} 
\bigl( \tilde \sigma_\la Q^{\theta_0}_\nu h\bigr) 
\cdot \bigl( \tilde \sigma_\la
Q^{\theta_0}_{ \nu'} 
h\bigr)
,
\end{multline}
where $\Xi_0$ indexes the remaining pairs such that $|\nu-\nu'|\lesssim \theta_0=\la^{-1/8}$, including the diagonal ones
where $\nu=\nu'$.

Let us then set  for our fixed $(j,\ell_0)$ and  $\tilde \sigma_\la =A_{j,\ell_0}\sigma_\la$
\begin{equation}\label{b.58}
\diag_{j,\ell_0}(h)=\diag(h)=\sum_{(\nu, \nu')\in \Xi_{\theta_0}} 
\bigl( \tilde \sigma_\la Q^{\theta_0}_\nu h\bigr) 
\cdot \bigl( \tilde \sigma_\la
Q^{\theta_0}_{ \nu'} 
h\bigr)
\end{equation}
and
\begin{equation}\label{b.59}
\far_{j,\ell_0}(h)=
\far(h)=\sum_{(\nu, \nu')\notin \Xi_{\theta_0}} 
\bigl( \tilde \sigma_\la Q^{\theta_0}_\nu h\bigr) 
\cdot \bigl( \tilde \sigma_\la
Q^{\theta_0}_{ \nu'} 
h\bigr)+O(\la^{-N}\|h\|^2_{L^2(B(x_j,2\delta)}),
\end{equation}
with the last term being the error term in \eqref{b.56}.
Due to this splitting we have
the analog of (5.5) in \cite{HSp}
\begin{equation}\label{b.60}
(\tilde \sigma_\la h)^2 = \diag(h) +\far(h).
\end{equation}

We shall use this decomposition when $n\ge3$, since then $q_c\le 4$, which allows us to use bilinear ideas
from \cite{TaoVargasVega}, exploiting the fact that $q_c/2\in [1,2]$.
When the dimension $n$ of $M$ equals 2, though, the critical exponent $q_c=6$, which, as in \cite{BHSsp} and \cite{HSp},
 requires a slight modification of the above splitting.

Specifically, for $n=2$, we first, as in these two earlier works, set
\begin{equation}\label{Tnu}
   T_\nu h=\sum_{\nu': \,(\nu, \nu')\in
\Xi_{\theta_0}}( \tilde\sigma_\la Q^{\theta_0}_\nu h )(\tilde\sigma_\la Q^{\theta_0}_{\nu'} h), 
\end{equation}
and write
\begin{equation}\label{b255}
    ( \diag(h))^{2} =\big(\sum_\nu T_\nu h\big)^2
    = \sum_{\nu_1, \nu_2}  T_{\nu_1} hT_{\nu_2} h.
\end{equation}
If, as above, we fix $\delta$ small enough then the 
sum in \eqref{b255} can be organized as
\begin{equation}\label{organize}
\begin{aligned}
    &\big( \sum_{\{k\in {\mathbb N}: \, k\ge 20 \, \, \text{and } \, 
\theta=2^k\theta_0\ll 1\}} \,  \, 
\sum_{\{(\mu_1, \mu_2): \, \tau^\theta_{\mu_1}
\sim \tau^\theta_{ \mu_2}\}}
\sum_{\{(\nu_1,\nu_2)\in
\tau^\theta_{\mu_1}\times \tau^\theta_{\mu_2}\}}
+\sum_{(\nu_1,  \nu_2)\in \overline\Xi_{\theta_0}}
\big)T_{\nu_1} hT_{\nu_2} h,  \\
&\qquad \qquad ={\overline\Upsilon^{\text{far}}}(h)+{\overline\Upsilon^{\text{diag}}}(h).
\end{aligned}
\end{equation}
Here $\overline\Xi_{\theta_0}$ indexes the  near diagonal pairs. This is another Whitney decomposition similar to
the one in  \eqref{b.57}, but the diagonal set $\overline\Xi_{\theta_0}$ is much larger than the set $\Xi_{\theta_0}$ in \eqref{b.57}. More explicitly, when $n=2$, it is not hard to check that  $|\nu-\nu'|\le 2^{11}\theta_0$ if $(\nu,\nu')\in \Xi_{\theta_0}$
while  $|\nu_1-\nu_2|\le 2^{21}\theta_0$ if $(\nu_1,\nu_2)\in \overline\Xi_{\theta_0}$. As noted in \cite{HSp}, this helps to simplify the calculations needed for ${\overline\Upsilon^{\text{far}}}(h)$.
Note that for our fixed $(j,\ell_0)$,
${\overline\Upsilon^{\text{diag}}}(h)={\overline\Upsilon_{j,\ell_0}^{\text{diag}}}(h)$,  ${\overline\Upsilon^{\text{far}}}(h)={\overline\Upsilon_{j,\ell_0}^{\text{far}}}(h)$, 
and $\far_{j,\ell_0}(h)=\far(h)$ as in \eqref{b.59}, we then have
\begin{equation}\label{b.64}
(\tilde \sigma_\la h)^4\le 2(\diag h)^2+2(\far h)^2=2{\overline\Upsilon^{\text{diag}}}(h)+2{\overline\Upsilon^{\text{far}}}(h)+2(\far(h))^2, \, \, \text{if }n=2.
\end{equation}


We have organized the sums expanding the left side of \eqref{b.56} exactly as in \cite{HSp}.  In view of \eqref{b.54} each of the summands in the above decompositions
is localized to our coordinate chart $\Omega=B(x_j,2\delta)$ on which we are using geodesic normal coordinates about the center.  Since our bounded geometry assumptions
ensures we have uniform control of the metric and its derivatives, for $\delta>0$ fixed small enough, we can simply repeat the proof of Lemma 5.1 in \cite{HSp} to obtain the following 
variable coefficient variant of Lemma~6.1 in Tao, Vargas and Vega~\cite{TaoVargasVega}.

\begin{lemma}\label{lemma1}  Let $\theta_0=\la^{-1/8}$ with $\la \gg 1$.
If $n\ge3$ there is a uniform constant $C=C_M$ independent of $(j,\ell_0)$  so that
if, as in \eqref{b.50}, $Q^{\theta_0}_\nu=Q^{\theta_0}_{j,\ell_0,\nu}$
\begin{multline}\label{b.65}
\bigl\| \diag_{j,\ell_0}(h) \bigr\|_{L^{q_c/2}(M)}
\le C\bigl( \, \sum_\nu \bigl\| A_{j,\ell_0}\sigma_\la Q^{\theta_0}_{j,\ell_0,\nu} h\bigr\|^{q_c}_{L^{q_c}(B(x_j,2\delta))}\, \bigr)^{2/q_c}
\\+O(\la^{\frac2{q_c}-}\|h\|^2_{L^2(B(x_j,2\delta))}).
\end{multline}
Also, for all $n\ge2$, if $q\in (2,\tfrac{2(n+2)}n]$ and $\mu(q)$ as in \eqref{ii.2},  there is a uniform constant
$C_q=C(q,M)$ so that
\begin{multline}\label{b.66}
\bigl\| \diag_{j,\ell_0}(h) \bigr\|_{L^{q/2}(M)}
\le C_q \bigl( \, \sum_\nu \bigl\| A_{j,\ell_0}\sigma_\la Q^{\theta_0}_{j,\ell_0,\nu} h\bigr\|^{q}_{L^{q}(B(x_j,2\delta))}\, \bigr)^{2/q}
\\+O(\la^{2\mu(q)-}\|h\|^2_{L^2(B(x_j,2\delta))}).
\end{multline}
Additionally, for $n=2$ there is a uniform constant $C=C(M)$ so that
\begin{multline}\label{b.67}
\bigl\| \diagbar_{j,\ell_0}(h) \bigr\|_{L^{3/2}(M)}
\\
\le C\bigl( \, \sum_\nu \bigl\| A_{j,\ell_0}\sigma_\la Q^{\theta_0}_{j,\ell_0,\nu} h\|^6_{L^6(B(x_j,2\delta))}\bigr)^{2/3}
+O(\la^{\frac23-}\|h\|^4_{L^2(B(x_j,2\delta))}).
\end{multline}
\end{lemma}

In the above and what follows $O(\la^{\mu-})$ denotes $O(\la^{\mu-\e_0})$ for some $\e_0>0$.

If we  fix $\delta$ as well as 
$\delta_1,\delta_2$ in \eqref{b.3} small enough, then we can use Lee's \cite{LeeBilinear} bilinear oscillatory integral theorem 
and repeat the proof of Lemma 5.2 in \cite{HSp}
to obtain the following.

\begin{lemma}\label{lemma2}
Let $n\ge2$ and $\far(h)=\far_{j,\ell_0}(h)$ be as above with $\theta_0=\la^{-1/8}$.  Then for all $\e>0$ there is a $C_\e=C(\e,M)$ so that
\begin{equation}\label{b.68}
\int_M |\far_{j,\ell_0}(h)|^{q/2} \, dx \le C_\e \la^{1+\e} \, \bigl(\la^{7/8}\bigr)^{\frac{n-1}2(q-q_c)}
\|h\|^q_{L^2(B(x_j,2\delta))}, \quad q=\tfrac{2(n+2)}n.
\end{equation}
Similarly, for all $n\ge2$, there is a constant $C_q=C(q,M)$ so that
\begin{equation}\label{b.69}
\int_M |\far_{j,\ell_0}(h)|^{q/2} \, dx \le C_q \, \la^{q\cdot \mu(q)-}\|h\|^q_{L^2(B(x_j,2\delta))}, \quad 2<q<\tfrac{2(n+2)}n,
\end{equation}
and, if $n=2$ and $\farbar(h)=\farbar_{j,\ell_0}(h)$ as in \eqref{organize},
\begin{equation}\label{b.70}
\int_M |\farbar_{j,\ell_0}(h)| \, dx \le C_\e \la^{1+\e}\la^{-7/8}\|h\|^4_{L^2(B(x_j,2\delta))}, \, \, \forall \, \e>0,
\end{equation}
with $C_\e=C(\e,M)$.
\end{lemma}

We now have collected the main ingredients that we need to prove the critical low height estimates.

\begin{proof}[Proof of \eqref{b.32}]
Let us assume that $n\ge3$.  A main step in the proof of the $A_-$ estimates then is to obtain the analog of (2.44) in \cite{HSp}.
We shall do so largely by repeating its proof, which we do so for the sake of completeness in order to note the small changes needed
to take into account that, unlike (2.44) in \cite{HSp}, \eqref{b.32} here is a vector valued inequality.  As noted before, we have taken
this framework to help us exploit our assumption of bounded geometry, and,  in particular, the fact that the doubles of the balls 
in our covering of $M$, $\{B(x_j,2\delta)\}$,
have uniformly finite overlap.

We first note that if $q=\tfrac{2(n+2)}n<q_c$, then by \eqref{b.27} and  \eqref{b.60} for our fixed $\ell_0$ we have
\begin{align*}
\bigl|&\bigl(\A\sigma_\la (\rho_\la f)(x,j)\bigr)^2\bigr|^{q_c/2}=\bigl|\bigl(A_{j,\ell_0}(\sigma_\la\rho_\la f)(x) \bigr)^2\bigr|^{q_c/2}
\\
&=|A_{j,\ell_0}(\sigma_\la \rho_\la f)(x) \cdot A_{j,\ell_0}(\sigma_\la \rho_\la f)(x)|^{\frac{q_c-q}2}
\bigl| \diag_{j,\ell_0}(\rho_\la f)(x) +\far_{j,\ell_0}(\rho_\la f)(x)\bigr|^{q/2}
\\
&\le |A_{j,\ell_0}(\sigma_\la \rho_\la f)(x) \cdot A_{j,\ell_0}(\sigma_\la \rho_\la f)(x)|^{\frac{q_c-q}2}
 2^{q/2}  \bigl( 
|\diag_{j,\ell_0}(\rho_\la f)(x)|^{q/2}
+|\far_{j,\ell_0}(\rho_\la f)(x)|^{q/2} \bigr).
\end{align*}
Also, if $A_-$ is as in \eqref{b.26},
$$\| \A(\sigma_\la \rho_\la f)\|^{q_c}_{L^{q_c}_x \ell_j^{q_c}(A_-)}
=\int_M \sum_j \1_{A_-}(x,j) \, |A_{j,\ell_0}(\sigma_\la \rho_\la f)(x)|^{q_c} \, dx.
$$
Thus,
\begin{align*}
\| &\A(\sigma_\la \rho_\la f)\|^{q_c}_{L^{q_c}_x \ell_j^{q_c}(A_-)}
=\sum_j \int \1_{A_-}(x,j) \, |A_{j,\ell_0}(\sigma_\la \rho_\la f)(x)
\cdot
A_{j,\ell_0}(\sigma_\la \rho_\la f)(x)
|^{q_c/2} \, dx
\\
&\le C \sum_j \int 
\bigl[ \1_{A_-}(x,j) \, |A_{j,\ell_0}(\sigma_\la \rho_\la f)(x) \cdot A_{j,\ell_0}(\sigma_\la \rho_\la f)(x)|^{\frac{q_c-q}2}
\bigr] \, |\diag_{j,\ell_0}(\rho_\la f)(x)|^{q/2} \, dx
\\
&\, \, +C\sum_j  \int 
\bigl[ \1_{A_-}(x,j) \, |A_{j,\ell_0}(\sigma_\la \rho_\la f)(x) \cdot A_{j,\ell_0}(\sigma_\la \rho_\la f)(x)|^{\frac{q_c-q}2}
\bigr] \, |\far_{j,\ell_0}(\rho_\la f)(x)|^{q/2} \, dx
\\
&\qquad \qquad =C(I+II).
\end{align*}

To handle $II$ we recall that by \eqref{b.26}, \eqref{b.27} and \eqref{b.54}
$$|\1_{A_-}(x,j) A_{j,\ell_0}\sigma_\la \rho_\la f(x)|\le \la^{\frac{n-1}4+\frac18}.$$
Thus, by \eqref{b.68},
\begin{align*}
CII&\lesssim \la^{(\frac{n-1}4+\frac18)(q_c-q)}\cdot \la^{1+\e} \, (\la^{\frac78})^{\frac{n-1}2(q-q_c)}\sum_j \|\rho_\la f\|^q_{L^2(B(x_j,2\delta)}
\\
&\lesssim \la^{1-\delta_n+\e} \|\rho_\la f\|_{L^2(M)}^q
\le \la^{1-\delta_n+\e}\|f\|_{L^2(M)}^q=\la^{1-\delta_n+\e},
\end{align*}
using also, in the second inequality, the bounded overlap of the $\{B(x_j,2\delta)\}$.  Also, a simple calculation shows that $\delta_n>0$.

To control $CI$, as in \cite{BHSsp} and \cite{HSp}, we use H\"older's inequality and Young's inequality along with \eqref{b.65}  to get
\begin{align*}
CI&\le \|\A(\sigma_\la \rho_\la f)\cdot \A(\sigma_\la\rho_\la f)\|_{L^{q_c/2}_x\ell_j^{q_c/2}(A_-)}^{\frac{q_c-q}2} \cdot C\| \diag_{j,\ell_0}(\rho_\la f)\|^{q/2}_{L^{q_c/2}_x\ell_j^{q_c/2}}
\\
&\le \tfrac{q_c-q}{q_c}\|\A(\sigma_\la \rho_\la f)\cdot \A(\sigma_\la\rho_\la f)\|_{L^{q_c/2}_x\ell_j^{q_c/2}(A_-)}^{q_c/2} +\tfrac{q}{q_c} C 
\| \diag_{j,\ell_0}(\rho_\la f)\|^{q_c/2}_{L^{q_c/2}_x\ell_j^{q_c/2}} 
\\
&\le \tfrac{q_c-q}{q_c}\|\A(\sigma_\la \rho_\la f)\|_{L^{q_c}_x\ell_j^{q_c}(A_-)}^{q_c}
\\
&\qquad +C' \tfrac{q}{q_c}
\bigl[\sum_{j,\nu} \|A_{j,\ell_0}\sigma_\la Q^{\theta_0}_{j,\ell_0,\nu}\rho_\la f\|^{q_c}_{L^{q_c}(B(x_j,2\delta))}
+\la^{1-}\bigl(\sum_j \|\rho_\la f\|_{L^2(B(x_j,2\delta))}\bigr)^{q_c/2} \bigr]
\\
&\le \tfrac{q_c-q}{q_c}\|\A(\sigma_\la \rho_\la f)\|_{L^{q_c}_x\ell_j^{q_c}(A_-)}^{q_c} 
+C'' \sum_{j,\nu} \|A_{j,\ell_0}\sigma_\la Q^{\theta_0}_{j,\ell_0,\nu}\rho_\la f\|^{q_c}_{L^{q_c}(B(x_j,2\delta))} +\la^{1-},
\end{align*}
again using the finite overlap of the $\{B(x_j,2\delta)\}$.

Since $\tfrac{q_c-q}{q_c}<1$ we can use the bounds for $CI$ and $CII$ to obtain the key inequality
\begin{equation*}
\|\A(\sigma_\la \rho_\la f)\|_{L^{q_c}_x\ell^{q_c}_j(A_-)}
\lesssim \bigl( \, \sum_{j,\nu}\| A_{j,\ell_0}\sigma_\la Q^{\theta_0}_{j,\ell_0,\nu}\rho_\la f\|_{L^{q_c}(M)}^{q_c}\bigr)^{1/q_c}+\la^{\frac1{q_c}-},
\end{equation*}
which is the analog of the estimate (2.44)  in Proposition~2.3 in \cite{HSp} for $n\ge3$.  One can similarly use \eqref{b.67} and \eqref{b.70} and
 modify the arguments in \cite{HSp}
to handle the two-dimensional case.  Similarly, if one also uses \eqref{b.69}  one can obtain an analog of the preceding
estimate for  subcritical exponents, $2<q<q_c$, which is more straightforward and does not require the norm in the left
to be taken over $A_-$.

Thus, just as the preceding inequality followed from straightforward modifications of the arguments in \cite{HSp}, so do 
the other estimates in the following result coming from variable coefficient variants of the bilinear harmonic
analysis in \cite{TaoVargasVega}.

\begin{proposition}\label{keyprop}
Fix a complete $n\ge2$ Riemannian manifold $(M,g)$ of bounded geometry and assume that \eqref{b.norm} is valid.
Then $\la\gg 1$ and $\theta_0=\la^{-1/8}$
\begin{equation}\label{key}
\|\A(\sigma_\la \rho_\la f)\|_{L^{q_c}_x\ell^{q_c}_j(A_-)}
\lesssim \bigl( \, \sum_{j,\nu}\| A_{j,\ell_0}\sigma_\la Q^{\theta_0}_{j,\ell_0,\nu}\rho_\la f\|_{L^{q_c}(M)}^{q_c}\bigr)^{1/q_c}+\la^{\frac1{q_c}-},
\end{equation}
assuming that $\delta$ and $\delta_1$ above are small enough.  Additionally, for $ 2<q\le \tfrac{2(n+2)}{n}$,
\begin{equation}\label{key2}
\bigl(\sum_{j} \|    A_{j,\ell_0}  (\sigma_\la \rho_\la f)\|_{L^{q}_x(M)}^q\bigr)^{1/q}
\lesssim \bigl( \, \sum_{j,\nu}\| A_{j,\ell_0}\sigma_\la Q^{\theta_0}_{j,\ell_0,\nu}\rho_\la f\|_{L^{q}(M)}^{q}\bigr)^{1/q}+\la^{\mu(q)-}.
\end{equation}
\end{proposition}

Due to \eqref{key}, in order to prove \eqref{b.32} and finish the proof of the $q_c$-estimates
in Theorem~\ref{bgsp}, it suffices to show that if, as above, we take $T=c_0\log\la$ as in \eqref{b.31}, then
\begin{multline}\label{b.75}
\bigl( \, \sum_{j,\nu}\| A_{j,\ell_0}\sigma_\la Q^{\theta_0}_{j,\ell_0,\nu}\rho_\la f\|_{L^{q_c}(M)}^{q_c}\bigr)^{1/q_c}
\lesssim \la^{\mu(q_c)}T^{-1/2}, 
\\
\text{if all the sectional curvatures of } 
M \, \text{are } \le -\kappa_0^2, \,
\text{some }\kappa_0>0,
\end{multline}
and
\begin{multline}\label{b.76}
\bigl( \, \sum_{j,\nu}\| A_{j,\ell_0}\sigma_\la Q^{\theta_0}_{j,\ell_0,\nu}\rho_\la f\|_{L^{q_c}(M)}^{q_c}\bigr)^{1/q_c}
\lesssim 
\bigl(\la T^{-1}\bigr)^{\mu(q_c)}, 
\\
 \text{if all the sectional curvatures of } 
M \, \text{are nonpositive}.
\end{multline}

To prove these two estimates we shall argue as in the proof of (2.56) in \cite{HSp} and use \eqref{b.6}, \eqref{b.16}
 and \eqref{comm} to obtain
 \begin{align*}
 \sum_{j,\nu} &\| A_{j,\ell_0}\sigma_\la Q^{\theta_0}_{j,\ell_0,\nu}\rho_\la f\|_{L^{q_c}}^{q_c}
 =\sum_{j,\nu} \| A_{j,\ell_0}\sigma_\la Q^{\theta_0}_{j,\ell_0,\nu}\rho_\la f\|_{L^{q_c}}^2 \cdot
  \| A_{j,\ell_0}\sigma_\la Q^{\theta_0}_{j,\ell_0,\nu}\rho_\la f\|_{L^{q_c}}^{q_c-2}
  \\
  &\le \sum_{j,\nu}  \| A_{j,\ell_0}\sigma_\la Q^{\theta_0}_{j,\ell_0,\nu}\rho_\la f\|_{L^{q_c}}^2 \cdot
  \| A_{j,\ell_0} Q^{\theta_0}_{j,\ell_0,\nu}  \sigma_\la  \rho_\la f\|_{L^{q_c}}^{q_c-2}
  \\
  &\quad +\sum_{j,\nu} \| A_{j,\ell_0}\sigma_\la Q^{\theta_0}_{j,\ell_0,\nu}\rho_\la f\|_{L^{q_c}}^2 \cdot
  \| \, (A_{j,\ell_0}\sigma_\la Q^{\theta_0}_{j,\ell_0,\nu}-A_{j,\ell_0}Q^{\theta_0}_{j,\ell_0,\nu}\sigma_\la)\rho_\la f\|^{q_c-2}_{L^{q_c}}
  \\
  &\lesssim \sum_{j,\nu} \| A_{j,\ell_0}\sigma_\la Q^{\theta_0}_{j,\ell_0,\nu}\rho_\la f\|_{L^{q_c}}^2 \cdot
  \| A_{j,\ell_0} Q^{\theta_0}_{j,\ell_0,\nu}  \sigma_\la  \rho_\la f\|_{L^{q_c}}^{q_c-2}
  \\
  &\quad +\sum_{j,\nu} \la^{2/q_c} \|Q^{\theta_0}_{j,\ell_0,\nu}\rho_\la f\|_{L^2}^2 \cdot \la^{(\frac1{q_c}-\frac14) (q_c-2)} \| \rho_\la f\|_{L^2}^{q_c-2}
  \\&\lesssim \sum_{j,\nu} \| A_{j,\ell_0}\sigma_\la Q^{\theta_0}_{j,\ell_0,\nu}\rho_\la f\|_{L^{q_c}}^2 \cdot
  \| A_{j,\ell_0} Q^{\theta_0}_{j,\ell_0,\nu}  \sigma_\la  \rho_\la f\|_{L^{q_c}}^{q_c-2} +\la^{1-\frac14 (q_c-2)}
  \\
  &\le C\bigl(\sum_{j,\nu} \| A_{j,\ell_0}\sigma_\la Q^{\theta_0}_{j,\ell_0,\nu}\rho_\la f\|_{L^{q_c}}^{q_c}\bigr)^{\frac2{q_c}}
  \bigl(\sum_{j,\nu} \| A_{j,\ell_0} Q^{\theta_0}_{j,\ell_0,\nu}  \sigma_\la  \rho_\la f\|_{L^{q_c}}^{q_c} \bigr)^{\frac{q_c-2}{q_c}}+C\la^{1-\frac14(q_c-2)}.
  \end{align*}
   By Young's inequality, the second to last term is bounded for any $\kappa>0$ by
 $$ C\Bigl[\,  \frac2{q_c}\kappa^{\frac{q_c}2}
  \sum_{j,\nu} \|A_{j,\ell_0}\sigma_\la Q^{\theta_0}_{j,\ell_0,\nu}\rho_\la f\|_{L^{q_c}}^{q_c}+
  \frac{q_c-2}{q_c} \kappa^{-\frac{q_c}{q_c-2}}
  \sum_{j,\nu} \| A_{j,\ell_0} Q^{\theta_0}_{j,\ell_0,\nu}\sigma_\la \rho_\la f\|_{L^{q_c}}^{q_c} \, \Bigr].
  $$
 
 If $\kappa$ is small enough the first term here is smaller than half of the left side of the preceding inequality.
 So, by an absorbing argument and the fact that $q_c>2$, we conclude that
 $$ \sum_{j,\nu} \| A_{j,\ell_0}\sigma_\la Q^{\theta_0}_{j,\ell_0,\nu}\rho_\la f\|_{L^{q_c}}^{q_c}
 \lesssim \sum_{j,\nu} \| A_{j,\ell_0} Q^{\theta_0}_{j,\ell_0,\nu}\sigma_\la \rho_\la f\|_{L^{q_c}}^{q_c}
 +\la^{1-}.
 $$
 If we next use \eqref{b.16}, \eqref{b.55}, followed by \eqref{b.10} we find that we can
 control the first term in the right as follows
 \begin{align*}
 \sum_{j,\nu} \| A_{j,\ell_0}\sigma_\la Q^{\theta_0}_{j,\ell_0,\nu}\rho_\la f\|_{L^{q_c}}^{q_c}
& \lesssim 
 \sum_{j,\nu} \|  Q^{\theta_0}_{j,\ell_0,\nu}  \sigma_\la   \rho_\la f\|_{L^{q_c}}^{q_c}
 \\
 &\lesssim
 \sum_{j,\nu}  \Bigl[ \, \|  Q^{\theta_0}_{j,\ell_0,\nu} \rho_\la f\|_{L^{q_c}}^{q_c}
 +
 \|  Q^{\theta_0}_{j,\ell_0,\nu} (I-\sigma_\la) \rho_\la f\|_{L^{q_c}}^{q_c}
 \Bigr]
 \\
 &\lesssim  \sum_{j,\nu}  \|  Q^{\theta_0}_{j,\ell_0,\nu} \rho_\la f\|_{L^{q_c}}^{q_c} + \|(I-\sigma_\la)\rho_\la f\|_{L^{q_c}}^{q_c}
 \\  &\lesssim  \sum_{j,\nu}  \|  Q^{\theta_0}_{j,\ell_0,\nu} \rho_\la f\|_{L^{q_c}}^{q_c} + \la \cdot (\log\la)^{-q_c}.
\end{align*}

If we combine \eqref{key} and the preceding two inequalities we conclude that we would obtain \eqref{b.31} and consequently
finish the proof of the estimates in Theorem~\ref{bgsp} if, for $T$ as in \eqref{b.7}, we could show that
\begin{equation}\label{b.77}
Uf(x,j,\nu)  =(Q^{\theta_0}_{j,\ell_0,\nu} \rho_\la f)(x),
\end{equation}
satisfies
$$
\| Uf \|_{\ell_j^{q_c}\ell_\nu^{q_c}L^{q_c}_x} \lesssim \la^{\frac1{q_c}} T^{-1/2} \|f\|_{L^2(M)}
$$
if all the sectional curvatures of $(M,g)$ are $\le-\kappa_0^2$ for some $\kappa_0>0$ as well
as
$$
\| Uf \|_{\ell_j^{q_c}\ell_\nu^{q_c}L^{q_c}_x} \lesssim \bigl(\la T^{-1}\bigr)^{\frac1{q_c}}  \|f\|_{L^2(M)} $$
if  all the sectional curvatures of  $(M,g)$ are nonpositive.  Equivalently, this would be a consequence
of the following
\begin{multline}\label{b.78}
\| UU^* \|_{\ell_{j'}^{q_c'}\ell_{\nu'}^{q_c'} L^{q_c'}_x \to \ell^{q_c}_j \ell^{q_c}_\nu L^{q_c}_x}
\\
\lesssim
\begin{cases}
\la^{2/q_c}T^{-1}, 
\, \, \text{if all the sectional curvatures of } 
M \, \text{are } \le -\kappa_0^2, \,
\text{some }\kappa_0>0,
\\ \bigl(\la T^{-1}\bigr)^{2/q_c}
\, \, \text{if all the sectional curvatures of } 
M \, \text{are nonpositive}.
\end{cases}
\end{multline}

To prove the large height $A_+$ estimates \eqref{b.31} we split $\rho_\la\rho_\la^*=L_\la+G_\la$ as in 
\eqref{b.33.3}.  To prove \eqref{b.78}, we require an additional dyadic decomposition, as well as taking into account
the second microlocal decomposition afforded by the $\{Q^{\theta_0}_{j,\ell,\nu}\}$.  To obtain this dyadic decomposition, we
fix a Littlewood-Paley bump function $\beta\in C^\infty_0((1/2,2))$ satisfying $\sum_{k=-\infty}^\infty \beta(s/2^k)=1$,
$s>0$.  If we let $\beta_0(t)=1-\sum_{k=1}^\infty \beta(|t|/2^k)$, then $\beta_0\in C^\infty_0(\R)$ equals one near the 
origin, and so plays the role of $a(t)$ in \eqref{b.33.2}.  So, analogous to \eqref{b.33.2}, we set
$$L_{\la,T}=\frac1{2\pi}\int_{-\infty}^\infty \beta_0(t)
\, T^{-1}\Hat \Psi(t/T) \, e^{it\la} \, e^{-itP} \, dt,
$$
with, as in Lemma~\ref{Gl}, $\Psi =|\rho|^2$.

If then $G_\la=G_{\la,T}$ is as in \eqref{b.33} with $a=\beta_0$, we use the dyadic decomposition given by
\begin{equation}\label{b.79}
G_{\la,T,N}=\frac1{2\pi} \int_{-\infty}^\infty \beta(|t|/N) \, T^{-1}\Hat \Psi (t/T) \, e^{it\la} \, e^{-itP} \, dt,
\end{equation}
so that, if we consider the resulting dyadic sum, we have
\begin{equation}\label{b.80}
G_\la =\sum_{1\le 2^k=N\lesssim T} G_{\la,T,N}.
\end{equation} 
Then, if we set,
\begin{equation}\label{b.81}
\bigl(W_N F\bigr)(x, j ,\nu)=\sum_{j',\nu'} Q^{\theta_0}_{j,\ell_0,\nu}
\circ G_{\la,T,N} \circ
\bigl(Q^{\theta_0}_{j',\ell_0,\nu'}\bigr)^* F(x,j',\nu')
\end{equation}
by \eqref{b.33.3}, \eqref{b.77} and \eqref{b.80} we have
\begin{multline}\label{b.82}
\bigl(U U^*F)(x,j,\nu)=\sum_{j',\nu'}
\bigl( \, Q^{\theta_0}_{j,\ell_0,\nu}\circ L_{\la,T}\circ (Q^{\theta_0}_{j',\ell_0,\nu'})^*\bigr)F(x ,j',\nu')
\\
+\sum_{1\le N=2^k\lesssim T}\bigl(W_NF\bigr)(x, j, \nu).
\end{multline}

The operator $L_{\la,T}$ satisfies the bounds in \eqref{b.33.4}.  If we use this along with the first inequality
in \eqref{b.555} for $q=q_c$ followed by this bound and then the second inequality for $p=q_c'$ in 
\eqref{b.55} we obtain
$$
\Bigl\|
\sum_{j',\nu'}
\bigl( \, Q^{\theta_0}_{j,\ell_0,\nu}\circ L_{\la,T}\circ (Q^{\theta_0}_{j',\ell_0,\nu'})^*\bigr)F(\, \cdot\, ,j',\nu')
\Bigr\|_{\ell^{q_c}_j\ell^{q_c}_\nu L^{q_c}_x} \le C\la^{2/q_c}T^{-1}
\|F\|_{\ell^{q_c'}_{j'}\ell^{q_c'}_{\nu'}L^{q_c'}_x},
$$
which agrees with the bounds in \eqref{b.78} for strictly negative sectional curvatures and is better than the bounds
posited for nonpositive curvature.

To obtain the desired bounds for the last term in \eqref{b.82}, we shall require the following result which plays here
the role of Lemma~\ref{Gl}.

\begin{lemma}\label{G2}  Let $G_{\la,T,N}$ be as in \eqref{b.79}.  Then for $\la\gg 1$ and $\theta_0=\la^{-1/8}$ we have the uniform bounds
\begin{equation}\label{b.83}
\bigl\| Q^{\theta_0}_{j,\ell_0,\nu} G_{\la,T,N}
\bigr\|_{L^1(M)\to L^\infty(M)} \le C_M T^{-1} \la^{\frac{n-1}2}
N^{1-\frac{n-1}2}, \, \, N\ge 1, 
\end{equation}
if $(M,g)$ is a complete manifold of bounded geometry all of whose sectional curvatures are nonpositive and
$T=c_0\log\la$ is fixed with $c_0=c_0(M)>0$ sufficiently small.  Moreover, if we assume that all of the sectional curvatures
are $\le -\kappa_0^2$, some $\kappa_0>0$, then we have the uniform bounds
\begin{equation}\label{b.84}
\bigl\| Q^{\theta_0}_{j,\ell_0,\nu} G_{\la,T,N}
\bigr\|_{L^1(M)\to L^\infty(M)}\le C_M T^{-1} \la^{\frac{n-1}2}
N^{-m}, \, \, \forall \, m\in {\mathbb N}.
\end{equation}
\end{lemma}

Like those in Lemma~\ref{Gl}, these two bounds follow from kernel estimates which we shall obtain at the end of this section.

Let us now see how we can use this lemma to see that the last term in the right side of \eqref{b.82} satisfies the
bounds in \eqref{b.78}.

We first notice that, by \eqref{b.79}, the operators in \eqref{b.79} have $O(T^{-1}N)$  $L^2(M)\to L^2(M)$ operator norms.
Thus, by \eqref{b.555} for $q=p=2$ (almost orthogonality), we have, by \eqref{b.81}
\begin{equation}\label{b.85}
\bigl\|
W_N 
\bigr\|_{\ell^2_{j'}\ell^2_{\nu'}L^2_x\to \ell^2_j\ell^2_\nu L^2_x}=O(T^{-1}N).
\end{equation}

If we use the second inequality in \eqref{b.555} for $p=1$ along with \eqref{b.81}
and \eqref{b.83}, we also obtain that for $T$ as above and $N=2^k\ge1$
\begin{equation}\label{b.86}
\|W_N\|_{\ell^1_{j'}\ell^1_{\nu'}L^1_x \to \ell^\infty_j \ell^\infty_\nu L^\infty_x}=
O(T^{-1}\la^{\frac{n-1}2}N^{1-\frac{n-1}2})
\end{equation}
if all of the sectional curvatures of $M$ are nonpositive, as well as
\begin{equation}\label{b.87}
\|W_N\|_{\ell^1_{j'}\ell^1_{\nu'}L^1_x \to \ell^\infty_j \ell^\infty_\nu L^\infty_x}=
O(T^{-1}\la^{\frac{n-1}2}N^{-m}), \, \,  \forall \, m\in {\mathbb N},
\end{equation}
if all of the principal curvatures are pinched below zero as in \eqref{b.84}.

If we interpolate between \eqref{b.85} and \eqref{b.87} we obtain
\begin{equation}\label{b.88} \|W_N\|_{\ell^{q_c'}_{j'}\ell^{q_c'}_{\nu'}L^{q_c'}_x \to \ell^{q_c}_j \ell^{q_c}_\nu L^{q_c}_x}=
O(T^{-1}\la^{2/q_c}N^{1-m}), \, \,  \forall \, m\in {\mathbb N},
\end{equation}
if all of the sectional curvatures of $M$ are $\le -\kappa_0^2$, some $\kappa_0>0$.  As a result, we can estimate the
last term in \eqref{b.82} as follows
\begin{multline}\label{b.89}
\Bigl\|
\sum_{1\le N=2^k\lesssim T}W_NF\Bigr\|_{\ell_j^{q_c}\ell^{q_c}_\nu L^{q_c}_x}\lesssim
T^{-1}\la^{2/q_c}\sum_{1=N\lesssim T}N^{-1}
\|F\|_{\ell^{q_c'}_{j'}\ell^{q_c'}_{\nu'}L^{q_c'}_x}
\\
\lesssim T^{-1}\la^{2/q_c}
\|F\|_{\ell^{q_c'}_{j'}\ell^{q_c'}_{\nu'}L^{q_c'}_x},
\end{multline}
and so this term also satisfies the bounds in \eqref{b.78}.

If we merely assume that all of the sectional curvatures are nonpositive, then \eqref{b.85}, \eqref{b.86} and interpolation yield
\begin{equation*}
\|W_N\|_{\ell^{q_c'}_{j'}\ell^{q_c'}_{\nu'}L^{q_c'}_x \to \ell^{q_c}_j \ell^{q_c}_\nu L^{q_c}_x}=
O(T^{-1}\la^{2/q_c}N^{1-\frac{n-1}{n+1}}), \, \,  \forall \, m\in {\mathbb N},
\end{equation*}
Since $\tfrac{n-1}{n+1}=\tfrac2{q_c}$, we therefore obtain
\begin{equation}\label{b.90}
\Bigl\|
\sum_{1\le N=2^k\lesssim T}W_NF\Bigr\|_{\ell_j^{q_c}\ell^{q_c}_\nu L^{q_c}_x} \lesssim
(\la T^{-1})^{2/q_c}\|F\|_{\ell^{q_c'}_{j'}\ell^{q_c'}_{\nu'}L^{q_c'}_x},
\end{equation}
as desired under this curvature assumption.

Inequalities \eqref{b.89}, \eqref{b.90} along with the earlier bounds for the first term in \eqref{b.82} yield \eqref{b.78}.
As a result, except for needing to prove Lemmas \ref{Gl} and \ref{G2}, the proof of the bounds in
Theorem~\ref{bgsp} for $q=q_c$ is complete.

Since we also earlier obtained the bounds for $q\in (q_c,\infty]$, it only remains to obtain the bounds for
$q\in (2,q_c)$.  If the curvatures are assumed to be nonpositive, then the bounds in \eqref{i.6} for these exponents
just follow from interpolating between the bounds for $q=q_c$ and the trivial $L^2$-estimate.  So, to complete the proof
of the Theorem, by \eqref{b.10} it suffices to show that for $T$ as above we have
$$\| \sigma_\la\rho_\la\|_{L^2(M)\to L^q(M)}=O(T^{-1/2}\la^{\mu(q)}), \, \, q\in (2,q_c)$$
when all the sectional curvatures of $M$ are pinched below zero and $T$ is as above.  By interpolating with the
$q=q_c$ estimate that we just obtained, it suffices to show that, under these assumptions, we have
\begin{equation}\label{b.91} 
\|\sigma_\la \rho_\la \|_{L^2(M)\to L^q(M)}=O(T^{-1/2} \la^{\mu(q)}), \, \, q\in (2,\tfrac{2(n+2)}n].
\end{equation}
This just follows from the above arguments which gave us the bounds in \eqref{b.32} for $q=q_c$ under this curvature assumption
 if we use \eqref{key2} in place of \eqref{key}.   The argument is  a bit simpler since the norms in the left side of 
 \eqref{key2} are over $M$.   So, we do not need for this case to split $M=A_-\cup A_+$ to handle the exponents in \eqref{b.91}.
\end{proof}


\bigskip

\bigskip

\noindent{\bf 3.2.  Log-scale Strichartz estimates}

In this section, let us see how we can follow the ideas in the previous section to adapt the proofs for the compact manifold case treated in \cite{blair2023strichartz,HSst}
to prove Theorem~\ref{bgst}.

To align with the numerology in the previous section on the spectral projection estimates, throughout this section,
we shall always assume the manifold is $(n-1)$-dimensional. 
Additionally, we will repeatedly use symbols such as $\sigma_\la$ and $Q_\nu^\theta$; however, it is important to note that they represent different operators in this section.

To start, let us fix
\begin{equation}\label{s.1}
\eta\in C^\infty_0((-1,1)) \quad
\text{with } \, \, \eta(t)=1, \, \, \, |t| \le 1/2.
\end{equation}
We shall consider the dyadic time-localized dilated Schr\"odinger operators
\begin{equation}\label{s.2} 
S_\la =\eta(t/T) e^{-it\la^{-1}\Delta_g} \beta(P/\la),
\end{equation}
where $\beta\in C^\infty_0((1/2,2))$ as in \eqref{i.2} and $T=c_0\log\la$ for some small constant $c_0$ we shall specify later. By changing scale in time, to prove Theorem~\ref{bgst}
 it suffices to show that if all of the sectional curvatures of $M$ are nonpositive then for $(p,q)$ satisfying \eqref{i.3} we have
\begin{equation}\label{s.3}
\| S_\la f\|_{L^p_tL^q_x(M\times [0,T])}\le C\la^{\frac1{p}}  \, \|f\|_{L^2(M)}, \quad
\text{if } \, \, T=c_0\log\la.
\end{equation}
Note that if we replace $[0, T]$ by $[0,1]$, then by using the analog of \eqref{i.2} for intervals $[0,\la^{-1}]$ along with a rescaling argument, we have for any complete manifold of bounded geometry
\begin{equation}\label{s.4}
\| S_\la f\|_{L^p_tL^q_x(M\times [0,1])}\le C\la^{\frac1{p}}  \, \|f\|_{L^2(M)}.
\end{equation}

Now we shall introduce the auxiliary operators that allow us to use bilinear
techniques. Let $\rho\in {\mathcal S}(\R)$ satisfying \eqref{b.3}, we define the local operators  
\begin{equation}\label{s.5}
\sigma_\la = \left(\rho\bigl(\la^{1/2}|D_t|^{1/2}-P\bigr)+\rho\bigl(\la^{1/2}|D_t|^{1/2}+P\bigr)\right) \, \tilde \beta(D_t/\la),
\end{equation}
where
\begin{equation}\label{s.6}
\tilde \beta\in C^\infty_0((1/8,8)) \quad \text{satisfies } \, \,
\tilde \beta=1 \, \, 
\text{on } \, \, [1/6,6].
\end{equation}
Note that by  by Euler’s formula,
\begin{equation}\label{s.7}
 \sigma_\la(x,t;y,s)
=\frac{1}{2\pi^2}
\iint e^{i(t-s)\tau}e^{ir\la^{1/2}\tau^{1/2}}
\tilde \beta(\tau/\la) \, \Hat \rho(r)
\, \cos(rP)(x,y) \, dr d\tau.
\end{equation}  
Thus by the support properties of $\Hat \rho$ as in \eqref{b.3} and finite propagation speed, we have 
\begin{equation}\label{s.8}
\sigma_\la(x,t;y,s)=0 \quad \text{if } \, \, d_g(x,y)>r, \, \, 
r= \delta_1(1+\delta_2)<1.
\end{equation}

For admissible pairs $(p,q)$ as in \eqref{i.3}, the local
operators  satisfy
\begin{equation}\label{s.9}
\|(I-\sigma_\la)\circ S_\la f\|_{L^p_tL^q_x(M\times [0,T])}\le CT^{\frac1{p}-\frac12}\la^{\frac1{p}}\|f\|_2,
\end{equation}
This is a straightforward generalization of Lemma 2.2 in \cite{blair2023strichartz} to all complete manifold of bounded geometry and all
 pairs $(p, q)$ satisfying \eqref{i.3}. The proof relies on the local dyadic Strichartz estimates \eqref{s.4} along with the spectral theorem and functional calculus for multiplier operators. We skip the details here and refer to \cite{blair2023strichartz}
for more details.

For each fixed $j$, if we use the microlocal pseudodifferential operators $A_{j,\ell}$ defined in \eqref{b.14}, we can write
\begin{equation}\label{s.11}
\psi_j(x) (\sigma_\la F)(t,x)=\sum_{\ell=1}^K \bigl(A_{j,\ell} \circ\sigma_\la\bigr)(F)(t,x)+R_jF(t,x).
\end{equation}
where as in \eqref{b.13}
\begin{equation}\label{s.12}
A_{j,\ell}(x,y), \, \, R_j(x, t;y, s)=0, \quad \text{if } \, x\notin B(x_j,\delta) \, \,
\text{or } \, \, y\notin B(x_j,3\delta/2).
\end{equation}
As before, by fixing $c_1>0$ small enough in the symbol of $A_{j,\ell}$ operators,
we have the uniform
bounds
\begin{equation}\label{s.13}
R_j(x, t; y, s)=O(\la^{-N}), \,\, N=1,2,...
\end{equation}

As in the previous section, the microlocal operators $A_{j,\ell}$ will be useful in the local harmonic analysis arguments we shall describe later.
Note also that by \eqref{b.00}, \eqref{s.12} and \eqref{s.13} we also have
\begin{equation}\label{s.14}
\|RF\|_{L^p_tL^q_x(M\times [0,T])}\le C_{p,q,M} \, \|F\|_{L^2_{t,x}(M\times\R)}, \, \, 1\le p\le q\le \infty, \, \,
\text{if }, \, \, R=\sum_j R_j.
\end{equation}

Note that for fixed $\ell_0$, if we let $A= \sum_{j}A_{j,\ell_0}$ as in \eqref{b.17}, in view of \eqref{b.1}, \eqref{s.9}, \eqref{s.11} and \eqref{s.14}, in order
to prove \eqref{s.3}, it suffices to
prove that if all the sectional curvatures of
$M$ are nonpositive
\begin{equation}\label{s.15}
\| A \sigma_\la S_\la f\|_{L^p_tL^q_x(M\times [0,T])}\le C\la^{\frac1{p}}\|f\|_2.
\end{equation}
And if we consider the vector-valued operators
\begin{equation}\label{s.16}
\A H(x,t)=(A_{1,\ell_0}H(x,t), \, A_{2,\ell_0}H(x,t),\dots)
\end{equation}
and argue as in \eqref{b.211}-\eqref{b.23}, \eqref{s.15} would be a consequence of 
\begin{equation}\label{s.17}
\| \A \sigma_\la S_\la f\|_{L^p_tL^q_x\ell^q_j(\N\times M\times [0,T])}\le C\la^{\frac1{p}}\|f\|_2.
\end{equation}
The operators $\A \sigma_\la S_\la$ play the role of the $\tilde S_\la$ operators in \cite{blair2023strichartz} and \cite{HSst}. As in the previous section,  
 the vector-valued approach will allow us to only have to carry out the local bilinear harmonic analysis
in individual coordinate patches coming from the geodesic normal coordinates in the balls $B(x_j,2\delta)$.

Now let's set up the height decomposition that we shall use, throughout this section, we assume
\begin{equation}\label{s.norm}
\|f\|_{L^2(M)}=1.
\end{equation}
Let us
define vector-valued sets
\begin{equation}\label{s.18}
\begin{split}
A_+&= \{(x, t, j)\in M\times [0,T]\times \N: \, |(\A \sigma_\la S_\la f)
(x,t,j)|\ge \la^{\frac{n-1}4+\e_1}\}
\\
A_-&=\{(x, t, j)\in M\times [0,T]\times \N: \, |(\A \sigma_\la S_\la f)
(x, t, j)|< \la^{\frac{n-1}4+\e_1}\}.
\end{split}
\end{equation}
Recall here that
\begin{equation}\label{s.18a}
(\A \sigma_\la S_\la f)(x, t, j)=A_{j,\ell_0}
\sigma_\la S_\la f(x, t).
\end{equation}

Due to the numerology of the powers of $\la$ arising,
the splitting occurs at height $\la^{\frac{n-1}4+\e_1}$, with $\frac{n-1}{4}$ same as the previous section. Here $\e_1>0$ is a small constant that may depend on the dimension $n-1$. As we shall see later, we can take $\e_1=\frac 1{100}$ for $n-1\ge 3$ while  for $n-1=2$, the choice of $\e_1$ depends on the exponent $q$ for {\em {admissible}} pairs $(p,q)$, with $\e_1\rightarrow 0$ as $q\rightarrow \infty$.

In order to prove \eqref{s.17} on the set $A_+$, we shall require the following lemma
\begin{lemma}\label{kerprop} Let $S_{t,\la}$ denote the operator
$\eta(t/T) \beta(P/\la)e^{-it\la^{-1}\Delta_g}$.
Then if $M$ has nonpositive sectional curvatures and $T=c_0\log\la$ with 
$c_0=c_0(M)>0$ sufficiently small, we have for $\la\gg 1$
\begin{equation}\label{j1}
\|S_{t,\la}S^*_{s,\la}\|_{L^1(M)\to L^\infty(M)}\le C\la^{\frac{n-1}2}|t-s|^{-\frac{n-1}2} \exp(C_M|t-s|).
\end{equation}
\end{lemma}
We shall postpone the proof of this lemma until the end of this section and first see how we can use it to prove  \eqref{s.17} on the set $A_+$.

\begin{proof}[Proof of \eqref{s.17} on the set $A_+$.]
We first note that, by \eqref{s.9}, \eqref{s.16} and \eqref{s.norm},
we have
$$ \| \A \sigma_\la S_\la f\|_{L^p_tL^q_x\ell^q_j(A_+)} \le \| \A S_\la f\|_{L^p_tL^q_x\ell^q_j(A_+)} + CT^{\frac1{p}-\frac12}\la^{\frac1{p}}.
$$
Since $p\ge2$ for $(p,q)$ as in \eqref{i.3},  \eqref{s.17} would follow from
\begin{equation}\label{high'}
\| \A S_\la f\|_{L^p_tL^q_x\ell^q_j(A_+)}\le C\la^{\frac1{p}}+\tfrac12 \|\A \sigma_\la S_\la f\|_{L^p_tL^q_x\ell^q_j(A_+)}.
\end{equation}

To prove this, similar to what was done in \cite{HSst}, we  choose $g=g(x,t,j)$ such that
\begin{multline}\label{s.20}
    \|g\|_{L^{p'}_tL^{q'}_x\ell^{q'}_j(A_+)}=1
\,\, \text{and } \, \,
\|\A S_\la f\|_{L^p_tL^q_x\ell^q_j(A_+)} \\
=\sum_j\iint \A S_\la f(x,t,j)\cdot
\overline{\bigl(\1_{A_+}\cdot g (x,t,j)\bigr)} \, dx dt.
\end{multline}
Then, since we are assuming that $\|f\|_2=1$, by
the Schwarz inequality
\begin{align}\label{s.21}
\|\A S_\la f\|^2_{L^p_tL^q_x\ell^q_j(A_+)}
&= \Bigl( \, \int f(x) \, \cdot \,  
\overline{\bigl(S_\la^*\A^*\bigr)(\1_{A_+}
\cdot g\bigr)(x)} \, dx \, \Bigr)^2
\\
&\le \int |S^*_\la \A^*(\1_{A_+}\cdot g)(x)|^2 \, dx
\notag
\\
&=\sum_j\iint \bigl(\A S_\la S^*_\la \A^*\bigr)(\1_{A_+}\cdot
g)(x,t,j) \, \overline{(\1_{A_+}\cdot
g)(x,t,j)} \, dx dt
\notag
\\
&=\sum_j\iint \bigl(\A\circ L_\la\circ \A^*\bigr)(\1_{A_+}\cdot
g)(x,t,j) \, \overline{(\1_{A_+}\cdot
g)(x,t,j)} \, dx dt \notag
\\
&\qquad+\sum_j\iint \bigl(\A\circ G_\la\circ \A^*\bigr)(\1_{A_+}\cdot
g)(x,t,j) \, \overline{(\1_{A_+}\cdot
g)(x,t,j)} \, dx dt \notag
\\
&=I + II, \notag
\end{align}
where $L_\la$ is the integral operator with kernel equaling that of $S_{t,\la}S^*_{s,\la}$ if $|t-s|\le 1$ and 
$0$ otherwise, i.e,
\begin{equation}\label{ok}
L_\la(x,t;y,s)=
\begin{cases}S_{t,\la}S^*_{s,\la}(x,y), \, \,
\text{if } \, \, |t-s|\le 1,
\\
0 \, \, \, \text{otherwise}.
\end{cases}
\end{equation}
Since $p\ge 2$, it is straightforward to see that \eqref{s.4} and \eqref{j1}
yield
\begin{equation}\label{s.22}
\|L_\la \|_{L^{p'}_tL^{q'}_x \to L^p_tL^q_x}
=O(\la^{\frac 2p}).
\end{equation}

If we use this, along with H\"older's inequality, \eqref{b.23} and 
\eqref{s.20}, we obtain for the term $I$ in \eqref{s.21}
\begin{align}\label{k301}
|I|&\le \|\A L_\la \A^*(\1_{A_+}\cdot g)\|_{ L^p_tL^q_x\ell^q_j}
\cdot \| \1_{A_+}\cdot g \|_{L^{p'}_tL^{q'}_x}
\\
&\lesssim \|L_\la \A^*(\1_{A_+}\cdot g)\|_{L^p_tL^q_x\ell^{q}_j} 
\cdot
\|\1_{A_+}\cdot g \|_{L^{p'}_tL^{q'}_x\ell^{q'}_j}
\notag
\\
&\lesssim \la^{\frac2{p}} \|\A^*(\1_{A_+}\cdot g) \|_{L^{p'}_tL^{q'}_x\ell^{q'}_j}
\cdot
\|\1_{A_+}\cdot g \|_{L^{p'}_tL^{q'}_x\ell^{q'}_j}
\notag
\\
&\lesssim \la^{\frac2{p}}\|g\|^2_{L^{p'}_tL^{q'}_x\ell^{q'}_j(A_+)}
=\la^{\frac2{p}}. \notag
\end{align}

To estimate $II$, note that 
if we choose $c_0$ small enough
so that if $C_M$ is the constant
in \eqref{j1}
$$\exp(2C_MT)\le \la^{\e_1}, \quad \text{if  }
\,
T=c_0\log\la \, \, \text{and } \, \, \la \gg 1.
$$
Then, since $\eta(t)=0$ for $|t|\ge1$, it follows
 \eqref{j1} that
$$\|G_\la \|_{L^1(M\times \R) 
\to L^\infty (M\times \R)}\le C
\la^{\frac{n-1}2+{\e_1}}.
$$
As a result, by H\"older's inequality \eqref{b.23} and 
\eqref{s.20}, we can repeat the arguments
to estimate $I$ to see that
\begin{multline}
    |II|\le C\la^{\frac{n-1}2}\la^{\e_1}
\| \1_{A_+}\cdot g\|^2_{L^1_{t,x}\ell^1_j}
\le C \la^{\frac{n-1}2}\la^{\e_1}
\|g\|^2_{L^{p'}_tL^{q'}_x\ell^{q'}_j} \cdot
\|\1_{A_+}\|_{L^{p}_tL^{q}_x\ell^q_j}^2 \\
=C \la^{\frac{n-1}2}\la^{\e_1}
\|\1_{A_+}\|_{L^{p}_tL^{q}_x\ell^q_j}^2.
\end{multline}

If we recall the definition of $A_+$ in
\eqref{s.18}, we can estimate the last
factor:
$$\|\1_{A_+}\|_{L^{p}_tL^{q}_x\ell^q_j}^2
\le \bigl(\la^{\frac{n-1}4+\e_1}\bigr)^{-2}
\|\A\sigma_\la S_\la f\|^2_{{L^{p}_tL^{q}_x}\ell^q_j(A_+)}.$$
Therefore, 
$$|II|\lesssim \la^{-\e_1}
\|\A\sigma_\la S_\la f\|^2_{{L^{p}_tL^{q}_x\ell^q_j}(A_+)}
\le \bigl(\tfrac12 \|\A\sigma_\la S_\la f\|_{{L^{p}_tL^{q}_x\ell^q_j}(A_+)}
\bigr)^2,
$$
assuming, as we may, that $\la$ is large enough.

If we combine this bound with the earlier one,
\eqref{k301} for $I$, we conclude that
\eqref{high'} is valid, which completes the 
proof of \eqref{s.17} on the set $A_+$. 
\end{proof}

Next, we shall give the proof of \eqref{s.17} on the set $A_-$. This requires the use of  local bilinear harmonic
analysis. Following the approach in the previous section, in view of \eqref{s.11}, it suffices to carry out the analysis in geodesic normal coordinates of each individual balls $B(x_j,2\delta)$, since our assumption of bounded geometry ensures
bounded transition maps and uniform bounds on derivatives of the metric. Also note that since the case $p=\infty, q=2$ in \eqref{s.3} simply follows from spectral theorem, to prove  \eqref{s.17} on the set $A_-$, for the remaining of this section, we shall assume
\begin{equation}\label{p}
(p, q)=(2,\tfrac{2(n-1)}{n-3}) \, \,
\text{if } \, n\ge 4, \, \, \, 
\text{or } \, (n-1)(\tfrac12-\tfrac1q)=\tfrac2p, \,\,4< q<\infty \, \, 
\text{if } \, \, n=3.
\end{equation}
The condition $q>4$ is equivalent to $q>p$ when $n-1=2$,
this will allow us to simplify some of the calculations to follow.

To set up
the second microlocalization needed for the Schr\"odinger setting, let us fix $j$ in \eqref{b.12}, as well as $\ell_0\in \{1,\dots,K\}$ and consider the resulting
pseudodifferential cutoff, $A_{j,\ell_0}$, which is a summand in \eqref{s.11}.  Its symbol then satisfies
the conditions in \eqref{aas}.  The resulting geodesic normal coordinates on $B(x_j,2\delta)$
vanish at $x_j$.  We may also
assume that $\xi_{j,\ell_0}=(0,\dots,0,1)$.  Since we are fixing
$j$ and $\ell_0$ for now, analogous to \cite{HSst}, let us simplify the notation a bit by letting
\begin{equation}\label{s.23}
\tilde \sigma_\la =A_{j,\ell_0}\sigma_\la,
\end{equation}

The $Q_\nu^\theta$ operators constructed in the last section provide “directional” microlocalization. We 
also need a “height” localization since the characteristics of the symbols of our scaled Schr\"odinger operators lie on
paraboloids. The variable coefficient operators that we shall use  are analogs of ones that are used in the study of Fourier restriction
problems involving paraboloids.

To construct these, choose $b\in C^\infty_0(\R)$ supported in $|s|\le1$  satisfying $\sum_{-\infty}^\infty b(s-\ell)\equiv 1$.
 We then define the compound symbols $Q^\theta_\ell
=Q^\theta_{j,\ell_0,\ell}$ and associated ``height'' operators by
\begin{multline}\label{s.24}
Q^\theta_\ell(x,y,\xi)= \tilde{\tilde \psi}(y)b(\theta^{-1}\la^{-1}(p(x,\xi) -\la\kappa^\theta_\ell)), \, \,
\kappa^\theta_\ell=1+\theta\ell, \quad
|\ell|\lesssim \theta^{-1}, \\ \text{and}\,\,\,
Q^\theta_\ell h(x) = (2\pi)^{-(n-1)}
\iint e^{i(x-y)\cdot \xi}   Q^\theta_\ell(x,y,\xi)\,  h(y) \, d\xi dy.
\end{multline}
Here $\tilde{\tilde \psi} \in C^\infty_0$ is supported in 
$|x|<2\delta$ which equals one when $|x|\le 3\delta/2$, as defined in \eqref{b.50}.

Unlike in the earlier works \cite{HSst,blair2023strichartz}, the height operators here are defined in local coordinates and have cutoffs in $y$ variable in order to avoid issues at infinity since $M$ is not assumed to be compact. In the compact case, the analogous height operator can be simply defined using spectral multipliers, see e.g., \cite{blair2023strichartz, HSst}.
 These operators microlocalize $p(x,\xi)$ to 
intervals of size $\approx \theta\la$ about ``heights''
$\la\kappa^\theta_\ell\approx \la$. By a simple integration by parts argument, if $Q^\theta_\ell(x,y)$ is the kernel of this operator then
\begin{equation}\label{s.25}
Q^\theta_\ell(x,y)=O(\la^{-N}) \, \forall \, N, \quad
\text{if } \, d_g(x,y)\ge C_0\theta,
\end{equation}
 for a fixed constant $C_0$ if $\theta\in [\la^{-1/2+\e},1]$ with $\e>0$.

For $\nu=(\nu', \ell)=(\theta k, \theta \ell)\in \theta{\mathbb Z}^{2(n-2)+1}$ we now define the cutoffs that we shall use:
\begin{equation}\label{s.26}
Q^\theta_{\nu}=Q^\theta_{\nu'}\circ Q^\theta_\ell.
\end{equation}
where $Q^\theta_{\nu'}$ are the directional microlocalization operators defined in \eqref{b.50}. Both $Q^\theta_{\nu'}$ and $Q^\theta_\ell$ operators here depend on our fixed $j,\ell_0$, and as in \eqref{b.51}, due to the way they are constructed, for small enough $\delta_0>0$ the principle symbol $q^\theta_\nu(x,y,\xi)$ of the $Q_\nu^\theta$ operators satisfy
\begin{multline}\label{s.27}
q^\theta_\nu(x,y,\xi)=q^\theta_\nu(z,y,\eta), \, (z,\eta)=\Phi_t(x,\xi), 
\\ \text{if } \, \text{dist }((x,\xi), \, \text{supp }A_{j,\ell_0})\le \delta_0
\, \text{and } \, \, |t|\le 2\delta_0.
%
\end{multline}

The symbol of $Q_\nu^\theta$ operators  in \eqref{s.26} vanishes when either $x$ or $y$ is outside the $2\delta$-ball about
the origin in our coordinates for $\Omega$.  By \eqref{s.7} \eqref{s.12} and \eqref{s.23}, we can fix $\delta_1$ in \eqref{b.3} small enough so that
we also have, analogous to (2.39) in \cite{HSst},
\begin{multline}\label{s.28}
\tilde \sigma_\la = \sum_\nu \tilde \sigma_\la Q^{\theta_0}_\nu + R, \quad R=R_{\la,j,\ell_0},  \, \, 
\, \,  \tilde \sigma_\la =A_{j,\ell_0}\sigma_\la,  
\\
\text{where } \, \, R(x, t; y, s)=O(\la^{-N}), \, \forall \,\,N \,\,\, \\
\text{and } R(x,t;y,s)=0, \, \, \text{if } x\notin B(x_j,2\delta) \, \,
\text{or } y\notin B(x_j,2\delta),
\end{multline}
with bounds for the remainder kernel independent of $j$. Here unlike in the previous section we take $\theta_0=\la^{-\e_0}$ for some small constant $\e_0$
that we shall specify later, the choice of $\e_0$ depend on the dimension $n-1$.

Let us now point out straightforward but useful properties of our operators.  First, by \eqref{s.12}, \eqref{s.28}
and the support properties of $\tilde \psi$, $\tilde{\tilde \psi}$, we have
\begin{multline}\label{s.29}
\tilde \sigma_\la Q^{\theta_0}_\nu H=\1_{B(x_j,2\delta)} \cdot \tilde \sigma_\la Q^{\theta_0}_\nu\bigl(
\1_{B(x_j,2\delta)} \cdot H\bigr), \, \, Q^{\theta_0}_\nu =Q^{\theta_0}_{j,\ell_0,\nu}
\\
\text{and } \, \, RH=\1_{B(x_j,2\delta)}\cdot R \bigl(
\1_{B(x_j,2\delta)} \cdot H\bigr), \, \, R=R_{\la,j,\ell_0}.
\end{multline}

Also, we have the uniform bounds
\begin{equation}\label{s.30}
\begin{split}
\|Q^{\theta_0}_\nu h\|_{\ell^q_\nu L^q(M)}\lesssim \|h\|_{L^q(M)}, \, \, 2\le q\le \infty
\\
\bigl\| \sum_{\nu'}(Q^{\theta_0}_\nu)^*H(\nu', \, \cdot \, )\|_{L^p(M)}\lesssim 
\|H\|_{\ell^p_{\nu'}L^p(M)}, \, \, 1\le p\le 2.
\end{split}
\end{equation}
The second estimate follows via duality from the first.  The first one is the analog of (2.42) in \cite{HSst}.  By interpolation,
one just needs to verify that the estimate holds for the two endpoints, $p=2$ and $p=\infty$.  The former follows via an almost
orthogonality argument, and the latter from the fact that 
for each $x$ the symbols vanish outside of cubes of sidelength $\theta \la$ and 
$| \partial^\gamma_\xi Q^\theta_\nu(x,y,\xi)|=O((\la\theta)^{-|\gamma|})$, thus it is not hard to show we have the uniform bounds
$$\sup_{x\in B(x_j,2\delta)}\int_{B(x_j,2\delta)}|Q^{\theta_0}_\nu(x,y)| \, dy\le C.$$

Note that if we use \eqref{s.30}, the support properties of the $Q_\nu^{\theta_0}$ operators and the finite overlap of the balls $\{B(x_j,2\delta)\}$ we obtain
for our fixed $\ell_0=1,\dots,K$
\begin{equation}\label{s.31}
\begin{split}
\bigl(\sum_{j,\nu}\|Q^{\theta_0}_{j,\ell_0,\nu} h\|_{ L^q(M)}^q\bigr)^{1/q}\lesssim \|h\|_{L^q(M)}, \, \, 2\le q\le \infty
\\
\bigl\| \sum_{j',\nu'}(Q^{\theta_0}_{j',\ell_0,\nu'})^*H(\nu',j', \, \cdot \, )\|_{L^p(M)}\lesssim 
\|H\|_{\ell^p_{\nu'}L^p(M)}, \, \, 1\le p\le 2.
\end{split}
\end{equation}

In addition to this inequality and \eqref{s.9}, we shall also require the following commutator bounds
\begin{equation}\label{comms}
\bigl\| (A_{j,\ell_0}\sigma_\la Q^{\theta_0}_{j,\ell_0,\nu}-A_{j,\ell_0}Q^{\theta_0}_{j,\ell_0,\nu}\sigma_\la)H\|_{L^p_tL^q_x(M\times [0,T])}
\le C_q \la^{\frac1p-\frac12+2\e_0}\|H\|_{L^2_{t,x}(B(x_j,2\delta)\times \R)},
\end{equation}
assuming that $\delta$, as well as $\delta_1$ in \eqref{b.3} are fixed small enough.

To see this, if we use the auxiliary operator $\tilde A_{j,\ell_0}$ and Young's inequality as in the previous section, and apply Bernstein inequality in time, it suffices to show 
\begin{equation}\label{commsa}
\bigl\| (A_{j,\ell_0}\sigma_\la Q^{\theta_0}_{j,\ell_0,\nu}-A_{j,\ell_0}Q^{\theta_0}_{j,\ell_0,\nu}\sigma_\la)H\|_{L^2_tL^2_x(M\times [0,T])}
\le C_q \la^{-1+2\e_0}\|H\|_{L^2_{t,x}(B(x_j,2\delta)\times \R)},
\end{equation}
since $(n-1)(\frac12-\frac1q)+\frac12-\frac1p=\frac1p+\frac12$ for $(p,q)$ satisfying \eqref{p}.

This follows from the 
proof of (2.59) in \cite{HSst} since, by \eqref{aas}, $A_{j,\ell_0}f$ vanishes outside
$B(x_j,2\delta)$ and the two operators in \eqref{comms} vanish when acting on functions vanishing on
$B(x_j,2\delta)$.  This
allows one to prove \eqref{commsa},
exactly as in \cite{HSp},  by just working in a coordinate chart
($B(x_j,2\delta)$ here) and, to obtain the inequality
using \eqref{s.27} and  Egorov's theorem related to the properties of the half wave operator $e^{itP}$ in this local coordinate.

Next, as in \cite{blair2023strichartz} and \cite{HSst}, if $H=S_\la f$, we note that we can write for $\theta_0$ and $\tilde \sigma_\la$ as in \eqref{s.23}
\begin{equation}\label{s.32}
\bigl(\tilde \sigma_\la H\bigr)^2 =
\sum_{\nu,\nu'}
\bigl(\tilde \sigma_\la Q^{\theta_0}_\nu H\bigr)\cdot 
\bigl(\tilde \sigma_\la Q^{\theta_0}_{\nu'}H\bigr) +O(\la^{-N}\|H\|^2_{L^2_{t,x}(B(x_j,2\delta)\times \R)}), \, \, \forall \, N.
\end{equation}
Recall that the $\nu=\theta_0\cdot {\mathbb Z}^{2(n-2)+1}$ index a $\la^{-\e_0}$-separated lattice in ${\mathbb R}^{2(n-2)+1}$.  
If we repeat the Whitney decomposition and the arguments in \eqref{b.57}-\eqref{b.64} as in the previous section. We can write 
\begin{equation}\label{s.33}
(\tilde \sigma_\la H)^2 = \diag_{j,\ell_0}(H) +\far_{j,\ell_0}(H)
\end{equation}
 when $n-1\ge4$. And when $n-1=3$, we further decompose $\diag(H)$ and write
 \begin{equation}\label{s.34}
(\tilde \sigma_\la H)^4\lesssim2{\overline\Upsilon_{j,\ell_0}^{\text{diag}}}(H)+2{\overline\Upsilon_{j,\ell_0}^{\text{far}}}(H)+2(\far(H))^2.
\end{equation}
 Here the operators $\diag_{j,\ell_0}, \far_{j,\ell_0}, {\overline\Upsilon_{j,\ell_0}^{\text{diag}}} $ and ${\overline\Upsilon_{j,\ell_0}^{\text{far}}}$ are defined exactly in the same manner as in \eqref{b.58}, \eqref{b.59} and \eqref{organize}, except that they now act on functions that also depend on the time variable.
 And we are treating the case $n-1=3$ separately as $\frac{2(n-1)}{n-3}=6$ when $n-1=3$, which requires a slight modification when we use bilinear ideas from \cite{TaoVargasVega}. The remaining case $n-1=2$ is analogous to $n-1=3$, we shall briefly outline the necessary arguments in the end of this section.

We shall need the the following variant of  Lemma 3.1 in \cite{HSst} 
\begin{lemma}\label{lemmas1}  Let $\theta_0=\la^{-\e_0}$ with $\la \gg 1$.
If $n-1\ge4$, $q_e=\frac{2(n-1)}{n-3}$ and $Q^{\theta_0}_\nu=Q^{\theta_0}_{j,\ell_0,\nu}$ as in \eqref{s.26}, 
\begin{equation}\label{s.35}
\begin{aligned}
    \int &\big(\sum_j \int|\diag_{j,\ell_0}(H)|^{\frac{q_e}{2}}dx\big)^{\frac{2}{q_e}} dt \\
&\le  \,\int \big(\sum_{j,\nu}\bigl\| A_{j,\ell_0}\sigma_\la Q^{\theta_0}_{j,\ell_0,\nu} H\bigr\|^{q_e}_{L^{q_e}_x(B(x_j,2\delta))}\big)^{\frac{2}{q_e}}dt+O(\la^{1-}\|H\|^2_{L^2_{t,x}}).
\end{aligned}
\end{equation}
Additionally, for $n-1=3$ we have 
\begin{equation}\label{s.35a}
\begin{aligned}
    \int &\big(\sum_j \int|\diagbar_{j,\ell_0}(H)|^{\frac{q_e}{4}}dx\big)^{\frac{2}{q_e}} dt \\
&\le  \,\int \big(\sum_{j,\nu}\bigl\| A_{j,\ell_0}\sigma_\la Q^{\theta_0}_{j,\ell_0,\nu} H\bigr\|^{q_e}_{L^{q_e}_x(B(x_j,2\delta))}\big)^{\frac{2}{q_e}}dt+O(\la^{1-}\|H\|^2_{L^2_{t,x}}).
\end{aligned}
\end{equation}
\end{lemma}
In the above and what follows $O(\la^{\mu-})$ denotes $O(\la^{\mu-\e})$ for some $\e>0$. As we shall see later in the proof, unlike Lemma~\ref{lemma1} in the previous section, we can not fix $j,\ell_0$ here. It is crucial that the $\ell^{\frac{q_e}{2}}_j$ norm is taken inside the $dt$ integral.  As in \cite{HSst}, since the $Q_\nu^\theta$ operators are time independent, 
the main step in the proof of \eqref{s.35} is to show that 
for arbitrary $h_{\nu}, h_{\tilde \nu}$, which may depend on $\nu$ and $\tilde\nu$, we have 
\begin{equation}\label{s.35b}
    \begin{aligned}
        \bigl\|\sum_{(\nu,\tilde \nu)\in \Xi_{\theta_0}}
Q_{\nu,j,\ell_0}^{\theta_0} h_{ \nu}\cdot 
Q_{\tilde \nu,j,\ell_0}^{\theta_0} h_{\tilde \nu}\bigr\|
_{L^{q_e/2}_x} \le C &
\Bigl(
\sum_{(\nu,\tilde \nu)\in \Xi_{\theta_0}}
\| Q_{\nu,j,\ell_0}^{\theta_0}  h_{ \nu}\cdot 
Q_{\tilde \nu,j,\ell_0}^{\theta_0}  h_{\tilde \nu}\|_{L^{q_e/2}_{x}}^{q_e/2}\,
\Bigr)^{2/q_e} \\
&\qquad
+O\big(\la^{-N}\sum_{(\nu,\tilde \nu)\in \Xi_{\theta_0}}\|h_{\nu}\|_{L^1_x}\|h_{\tilde \nu}\|_{L^1_x}\big),\,\,\,\forall N.
    \end{aligned}
\end{equation}
The constant $C$ here is independent of $j,\ell_0$. This result is analogous to (3.20) in \cite{HSst} and follows from the same proof provided there. Similarly, the proof of \eqref{s.35a} follows from a variant of \eqref{s.35a} involving the product of four $Q_{j,\ell_0, \nu}^{\theta_0}$ operators.

If we  fix $\delta$ as well as 
$\delta_1,\delta_2$ in \eqref{b.3} small enough, then we can use Lee's \cite{LeeBilinear} bilinear oscillatory integral theorem 
and repeat the proof of Lemma 3.2 in \cite{blair2023strichartz} as well as the arguments in (3.38)-(3.42) of \cite{HSst} for the case $n-1=3$
to obtain the following.

\begin{lemma}\label{lemmas2}
Let $\far_{j,\ell_0}(H), \farbar_{j,\ell_0}(H)$ be as above with $\theta_0=\la^{-\e_0}$ and $q=\tfrac{2(n+2)}n$.  Then for all $\e>0$ there is a $C_\e=C(\e,M)$ so that 
\begin{equation}\label{s.36}
\int_M |\far_{j,\ell_0}(H)|^{q/2} \, dxdt \le C_\e \la^{1+\e} \, \bigl(\la^{1-\e_0}\bigr)^{\frac{n-1}2(q-\frac{2(n+1)}{n-1})}
\|H\|^q_{L^2_{t,x}(B(x_j,2\delta)\times\R)}.
\end{equation}
Similarly, for $n-1=3$ and $q=\tfrac{2(n+2)}n$,
\begin{equation}\label{s.37}
\int_M |\farbar_{j,\ell_0}(H)|^\frac{q}{4} \, dx dt\le C_\e \la^{1+\e} \, \bigl(\la^{1-\e_0}\bigr)^{\frac{n-1}2(q-\frac{2(n+1)}{n-1})}
\|H\|^q_{L^2_{t,x}(B(x_j,2\delta)\times\R)}.
\end{equation}
\end{lemma}

We now have collected the main ingredients that we need to prove the critical low height estimates.

\begin{proof}[Proof of \eqref{s.17} on the set $A_-$.]
Let us assume that $n-1\ge4$ and thus it suffices to consider $p=2, q=q_e=\frac{2(n-1)}{n-3}$.  A main step in the proof of the $A_-$ estimates then is to obtain the analog of (2.45) in \cite{HSst}.
We shall do so largely by repeating its proof, which we do so for the sake of completeness in order to note the small changes needed
to take into account that, unlike (2.45) in \cite{HSst}, \eqref{s.17} here is a vector valued inequality.  As noted before, we have taken
this framework to help us exploit our assumption of bounded geometry, and,  in particular, the fact that finitely many  of the 
 doubles of the balls $\{B(x_j,2\delta)\}$
in our covering of $M$  overlap.

We first note that if $q=\tfrac{2(n+2)}n<q_e$, then by \eqref{s.18a} and  \eqref{s.33} for our fixed $j, \ell_0$ we have for $H=S_\la f$
\begin{align*}
&\bigl|\bigl(\A\sigma_\la (H)(x,t,j)\bigr)^2\bigr|^{q_e/2}=\bigl|\bigl(A_{j,\ell_0}(\sigma_\la H)(x,t,j) \bigr)^2\bigr|^{q_e/2}
\\
&=|A_{j,\ell_0}(\sigma_\la H)(x,t) \cdot A_{j,\ell_0}(\sigma_\la H)(x,t)|^{\frac{q_e-q}2}
\bigl| \diag_{j,\ell_0}(H)(x,t) +\far_{j,\ell_0}(H)(x,t)\bigr|^{q/2}
\\
&\le |A_{j,\ell_0}(\sigma_\la H)(x,t) \cdot A_{j,\ell_0}(\sigma_\la H)(x,t)|^{\frac{q_e-q}2}
 2^{q/2}  \bigl( 
|\diag_{j,\ell_0}(H)(x,t)|^{q/2}
+|\far_{j,\ell_0}(H)(x,t)|^{q/2} \bigr).
\end{align*}
Thus if $A_-$ is as in \eqref{s.18},
\begin{equation}
    \begin{aligned}\label{s.38}
&\| \A \sigma_\la H
\|_{L^{2}_tL^{q_e}_x(A_-)}^2
=\int \Big(\int \sum_j \1_{A_-}(x,t,j) \, \left|A_{j,\ell_0}(\sigma_\la H)\cdot A_{j,\ell_0}(\sigma_\la H)|(x,t)\right|^{q_e/2}dx\Big)^{\frac2{q_e}}dt
\\
&\lesssim 
\int \Big(\sum_j \int 
\bigl[ \1_{A_-}(x,t,j) \, |A_{j,\ell_0}(\sigma_\la H)(x) \cdot A_{j,\ell_0}(\sigma_\la H)(x)|^{\frac{q_e-q}2}
\bigr] \, |\diag_{j,\ell_0}(H)(x)|^{q/2}\Big)^{\frac2{q_e}}dt
\\
&+\int \Big(\sum_j \int 
\bigl[ \1_{A_-}(x,t,j) \, |A_{j,\ell_0}(\sigma_\la H)(x) \cdot A_{j,\ell_0}(\sigma_\la H)(x)|^{\frac{q_e-q}2}
\bigr] \, |\far_{j,\ell_0}(H)(x)|^{q/2}dx\Big)^{\frac2{q_e}}dt \\ &=C(I+II).\notag
\end{aligned}
\end{equation}

To estimate $II$, 
first note that by H\"older's inequality
\begin{equation}\label{s.39}
\begin{aligned}
     II &\lesssim \|\1_{A_-}(x,t,j) A_{j,\ell_0}(\sigma_\la H)(x,t)\|_{L^\infty(A_-)}^{\frac{2(q_e-q)}{q_e}} \cdot 
\big(  \int \big(\sum_j \int|\far_{j,\ell_0}(H)|^{\frac{q}{2}}dx\big)^{\frac{2}{q_e}} dt \big) \\
&\lesssim T^{1-\frac{2}{q_e}}\|\1_{A_-}(x,t,j) A_{j,\ell_0}(\sigma_\la H)(x,t)\|_{L^\infty(A_-)}^{\frac{2(q_e-q)}{q_e}} \cdot 
\big(  \sum_j\int  \int|\far_{j,\ell_0}(H)|^{\frac{q}{2}}dxdt \big)^{\frac{2}{q_e}},
\end{aligned}
\end{equation}
Recall that by \eqref{s.18} and \eqref{s.18a},
$$|\1_{A_-}(x,t,j) A_{j,\ell_0}(\sigma_\la H)(x,t) |\lesssim \la^{\frac{n-1}{4}+\e_1}.
$$
Thus by \eqref{s.36}
\begin{equation}\label{far}
\begin{aligned}
    II \le  T^{1-\frac{2}{q_e}}& \la ^{(\frac{n-1}4+\e_1)(\frac{2(q_e-q)}{q_e})} \\
&\cdot \Big(\la^{1+\e} \, \bigl(\la^{1-\e_0}\bigr)^{\frac{n-1}2
(q-\frac{2(n+1)}{n-1})} \Big)^{\frac{2}{q_e}}(\sum_j\|H\|^q_{L^2_{t,x}(B(x_j,2\delta)\times\R)})^{\frac{2}{q_e}}.
\end{aligned}
\end{equation}
If we take $\e_0, \e_1$ and $\e$ to be small enough, e.g., $\e=\e_1=\frac1{100}$ and $\e_0=\frac{1}{2n+2}$, it is straightforward to check that 
\begin{equation}\label{far1}
\begin{aligned}
    II \lesssim\la^{1-}(\sum_j\|H\|^q_{L^2_{t,x}(B(x_j,2\delta)\times\R)})^{\frac{2}{q_e}}
\lesssim \la^{1-}\|H\|_{L^2_{t,x}}^{\frac{2q}{q_e}}=O(\la^{1-}\|H\|_{L^2_{t,x}}^{2}).
\end{aligned}
\end{equation}
Here we also  used the fact that 
$\|H\|_{L^2_{t,x}}^{2}$ dominates 
$\|H\|_{L^2_{t,x}}^{\frac{2q}{q_e}}$ since $q_e>q$ and
$\|H\|_{L^2_{t,x}}\approx T$ since $H=S_\la f$, $\|f\|_2=1$
and $e^{-it\la^{-1}\Delta_g}$ is a unitary operator on $L^2_x$.

To control $I$, as in \cite{HSst},  we
use H\"older's inequality followed by Young's 
inequality along with \eqref{s.35} to get
\begin{equation}
    \begin{aligned}
        I=&
\int \Big(\sum_j \int 
\bigl[ \1_{A_-}(x,t,j) \, |A_{j,\ell_0}(\sigma_\la H)(x) \cdot A_{j,\ell_0}(\sigma_\la H)(x)|^{\frac{q_e-q}2}
\bigr] \, |\diag_{j,\ell_0}(H)(x)|^{q/2}\Big)^{\frac2{q_e}}dt\\
\le & \int\Big( \|\1_{A_-}\cdot \A \sigma_\la H \cdot
 \A \sigma_\la H\|_{\ell_j^{\frac{q_e}{2}}L_x^{\frac{q_e}{2}}}^{\frac{q_e-q}2} \, 
 \big(\sum_j \int|\diag_{j,\ell_0}(H)|^{\frac{q_e}{2}}dx\big)^{\frac{q}{q_e}}\,\Big)^{\frac{2}{q_e}}dt \\
\le &\| \1_{A_-}\A \sigma_\la H \cdot
 \A \sigma_\la H\|_{L^1_t\ell_j^{\frac{q_e}{2}}L_x^{\frac{q_e}{2}}}^{\frac{q_e-q}{q_e}} \, 
\big( \int \big(\sum_j \int|\diag_{j,\ell_0}(H)|^{\frac{q_e}{2}}dx\big)^{\frac{2}{q_e}} dt \big)^{\frac{q}{q_e}} \\
\le &\tfrac{q_e-q}{q_e}\| \A \sigma_\la H \|^2_{L^2_t\ell_j^{q_e}L_x^{q_e}(A_-)} +\tfrac{q}{q_e}
\big( \int \big(\sum_j \int|\diag_{j,\ell_0}(H)|^{\frac{q_e}{2}}dx\big)^{\frac{2}{q_e}} dt \big)\\
\le &\tfrac{q_e-q}{q_e}\| \A \sigma_\la H \|^2_{L^2_t\ell_j^{q_e}L_x^{q_e}(A_-)} +C(\,\int \big(\sum_{j,\nu}\bigl\| A_{j,\ell_0}\sigma_\la Q^{\theta_0}_{j,\ell_0,\nu} H\bigr\|^{q_e}_{L^{q_e}_x(B(x_j,2\delta))}\big)^{\frac{2}{q_e}}dt)\\
&\quad+O(\la^{1-}\|H\|^2_{L^2_{t,x}}).
    \end{aligned}
\end{equation}

If we combine the above two estimate, we have the following which is the analog of proposition 2.3 in \cite{HSst}.
\begin{proposition}\label{locprop}
Fix a complete $n-1\ge 2$ dimensional Riemannian manifold $(M,g)$ of bounded geometry and assume that \eqref{s.norm} is valid.
If $H=S_\la f$ is as in \eqref{s.2}, $(p,q)$ satisfies \eqref{p} and $\e_0,\e_1$ in the definition of $A_-$ and $\theta_0$ are small enough, we have
\begin{equation}\label{s.40}
\|\A \sigma_\la H\|_{L^{p}_tL^{q}_x\ell_j^q(A_-)}
\lesssim \big(\,\int \big(\sum_{j,\nu}\bigl\| A_{j,\ell_0}\sigma_\la Q^{\theta_0}_{j,\ell_0,\nu} H\bigr\|^{q}_{L^{q}_x(B(x_j,2\delta))}\big)^{\frac{2}{q}}dt\big)^{\frac12}
+\la^{\frac1{p}-}.
\end{equation}
\end{proposition}

The case $n-1\ge 4$ in \eqref{s.40} directly follows from the above estimates for $I$ and $II$.
One can similarly use \eqref{s.35a} and \eqref{s.37} and
 modify the arguments in \cite{HSst}
to handle the case when $n-1=3$.

The arguments for $n-1=2$ and general $(p,q)$ in \eqref{p} is similar to the case $n-1=3$. Recall that when $n-1=3$, $q_e=\frac{2(n-1)}{n-3}=6\in [2^2, 2^3]$. As a result, an additional round of 
Whitney decomposition is needed for $(\diag)^2$ in order to get the desired estimate \eqref{s.35a} in Lemma~\ref{lemmas1}.  When $n-1=2$, $q$ can be arbitrary large, if $q\in [2^{k+1}, 2^{k+2}]$ for some $k\in \mathbb{N}^+$, then one can repeat the arguments for $n-1=3$ in \cite{HSst} $k$ times,  the resulting diagonal term will involve a product of $2^{k+1}$ terms of involving $A_{j,\ell_0}\sigma_\la Q^{\theta_0}_{j,\ell_0,\nu} H$ and will satisfy the analog of \eqref{s.35a} with $q_e/4$ replaced by $q/2^k$. Each iteration of 
Whitney decomposition also generates off-diagonal terms,
which can be treated using bilinear oscillatory integral estimates. However, as $q\rightarrow \infty$, unlike \eqref{far}, we need to take $\e_0$ and $\e_1$ to be small enough depending on $q$, instead of some fixed small constant.

Thus to prove \eqref{s.17}, it remains to control the first term on the right side of \eqref{s.40}.
By \eqref{comms} along with the fact that $\ell^{2}\subset \ell^{q}$ if $q\ge 2$ , we have 
\begin{equation}\label{s.41}
    \begin{aligned}
      \big(\,\int & \big(\sum_{j,\nu}\bigl\| A_{j,\ell_0}\sigma_\la Q^{\theta_0}_{j,\ell_0,\nu} H\bigr\|^{q}_{L^{q}_x(B(x_j,2\delta))}\big)^{\frac{2}{q}}dt\big)^{\frac12}
\\
&\lesssim    \big(\,\int  \big(\sum_{j,\nu}\bigl\| A_{j,\ell_0} Q^{\theta_0}_{j,\ell_0,\nu} \sigma_\la H\bigr\|^{q}_{L^{q}_x(B(x_j,2\delta))}\big)^{\frac{2}{q}}dt\big)^{\frac12} \\
&\quad+   \big(\,\int  \big(\sum_{j,\nu}\bigl\| (A_{j,\ell_0}\sigma_\la Q^{\theta_0}_{j,\ell_0,\nu}-A_{j,\ell_0} Q^{\theta_0}_{j,\ell_0,\nu}) \sigma_\la H\bigr\|^{q}_{L^{q}_x(B(x_j,2\delta))}\big)^{\frac{2}{q}}dt\big)^{\frac12}
\\
&\lesssim 
 \big(\int\sum_{j,\nu}\bigl\| A_{j,\ell_0} Q^{\theta_0}_{j,\ell_0,\nu}\sigma_\la H\bigr\|^{q}_{L^{q}_x(B(x_j,2\delta))}\big)^{\frac{2}{q}}dt\big)^{\frac12}+\la^{\frac1p-\frac12+2\e_0}\big( \sum_\nu \| H\,\|^2_{L^2_{t,x}}\big)^{\frac12} .
    \end{aligned}
\end{equation}
Since the number of choices of $\nu$ is $O(\la^{(2n-3)\e_0})$ and $H$ is independent of $\nu$, the second term in the right is dominated by $\la^{(n-\frac32)\e_0}\| \, 
H\,\|_{L^2_{t,x}}$. Thus if we choose $\e_0<\tfrac{1}{2n+1}$,  the second term on the right side of \eqref{s.41} is 
$O(\la^{\frac1{p}-})$.

Next recall that $H=S_\la f$ and $\|f\|_2=1$, if we use \eqref{b.16}, \eqref{s.31},  followed by \eqref{s.9}, we can control the term in the right as follows
\begin{equation}\label{s.42}
\begin{aligned}
 \big(\int&\sum_{j,\nu}\bigl\| A_{j,\ell_0} Q^{\theta_0}_{j,\ell_0,\nu}\sigma_\la S_\la f\bigr\|^{q}_{L^{q}_x(B(x_j,2\delta))}\big)^{\frac{2}{q}}dt\big)^{\frac12} \\
 &\le  \big(\int\sum_{j,\nu}\bigl\|  Q^{\theta_0}_{j,\ell_0,\nu}\sigma_\la S_\la f\bigr\|^{q}_{L^{q}_x}\big)^{\frac{2}{q}}dt\big)^{\frac12}\\
&\le  \big(\int\sum_{j,\nu}\bigl\|  Q^{\theta_0}_{j,\ell_0,\nu} S_\la f\bigr\|^{q}_{L^{q}_x}\big)^{\frac{2}{q}}dt\big)^{\frac12} +  \big(\int\sum_{j,\nu}\bigl\|  Q^{\theta_0}_{j,\ell_0,\nu}(I-\sigma_\la) S_\la f\bigr\|^{q}_{L^{q}_x}\big)^{\frac{2}{q}}dt\big)^{\frac12} 
\\ &\le   \big(\int\sum_{j,\nu}\bigl\|  Q^{\theta_0}_{j,\ell_0,\nu} S_\la f\bigr\|^{q}_{L^{q}_x}\big)^{\frac{2}{q}}dt\big)^{\frac12} +  \big(\int\bigl\|  (I-\sigma_\la) S_\la f\bigr\|^{q}_{L^{q}_x}\big)^{\frac{2}{q}}dt\big)^{\frac12} 
\\ &\le   \big(\int\sum_{j,\nu}\bigl\|  Q^{\theta_0}_{j,\ell_0,\nu} S_\la f\bigr\|^{q}_{L^{q}_x}\big)^{\frac{2}{q}}dt\big)^{\frac12} + \la^{\frac1p}T^{\frac1p-\frac12}.
\end{aligned}
\end{equation}

If we combine \eqref{s.40} and the preceding two inequalities we conclude that we would obtain \eqref{s.17} and consequently
finish the proof of the estimates in Theorem~\ref{bgst} if, for $(p,q)$ as in \eqref{p} and $T$ as in \eqref{s.2}, we could show that
\begin{equation}\label{s.43}
Uf(t,x,j,\nu)  =(Q^{\theta_0}_{j,\ell_0,\nu} S_\la f)(x,t),
\end{equation}
satisfies
\begin{equation}\label{st}
\| Uf \|_{L^p_t\ell_j^{q}\ell_\nu^{q}L^{q}_x} \lesssim \la^{\frac1{q}} \|f\|_{L^2(M)}.
\end{equation}

We shall require the following lemma
\begin{lemma}\label{kerprop1}  Fix $t, j,\ell_0,\nu$, let $K_{t,\la}$ denote  the operator 
$$\eta( t/T)Q^{\theta_0}_{j,\ell_0,\nu} \beta(P/\la)e^{-it\la^{-1}\Delta_g}.$$
Then if $(M,g)$ is a complete manifold of bounded geometry all of whose sectional curvatures are nonpositive and $T=c_0\log\la$ is fixed with $c_0=c_0(M)>0$ sufficiently small, we have for $\la\gg 1$
\begin{equation}\label{j1a}
\|K_{t,\la}K^*_{s,\la}\|_{L^1(M)\to L^\infty(M)}\le C\la^{\frac{n-1}2}|t-s|^{-\frac{n-1}2}.
\end{equation}
\end{lemma}
We shall postpone the proof of this lemma until the end of this section and first see
how we can use it to prove \eqref{st}.
By applying the abstract theorem of Keel-Tao \cite{KT} and a simple rescaling argument, 
 we 
would have \eqref{st} if
\begin{equation}\label{s.44}
\| Uf(t,\cdot)\|_{\ell^2_j\ell_\nu^{2}L^{2}_{x}}
\le C \|f\|_{L^2_x},
\end{equation}
and 
\begin{equation}\label{s.45}
\| U(t)U^*(s)G\|_{\ell^\infty_j\ell_\nu^{\infty}L^{\infty}_x}
\le  C\la^{\frac{n-1}2} \, |t-s|^{-\frac{n-1}2}  \|G\|_{\ell_j^1\ell_\nu^{1}L^{1}_x},
\end{equation}
with
\begin{align}\label{s.46}
&\big(U(t)U^*(s)G\big)(x,j,\nu)=
\\
&= \eta(t/T) \sum_{j', \nu'} 
 \eta(s/T) \Bigl[ \bigr( Q_{j,\ell_0, \nu}^{\theta_0} e^{-i(t-s)\la^{-1}\Delta_g} (Q^{\theta_0}_{j',\ell_0,\nu'})^* \bigr)G( \, \cdot ,\,j',\,\nu')\Bigr] (x)., \notag
\end{align}
It is not hard to check that \eqref{s.44} follows from  \eqref{s.31}  with $p=2$ and the fact that $e^{-it\la^{-1}\Delta_g}$ is unitary, and \eqref{s.45} follows from the estimate \eqref{j1a}.
\end{proof}


\noindent{\bf 3.3. Kernel estimates}

Let us start out by proving the bounds in Lemmas \ref{Gl} and \ref{G2} that were used to prove the spectral projection
estimates in Theorem~\ref{bgsp}.

\begin{proof}[Proof of Lemma~\ref{Gl}] 
We first note that since $P$ is nonnegative, if we replace $e^{-itP}$ in \eqref{b.33} with $e^{itP}$, then, by \eqref{ii.3}, 
the resulting operator maps $L^1(M)\to L^\infty(M)$ with norm $O(\la^{-N}) \, \forall \, N$.  Thus, by Euler's formula, if
\begin{equation}\label{k1}
\tilde G_\la(x,y)= \int_{-\infty}^\infty (1-a(t)) \, T^{-1}
\Hat \Psi(t/T) \, \bigl(\cos t\sqrt{-\Delta_g}\bigr)(x,y) \, dt,
\end{equation}
it suffices to show that, under the assumptions of Lemma~\ref{Gl}, we have
\begin{equation}\label{k.2}
\tilde G_\la(x,y)=O(\la^{\frac{n-1}2} \exp(C_MT)),
\end{equation}
assuming that $T=c_0\log\la$, with $c_0=c_0(M)>0$ sufficiently small.

To prove this, we can use the arguments of B\'erard~\cite{Berard}.  Indeed, if we use the covering map coming from
the exponential map $\kappa=\exp_x: T_xM \simeq \Rn \to M$ at $x$, then $\kappa$ is a covering map and
$(\Rn, \tilde g)$ $\kappa^*g=\tilde g$, is the universal cover.  Like $(M,g)$, all of the sectional curvatures of
$(\Rn,\tilde g)$ are nonpositive.  As in \eqref{k.3} above, let $\Gamma$ be the associated
deck transformations and choose a Dirichlet domain $D$ associated with the origin, which is  in the lift of $x$.
If $\tilde x,\tilde y$ are the lifts to $D$ of $x,y\in M$, we have the formula 
\begin{equation}\label{k.3'}
\bigl(\cos t\sqrt{-\Delta_g} \bigr)(x,y)=\sum_{\alpha\in \Gamma} \bigl(\cos t\sqrt{-\Delta_{\tilde g}} \bigr)(\tilde x, \alpha(\tilde y)).
\end{equation}
As a result, 
\begin{equation}\label{k.4}
\tilde G_\la(x,y)=\sum_{\alpha\in \Gamma}
\int_{-\infty}^\infty (1-a(t)) \, T^{-1}\Hat \Psi(t/T) \,
\bigl(\cos t\sqrt{-\Delta_{\tilde g}}\bigr)(\tilde x,\alpha(\tilde y)) \, dt.
\end{equation}

To use this formula, we first note that, by \eqref{b.3},
$\Hat \Psi(s)=0$ if $|s|>2$, which means that the integrands in \eqref{k.4} vanishes for $|t|>2T$.  Also, by finite propagation
speed for the wave operator,
$$\bigl(\cos t\sqrt{-\Delta_{\tilde g}}\bigr)(\tilde x,\tilde z)=0 \, \, \text{if } \, \,
d_{\tilde g}(\tilde x,\tilde z)>|t|,$$
and so each of the summands in \eqref{k.4}
\begin{multline}\label{k.5}
K_\alpha (\tilde x,\tilde y)=\int_{-\infty}^\infty (1-a(t)) \, T^{-1}\Hat \Psi(t/T)
\,  \bigl(\cos t\sqrt{-\Delta_{\tilde g}}\bigr)(\tilde x,\alpha (\tilde y))\, dt =0,
\\
\text{if } \, \, d_{\tilde g}(\tilde x,\alpha(\tilde y))>2T=2c_0\log\la.
\end{multline}
Furthermore, if $c_0>0$ here is small enough then since $(\Rn,\tilde g)$ is of bounded geometry and all of its
sectional curvatures are nonpositive, as in \cite{Berard}, \cite[\S 3.6]{SoggeHangzhou}, one can use the
Hadamard parametrix
and stationary phase arguments to see that for $T$ as above one has the uniform bounds
\begin{equation}\label{k.6}
K_\alpha(\tilde x,\tilde y)=O(\la^{\frac{n-1}2}).
\end{equation}

As a result, we would obtain the bound \eqref{k.2} if we could verify that there are $O(\exp(C_MT))$ nonzero summands
in \eqref{k.4} for $T$ as above.  To do this, we let $r=\Inj (M)/4$.  Then if $B_{\tilde g}(\tilde z,r)$ is the geodesic ball
in $(\Rn,\tilde g)$ with center $\tidle z$ and radius $r$, we must have
\begin{equation}\label{k.7}
B_{\tilde g}(\alpha(\tilde y),r) \cap B_{\tilde g}(\alpha'(\tilde y),r)=\emptyset
\, \, \text{if } \, \, \alpha\ne \alpha', \, \,
\text{and } \, \alpha,\alpha'\in \Gamma.
\end{equation}
 Note that, by the above, in order for $K_\alpha(\tilde x,\tilde y)$ to be nonzero
we must also have that $B_{\tilde g}(\alpha(\tilde y),r)\subset B_{\tilde g}(\tilde x,2T+r)$.
Additionally, the volume of $B_{\tilde g}(\tilde z,r)$ must be $O(1)$ due to the fact that $(\Rn,\tilde g)$ is
of bounded geometry.  Similarly, since the sectional curvatures of $(\Rn,\tilde g)$ must be bounded below,
by standard volume comparison theorems (see e.g., \cite{ChavelRiemannianGeometry}) the volume of
$B_{\tilde g}(\tilde x,2T+r)$ must be $O(\exp(C_MT))$, assuming, as we may that $T>r$.  These two crude
volume estimates along with \eqref{k.7} yield the above claim about the number of nonzero summands
in \eqref{k.5}, which finishes the proof.
\end{proof}

\begin{proof}[Proof of Lemma~\ref{G2}] 
In view of the first estimate in \eqref{b.55} for $q=\infty$, we can use Euler's formula as above to see that
we would have \eqref{b.83} if we could show that for $T=c_0\log\la$ with $c_0>0$ sufficiently small we have
for $\la\gg 1$
\begin{multline}\label{k.8}
Q^{\theta_0}_{j,\ell_0,\nu} \tilde G_{\la,N}(x,y)
=\int_{-\infty}^\infty (1-a(t)) \, T^{-1}  \Hat \Psi(t/T) \, \beta(|t|/N) \,
Q^{\theta_0}_{j,\ell_0,\nu}  \bigl(\cos t\sqrt{-\Delta_g}\bigr)(x,y)\, dt
\\
=O(T^{-1}\la^{\frac{n-1}2}N^{1-\frac{n-1}2}),
\end{multline}
assuming that the sectional curvatures of $(M,g)$ are nonpositive.

We can use \eqref{k.3'} to write
\begin{multline}\label{k.9}
Q^{\theta_0}_{j,\ell_0,\nu} \tilde G_{\la,N}(x,y) =\sum_{\alpha\in \Gamma}K^{j,\ell_0,\nu}_\alpha(\tilde x,\tilde y),
\, \, \text{where } 
\\
K^{j,\ell_0,\nu}_\alpha(\tilde x,\tilde y) =\int_{-\infty}^\infty (1-a(t)) T^{-1}\Hat \Psi(t/T) \beta(|t|/N)
Q^{\theta_0}_{j,\ell_0,\nu}  \bigl(\cos t\sqrt{-\Delta_{\tilde g}}\bigr)(\tilde x,\alpha (\tilde y))\, dt,
\end{multline}
abusing notation a bit here by letting $Q^{\theta_0}_{j,\ell_0,\nu}$ here denote the lift of the operator on $(M,g)$
to $(\Rn,\tilde g)$ via the covering map.

Since the integrand in \eqref{k.9} vanishes when $|t|\notin (N/2,2N)$ one can use the Hadamard parametrix along
with \eqref{b.55} to see that, by the arguments in \cite{BSTop}, 
\begin{equation}\label{k.10}
K^{j,\ell_0,\nu}_\alpha(\tilde x,\tilde y) =
\begin{cases} O(\la^{\frac{n-1}2}N^{-\frac{n-1}2}) \, \, \,
\text{if } \, d_{\tilde g}(\tilde x,\alpha(\tilde y)) \in [N/4,4N]
\\
O(\la^{-m}) \, \, \forall \, m\in {\mathbb N}\, \, \text{otherwise},
\end{cases}
\end{equation}
if $T=c_0\log\la$ with $c_0>0$ sufficiently small.

This, by itself will not yield \eqref{b.83}.  For this, let $\tilde \gamma =\tilde \gamma_{j,\ell_0,\nu}\subset \Rn$
be the geodesic through the origin of the lift of the geodesic $\gamma_{j,\ell_0,\nu}\subset M$ associated
with $Q^{\theta_0}_{j,\ell_0,\nu}$.  Then the arguments in \cite{BSTop} also yield that if
$T=c_0\log\la$ with $c_0>0$ small enough one has
\begin{equation}\label{k.11}
K^{j,\ell_0,\nu}_\alpha(\tilde x,\tilde y) =O(\la^{-m}) \, \,
\forall \, m\in {\mathbb N} \, \, \text{if } \,
d_{\tilde g}(\tilde \gamma, \alpha(\tilde y))\ge C_0,
\end{equation}
for some fixed $C_0=C_0(M)$.  Since we can also use the volume counting arguments in \cite{BSTop} to see that
that number of $\alpha\in \Gamma$ for which $d_{\tilde g}(\tilde x,\alpha (\tilde y))\in [N/4,4N]$
and $d_{\tilde g}(\tilde \gamma, \alpha(\tilde y))\le C_0$ is $O(N)$, we obtain 
\eqref{b.83} from \eqref{k.9}, \eqref{k.10} and \eqref{k.11}.

If we assume that the sectional curvatures of $(M,g)$, and hence $(\Rn,\tilde g)$, are pinched below zero as in
\eqref{b.84}, then we have much more favorable dispersive estimates for the main term in the Hadamard parametrix,
as noted in \cite{BHSsp} and \cite{HSp}.  This leads to the improvement of the first part of \eqref{k.10} under this
curvature assumption:
\begin{equation}\label{k.12}
K^{j,\ell_0,\nu}_\alpha(\tilde x,\tilde y)=O_m(\la^{\frac{n-1}2}N^{-m}) \, \, \forall \, m\in {\mathbb N}.
\end{equation}
By using this along with the above arguments, we obtain the other estimate, \eqref{b.84}, in
Lemma~\ref{G2}.
\end{proof}

Now we shall prove the bounds that were used for the Strichartz estimates.
\begin{proof}[Proof of Lemma~\ref{kerprop}]
    
To prove \eqref{j1}, we shall mostly follow the proof of Proposition 4.1 in \cite{blair2023strichartz} as well as the ideas in the proof of Lemma~\ref{Gl} above.  Note that for fixed $t$ and $s$,
$\beta^2(P/\la)e^{-i(t-s)\la^{-1}\Delta_g}=\beta^2(P/\la)e^{i(t-s)\la^{-1}P^2}$ is the  Fourier multiplier operator
on $M$ with
\begin{equation}\label{j2}
     m(\la,t-s;\tau)=\beta^2(|\tau|/\la)e^{i(t-s)\la^{-1}\tau^2}.
\end{equation}
We have extended $m$ to be an even function of $\tau$ so that we can write
\begin{equation}\label{j3}
\beta^2(P/\la)e^{-i(t-s)\la^{-1}\Delta_g}=(2\pi)^{-1}\int_{-\infty}^\infty 
\Hat m(\la,t-s;r) \, \cos r\sqrt{-\Delta_g} \, dr,
\end{equation}
where
\begin{equation}\label{j4}
\Hat m(\la,t-s;r)=\int_{-\infty}^\infty e^{-i\tau r}
\beta^2(|\tau|/\la) \, e^{i(t-s)\la^{-1}\tau^2} \, d\tau.
\end{equation}

We note that, by a simple integration by parts argument,
\begin{multline}\label{j6}
\partial^k_r \Hat m(\la,t-s;r)=O(\la^{-N}(1+|r|)^{-N}) \, \forall \, N, 
\\
 \text{if } \, |t-s|\le 2^j, \, \, \text{and } \, \, |r|\ge C_02^j, \, \, 
j=0,1, 2,\dots,
\end{multline}
with $C_0$ fixed large enough.  Since $\beta(|\tau|/\la)=0$ if $|\tau|\notin [\la/4,2\la]$ one may take 
$C_0=100$, as we shall do.

To use this fix an even function $a\in C^\infty_0(\R)$ satisfying
$$a(r)=1, \, \, |r|\le 100 \quad \text{and } \, \, a(r)=0\, \, \text{if } \, \,  |r|\ge200.$$
Then  if 
we let
\begin{equation}\label{j8}
\tilde S_{\la,j}(t,s)(P)=(2\pi)^{-1}\int a(2^{-j}r)\Hat m(\la,t-s,r) \cos rP\, dr
\end{equation}
we have the symbol $F_{\la,j}(\tau)$ of the multiplier operator $$F_{\la,j}(P)=\tilde S_{\la,j}(t,s)(P)-\beta^2(P/\la)e^{i(t-s)\la^{-1}P^2}$$ is $O(\la^{-N_1}(1+\tau)^{-N_2})\,\,\,\forall N_1, N_2$ if $ |t-s|\le 2^j$. Thus by \eqref{ii.3} we have $$\|F_{\la,j}(P)\|_{L^1(M)\to L^\infty (M)}\lesssim 1, \,\,\,\text{if}\,\,\, |t-s|\le 2^j.$$

Consequently, if we let  $  \tilde S_{\la,j}(x,t;y,s)$ denote
the kernel of the multiplier operator $\tilde S_{\la,j}(t,s)(P) $,
we would have \eqref{j1} if we could show that
\begin{multline}\label{j12}
|\tilde S_{\la,j}(x,t;y,s)|\le \la^{\frac{n-1}2}|t-s|^{-\frac{n-1}2} \exp(C2^j), \quad \text{if } \, |t-s|\le 2^j
\\
\text{with } \, \, j=0,1,2,\dots \, \, \text{and } \, \,  \, 2^j \le c_0\log\la
\end{multline}
with $c_0=c_0(M)$ fixed small enough.

To prove \eqref{j12}, as in the proof of Lemma~\ref{Gl}, we shall use the 
Hadamard parametrix and the 
Cartan-Hadamard theorem to lift the 
calculations that will be needed up to the
universal cover $({\mathbb R}^{n-1},\tilde g)$
of $(M,g)$. Let $\Gamma$ be the associated
deck transformations and choose a Dirichlet domain $D$ associated with the origin.
If $\tilde x,\tilde y$ are the lifts to $D$ of $x,y\in M$,
by \eqref{k.3} if we set 
\begin{equation}\label{j14}
K_{\la,j}(\tilde x,t;\tilde y,s)=
(2\pi)^{-1}\int a(2^{-j}r)\Hat m(\la,t-s;r)
\, \bigl( \cos r\sqrt{-\Delta_{\tilde g}}\bigr)(\tilde x,
\tilde y) \, dr,
\end{equation}
we have the formula
\begin{equation}\label{j15}
\tilde S_{\la,j}(x,t;y,s)= 
\sum_{\alpha \in \Gamma}
K_{\la,j} (\tilde x,t; \alpha(\tilde y),s).
\end{equation}

Also, by Huygen's principle and the support properties
of $a$, we have that
\begin{equation}\label{j18}
 K_{\la,j}(\tilde x,\tilde y)=0 \, \, \text{if } \, \, 
d_{\tilde g}(\tilde x,\tilde y)\le C_12^j
\end{equation}
for a uniform constant $C_1$.  Based on this, if we argue as in the proof of Lemma~\ref{Gl} using \eqref{k.7} along with simple volume estimates related to the bounded geometry assumption, it is not hard to show that 
 the number
of non-zero summands on the right side of  \eqref{j15} is $O(\exp(C2^j))$.  As a result, we would obtain \eqref{j12} if we could show that
\begin{multline}\label{j20}
|K_{\la,j}(\tilde x,t;\tilde y,s)|\le C\la^{\frac{n-1}2}|t-s|^{-\tfrac{n-1}2},  \quad \text{if } \, |t-s|\le 2^j \, \, \\ \text{with}\,\,\,j=0,1,2,\dots, \, 2^j\le  c_0\log\la.
\end{multline}

As in the previous section, to prove \eqref{j20}, we can use the Hadamard
parametrix for $\partial_r^2-\Delta_{\tilde g}$
since $({\mathbb R}^{n-1},\tilde g)$ is a Riemannian
manifold without conjugate points, i.e., its
 injectivity radius is infinite.  More explicitly for
$\tilde x\in D$, $\tilde y\in {\mathbb R}^{n-1}$
and $|r|>0$
\begin{equation}\label{k13}
\bigl(\cos r\sqrt{-\Delta}_{\tilde g}\bigr)(\tilde x,
\tilde y)=
\sum_{\nu=0}^N
w_\nu(\tilde x,\tilde y)W_\nu(r,\tilde x,\tilde y)
+R_N(r,\tilde x,\tilde y)
\end{equation}
where $w_\nu, W_\nu$ and $R_N$ satisfies \eqref{k14}-\eqref{k19}.

 By \eqref{k13}, it suffices
to see that if we replace $(\cos r\sqrt{-\Delta_{\tilde g}})(\tilde x,\tilde y)$ in \eqref{j14} by each
of the terms in the right side of \eqref{k13} then each such expression will satisfy the bounds
in \eqref{j20}.

Let us start with the contribution of the main term in the Hadamard parametrix which is the
$\nu=0$ term in \eqref{k13}.  In view of \eqref{k14} and \eqref{k17} it would give rise to these
bounds if
\begin{multline}\label{j28}
(2\pi)^{-n}\int_{-\infty}^\infty \int_{{\mathbb R}^{n-1}} e^{id_{\tilde g}(\tilde x,\tilde y)\xi_1}
\cos (r|\xi|) \, a(2^{-j}r) \, \Hat m(\la,t-s;r) \, dr d\xi
\\
=O(\la^{\frac{n-1}2}|t-s|^{-\frac{n-1}2}) \quad
\text{when }\,\,\, |t-s|\le 2^j .
\end{multline}
However, by \eqref{j2} and \eqref{j6} and the support properties of $a$, 
\begin{align}\label{j29}
(2\pi)^{-1}&\int_{-\infty}^\infty \int_{{\mathbb R}^{n-1}} e^{id_{\tilde g}(\tilde x,\tilde y)\xi_1}
\cos (r|\xi|) \, a(2^{-j}r) \, \Hat m(\la,t-s;r) \, dr d\xi
\\
&=
(2\pi)^{-1}\int_{-\infty}^\infty \int_{{\mathbb R}^{n-1}} e^{id_{\tilde g}(\tilde x,\tilde y)\xi_1}
\cos (r|\xi|) \Hat m(\la,t-s;r) \, dr d\xi+O(\la^{-N}) \notag
\\
&=\int_{{\mathbb R}^{n-1}} e^{id_{\tilde g}(\tilde x,\tilde y)\xi_1} \beta^2(|\xi|/\la) e^{i(t-s)\la^{-1}|\xi|^2} \, d\xi
+O(\la^{-N}). \notag
\end{align}
A simple stationary phase argument shows that the last integral is $O(\la^{\frac{n-1}2}|t-s|^{-\frac{n-1}2})$, and so  
we conclude that the main term in the Hadamard parametrix leads to the desired bounds.

Similarly, one can use stationary phase to show that if $ |t-s|\le 2^j $
\begin{equation}
    \begin{aligned}
        (2\pi)^{-1} \iint &e^{id_{\tilde g}(\tilde x,\tilde y)\xi_1} e^{\pm ir|\xi|} \alpha_\nu(|\xi|) \, a(2^{-j}r)\Hat m(\la,t-s;r) \, dr d\xi
\\
&=\int_{{\mathbb R}^{n-1}} e^{id_{\tilde g}(\tilde x,\tilde y)\xi_1} \beta^2(|\xi|/\la) e^{i(t-s)\la^{-1}|\xi|^2} \, \alpha_\nu(|\xi|) \, d\xi
+O(\la^{-N})\\
&=O(\la^{\frac{n-1}2-\nu}|t-s|^{-\frac{n-1}2}).
    \end{aligned}
\end{equation}
Note that  by \eqref{j18}
we may assume that $d_{\tilde g}(\tilde x,\tilde y)\le Cc_0\log\la$ .  So by \eqref{k13} and \eqref{k19}, if we choose $c_0$ small enough, the contributions from the higher order terms would
be
$O(\la^{\frac{n-1}2-\frac12}|t-s|^{-\frac{n-1}2})$.

We also need to see that the remainder term in \eqref{k13} leads to the bounds
\begin{multline}\label{j30}
\int^\infty_{-\infty} a(2^{-j}r)\Hat m(\la,t-s;r) R(r,\tilde x,\tilde y) \, dr
\\
=\int^\infty_{-\infty} \beta^2(|\tau|/\la) e^{i(t-s)\la^{-1}\tau^2} \, 
\bigl[ a(2^{-j}\, \cdot \, )R(\, \cdot\, ,\tilde x,\tilde y)\bigr]\, \widehat{}\, \,(\tau) \, d\tau =O(\la^{-N}), \, \, \forall N.
\end{multline}
Since we are assuming that $d_{\tilde g}(\tilde x,\tilde y)\le Cc_0\log\la$,
by \eqref{k16} and support properties of $\alpha$, the last factor in the integral in the right, which is the Fourier transform of
$r\to a(r)R(r,\tilde x,\tilde y)$, is $O(|\tau|^{-N}\exp(CNc_0\log\la))$.  So, by the support properties of $\beta$, the last integral
in \eqref{j30} is $O(\la^{-N})$ if we fix $c_0$ small enough.
\end{proof}

\begin{proof}[Proof of Lemma~\ref{kerprop1}]
If we use the second part of \eqref{s.30}, it suffices to show 
\begin{equation}\label{j31}
\|\eta( t/T)\eta(s/T)Q^{\theta_0}_{j,\ell_0,\nu} \beta^2(P/\la)e^{-i(t-s)\la^{-1}\Delta_g}\|_{L^1(M)\to L^\infty(M)}\le C\la^{\frac{n-1}2}|t-s|^{-\frac{n-1}2}.
\end{equation}
Recall that by the first part of \eqref{s.30}, we have $\|Q^{\theta_0}_{j,\ell_0,\nu}\|_{L^\infty(M)\to L^\infty(M)}\le C$. Thus if we repeat the arguments in the proof of Lemma~\ref{kerprop} above, it suffices to show 
\begin{multline}\label{j32}
|\sum_{\alpha \in \Gamma}
K_{\la,j} (\tilde x,t; \alpha(\tilde y),s))|\le C\la^{\frac{n-1}2}|t-s|^{-\tfrac{n-1}2}, \, \,  \quad \text{if } \, |t-s|\le 2^j \\ \text{with}\,\,\,j=0,1,2,\dots, \, 2^j\le  c_0\log\la.
\end{multline}
where 
\begin{equation}\label{j33}
K_{\la,j}(\tilde x,t;\tilde y,s)=
(2\pi)^{-1}\int a(2^{-j}r)\Hat m(\la,t-s;r)\,\big(Q^{\theta_0}_{j,\ell_0,\nu} \circ
\,  \cos r\sqrt{-\Delta_{\tilde g}}\big)(\tilde x,
\tilde y) \, dr,
\end{equation}
and
\begin{equation}\label{j34}
 K_{\la,j}(\tilde x,\tilde y)=0 \, \, \text{if } \, \, 
d_{\tilde g}(\tilde x,\tilde y)\ge C_12^j
\end{equation}
for a uniform constant $C_1$.  As in \eqref{j15},
 the number
of non-zero summands on the right side of \eqref{j32} is $O(\exp(C2^j))$.

If we repeat the arguments in \eqref{j28}-\eqref{j30}, it suffices to replace $ \cos r\sqrt{-\Delta_{\tilde g}}$ by the main term in the Hadamard parametrix as the higher order terms and remainder term will contribute errors of
 $O(\la^{\frac{n-1}2-\frac12})$ as long as we choose $c_0$ small enough as above. Thus the proof of Lemma~\ref{kerprop1}  would be complete if we can show that
 \begin{equation}\label{j35}
     \begin{aligned}
     (2\pi)^{-2n-1} &\sum_{\alpha \in \Gamma}\int_{-\infty}^\infty \iiint e^{i((\tilde x-\tilde z)\cdot \eta+ d_{\tilde g}(\tilde z,\alpha(\tilde y))\xi_1)}  Q^{\theta_0}_{j, \ell_0, \nu}(\tilde x,\tilde z,\eta)\\
         &\qquad\qquad\qquad\cdot
\cos (r|\xi|) \, a(2^{-j}r) \, \Hat m(\la,t-s;r) \, dr d\eta  d\tilde yd\xi \\ &
=O(\la^{\frac{n-1}2}|t-s|^{-\frac{n-1}2}) \quad
\text{when }\,\,\, |t-s|\le 2^j .
     \end{aligned}
 \end{equation}
As in \eqref{j29}, 
by \eqref{j2} and \eqref{j6} and the support properties of $a$, each term in the summand of \eqref{j35} can be simplified as 
\begin{align}\label{j36}
(2\pi)^{-2n-2}\int e^{i((\tilde x-\tilde z)\cdot \eta+ d_{\tilde g}(\tilde z,\alpha(\tilde y))}Q^{\theta_0}_{j, \ell_0, \nu}(\tilde x,\tilde z,\eta)\beta^2(|\xi|/\la) e^{i(t-s)\la^{-1}|\xi|^2} \, d\eta  d\tilde yd\xi
+O(\la^{-N}). \notag
\end{align}
Here $Q^{\theta_0}_{j, \ell_0, \nu}(\tilde x,\tilde z,\eta)$ is the symbol for the operator $Q^{\theta_0}_{j, \ell_0, \nu}$.

By a simple stationary phase argument, the last integral is $O(\la^{\frac{n-1}2}|t-s|^{-\frac{n-1}2})$. On the other hand, 
recall that as in \eqref{s.26}
$Q^\theta_{j, \ell_0, \nu}=Q^\theta_{j,\ell_0, \nu'}\circ Q^\theta_{j,\ell_0, \ell}.$  If we let $\tilde \gamma =\tilde \gamma_{j,\ell_0,\nu}\subset \Rn$
be the geodesic through the origin of the lift of the geodesic $\gamma_{j,\ell_0,\nu}\subset M$ associated
with $Q^{\theta_0}_{j,\ell_0,\nu'}$ and $\kappa^{\theta_0}_\ell$ as in the definition of $Q^{\theta_0}_{j,\ell_0, \ell}$ in \eqref{s.24},  then
one can follow the arguments in the proof of Proposition 4.2 in \cite{blair2023strichartz} to see that the last integral is $O(\la^{-m}) \, \,
\forall \, m\in {\mathbb N}$ unless
\begin{equation}\label{k.11a}
d_{\tilde g}(\tilde \gamma, \alpha(\tilde y))\le C_0,
\end{equation}
and
\begin{equation}\label{k.11b}
    d_{\tilde g}(\tilde x, \alpha(\tilde y))\in
\bigl[ 2|t-s|(\kappa^{\theta_0}_\ell-C_0\la^{-\e_0}), \,
2|t-s|(\kappa^{\theta_0}_\ell+C_0\la^{-\e_0})\big]
\end{equation}
for some fixed $C_0=C_0(M)$ and $\theta_0=\la^{-\e_0}$. By \eqref{k.7} with simple volume counting arguments, one can see
that number of $\alpha\in \Gamma$ for which \eqref{k.11a} and \eqref{k.11b} hold is $O(1)$. This finishes the proof of Lemma~\ref{kerprop1}.
\end{proof}

\newsection{Littlewood-Paley estimates}

\begin{lemma}\label{littlewood}
    Let 
    $\beta\in C_0^\infty (1/2, 2)$ with $\sum_{k=-\infty}^\infty \beta(s/2^k)=1$, and define $\beta_k(s)=\beta(s/2^k)$, $\beta_0(s)= \sum_{k\le 0} \beta(s/2^k)$. If $(M,g)$ is a complete manifold of bounded geometry, we have for  $2\le q<\infty$
    \begin{equation}\label{little1}
        \|u\|_{L^q(M)}\lesssim \|Su\|_{L^q(M)} +\|u\|_{L^2(M)},
    \end{equation}
    where $Su=\left(\sum_{k\ge0} |\beta_k(P)u|^2\right)^{\frac12}$ with $P=\sqrt{-\Delta_g}$.
\end{lemma}
Lemma~\ref{littlewood} is a generalization of  Bouclet~\cite[Theorem 1.3]{Boulp} to complete manifolds of bounded geometry, and its proof mostly follows from the same arguments there. For the sake of completeness, we provide the detailed proof below.

On compact manifolds, the above estimate holds without the additional term $\|u\|_{L^2(M)}$. However, on non-compact manifolds,  the estimate may fail without this term, since otherwise it would imply the $L^q$ boundness of the multiplier operator $\beta_k(P)$. See \cite{AGL} for a discussion in the context of hyperbolic spaces.

By Minkowski's integral inequality, \eqref{little1} implies 
 \begin{equation}\label{litte1}
        \|u\|_{L^q(M)}\lesssim \left(\sum_{k\ge0}\|\beta_k(P)u\|^2_{L^q(M)}\right)^{\frac12} +\|u\|_{L^2(M)}.
    \end{equation}
This combined with Theorem~\ref{bgst} and $L^2$ orthogonality yield
Corollary~\ref{bgsta}.

To prove Lemma~\ref{littlewood}, we shall require the following
\begin{lemma}\label{parametrix} For $k\ge 1$, we can write 
$\beta_k(P)=B_k+C_k$ with
\begin{equation}\label{bk}
    \|\sum_{k\ge 0}a_kB_ku\|_{L^q(M)}\lesssim  \|u\|_{L^q(M)},\,\,\,\text{if}\,\,\,a_k=\pm 1 \,\,\forall k \,\,\,\text{and}\,\,\,1<q<\infty.
\end{equation}
And for $q\ge 2$, 
\begin{equation}\label{ck}
    \|C_ku\|_{L^q(M)}\lesssim_N 2^{-Nk}\|u\|_{L^2(M)}.
\end{equation}
\end{lemma}
\begin{proof} We can extend $\beta\in C_0^\infty (1/2, 2)$ to an even function by letting $\beta(s)=\beta(|s|)$. For $\delta<\Inj(M)/2$,  we can fix $\rho\in C_0^\infty$ satisfying $\rho(t)= 1$, $|t|\le \delta/2$ and $\rho(t)=0 $, $|t|\ge \delta$, and define
\begin{equation}
\begin{aligned}
     \beta_k(P)&=(2\pi)^{-1}\int\hat \beta_k(t)\cos tP dt  \\
      &=(2\pi)^{-1}\int\rho(t)\hat \beta_k(t)\cos tP dt+ (2\pi)^{-1}\int (1-\rho(t))\hat \beta_k(t)\cos tP dt\\
      &=B_k+C_k.
\end{aligned}
\end{equation}
    It is not hard to check that the symbol of $C_k$ is $O((1+|\tau|+2^k)^{-N})$, thus \eqref{ck} follows from Sobolev estimates. To prove \eqref{bk}, we cover $M$ by geodesic balls of radius $\delta$. Using the finite propagation speed property of the wave propagator $\cos tP$ and locally finite property of the covering, we can reduce the calculations needed to a fixed geodesic ball. Then, \eqref{bk} follows from standard arguments using the Hadamard parametrix for $\cos tP$.
\end{proof}
\begin{proof}[Proof of Lemma~\ref{littlewood}]
    Let us denote $S_B=\left(\sum_{k\ge0} |B_ku|^2\right)^{\frac12}$. Note that by using a standard argument using Rademacher functions( see e.g.,\cite[§ 0]{SFIO2}), \eqref{bk} implies the following square function estimate
    \begin{equation}\label{Bsquare}
            \|S_Bu\|_{L^q(M)}\lesssim  \|u\|_{L^q(M)},\,\, 1<q<\infty.
    \end{equation}

    Since  $\beta_{k_1}(P)\beta_{k_2}(P)\equiv 0$ if $|k_1-k_2|\ge 2$, we have 
    \begin{equation}\label{lia}
        \begin{aligned}
          \int_M &u_1 \bar u_2 dx\lesssim \|S_B u_1\|_{L^q(M)}\|S_B u_2\|_{L^{q'}(M)} 
            +\| u_2\|_{L^{q'}(M)}
            \\
             &\times \big(\sum_{\{k_1, k_2\ge 0, |k_1-k_2|\le 1\}}   \|B_{k_1}C_{k_2} u_1\|_{L^{q}(M)}+\|C_{k_1}B_{k_2} u_1\|_{L^{q}(M)}+\|C_{k_1}C_{k_2} u_1\|_{L^{q}(M)}\big)
        \end{aligned}
    \end{equation}
  By \eqref{bk}, we have $\|B_k\|_{L^{q}(M)\to L^{q}(M) }\lesssim1$. If we combine this with \eqref{Bsquare} and \eqref{ck}, it is not hard to show that 
  \begin{equation}
              \int_M u_1 \bar u_2 dx\lesssim \| u_2\|_{L^{q'}(M)}(\|S_B u_1\|_{L^{q}(M)}
           +\|u_1\|_{L^{2}(M)}).
  \end{equation}
This implies 
  \begin{equation}\label{liaa}
              \| u\|_{L^{q}(M)}\lesssim\|S_B u\|_{L^{q}(M)}
           +\|u\|_{L^{2}(M)}.
  \end{equation}
  
  To replace $S_B$ by $S$, let us define $S_C=\left(\sum_{k\ge0} |C_ku|^2\right)^{\frac12}$, then 
  \begin{equation}\label{lib}
      \begin{aligned}
          \|S_B u\|_{L^{q}(M)}&\le \|Su\|_{L^{q}(M)}+\|S_C u\|_{L^{q}(M)}\\
          &\le \|Su\|_{L^{q}(M)}+\sum_{k\ge 0}\|C_k u\|_{L^{q}(M)} \\
          &\le \|Su\|_{L^{q}(M)}+\| u\|_{L^{2}(M)}. 
      \end{aligned}
  \end{equation}
  In the last inequality we used \eqref{ck}. This finished the proof of Lemma~\ref{littlewood}
\end{proof}

\bibliography{refs2.bib}
\bibliographystyle{abbrv}

%

\end{document}